\documentclass[pdflatex,sn-mathphys-num]{sn-jnl}

\usepackage{graphicx}%
\usepackage{multirow}%
\usepackage{amsmath,amssymb,amsfonts}%
\usepackage{amsthm}%
\usepackage{mathrsfs,mathtools}%
\usepackage[title]{appendix}%
\usepackage{xcolor}%
\usepackage{textcomp}%
\usepackage{manyfoot}%
\usepackage{booktabs}%
\usepackage{algorithm}%
\usepackage{algorithmicx}%
\usepackage{algpseudocode}%
\usepackage{listings}%
\usepackage{enumitem}
\usepackage[nameinlink,capitalise,noabbrev]{cleveref}
\usepackage{accents}

\theoremstyle{thmstyleone}%
\newtheorem{theorem}{Theorem}
\theoremstyle{thmstyletwo}%
\newtheorem{remark}{Remark}%

\theoremstyle{thmstylethree}%
\newtheorem{lemma}{Lemma}
\newtheorem{problem}{Problem}%
\newlist{assump}{enumerate}{2}
\setlist[assump,1]{label=(A\arabic*),ref=(A\arabic*),leftmargin=*, labelsep=1em}
\setlist[assump,2]{label=(A\arabic{assumpi}.\arabic*),ref=(A\arabic{assumpi}.\arabic*)}
\setlist[assump]{resume}
\crefname{assumpi}{assumption}{assumptions}
\crefname{assumpii}{assumption}{assumptions}

\newlist{bcs}{enumerate}{1}
\setlist[bcs,1]{label=(BC\arabic*),ref=(BC\arabic*)}
\crefname{bcsi}{boundary conditions}{boundary conditions}

\def\RR{\mathbb{R}}

\def\M{M}
\def\K{K}

\def\I{\mathbf{I}}
\def\G{\mathbf{G}}
\def\u{\mathbf{u}}
\def\vv{\mathbf{v}}
\def\w{\mathbf{w}}
\def\x{\mathbf{x}}
\def\y{\mathbf{y}}
\def\b{\mathbf{b}}
\def\vtheta{\vartheta}
\def\Th{\mathcal{T}_h}
\def\Vh{\mathcal{V}_h}
\def\Qh{\mathcal{Q}_h}
\def\Xh{\mathcal{X}_h}
\def\Itau{\mathcal{I}_\tau}
\def\Itauk{\mathcal{I}_{\tau_k}}
\def\bsig{\boldsymbol{\sigma}}

\def\cskw{\mathbf{c}_{\mathrm{skw}}}

\def\dt{\partial_t}
\def\dphi{\partial_\phi}
\def\dgradphi{\partial_{\nabla\phi}}
\def\dtheta{\partial_\vtheta}
\def\dtau{d^{n+1}_\tau}
\def\Du{\mathrm{D}\u}
\def\Dv{\mathrm{D}\vv}
\def\div{\operatorname{div}}

\def\ds{~\mathrm{d}s}

\def\la{\langle}
\def\ra{\rangle}

\DeclarePairedDelimiter{\norm}{\|}{\|}
\DeclarePairedDelimiter{\snorm}{|}{|}

\def\softd{{\leavevmode\setbox1=\hbox{d}%
		\hbox to 1.05\wd1{d\kern-0.4ex{\char039}\hss}}}

\begin{document}

\title[Structure-preserving approximation for melt flow]{Structure-preserving approximation for non-isothermal phase-field models in melt flow}


\author[1]{\fnm{Aaron} \sur{Brunk}}\email{abrunk@uni-mainz.de}
\equalcont{These authors contributed equally to this work.}
\author*[1]{\fnm{Dennis} \sur{Höhn}}\email{dennis.hoehn@uni-mainz.de}
\equalcont{These authors contributed equally to this work.}

\affil[1]{\orgdiv{Institute of Mathematics}, \orgname{Johannes Gutenberg University}, \\\orgaddress{\street{Staudinger Weg 9}, \city{Mainz}, \postcode{55128}, \country{Germany}}}


\abstract{
        This work presents a conforming finite-element scheme for the non-isothermal Allen-Cahn-Navier-Stokes system, incorporating periodic, closed, and thermal boundary conditions. The system comprises the incompressible Navier-Stokes equations coupled with the non-isothermal Allen-Cahn equation, which includes a non-conserved phase-field equation and a temperature equation.
        The proposed numerical scheme preserves entropy production exactly and maintains total energy conservation up to a negative numerical dissipation. Convergence tests in both space and time are conducted, and representative examples are provided to demonstrate the scheme's effectiveness.}




\maketitle

\section{Introduction}

    The melting and resolidification of materials are fundamental processes in diverse applications, ranging from the behaviour of ice in water to industrial manufacturing. These processes encompass various physical mechanisms and a complex interplay of effects, such as undercooling, surface tension, and latent heat. Conventionally, these phenomena are modelled using the classical Stefan problem, which treats heat transfer in the solid and liquid phases separately via a sharp interface~\cite{Rubinstein1971,Alexiades1992}. Nevertheless, sharp-interface models can be insufficient when topological transitions, including droplet breakup, coalescence, or dendritic branching, play a significant role.
    
    To overcome these limitations, phase-field models are widely used in the literature, replacing sharp interfaces with diffuse ones. In this framework, the non-isothermal Allen–Cahn equation and related diffuse-interface models are derived in a thermodynamically consistent manner as approximations to Stefan-type evolution~\cite{Caginalp1986,Caginalp1989}. Phase-field models are particularly prevalent in simulations of additive manufacturing and solidification microstructures~\cite{ProvatasElder2010,Steinbach2009}, with notable applications in dendritic growth, grain growth, and melt-pool morphology~\cite{KarmaRappel1998,BeckermannKarma1997,Liang}. In powder-bed fusion, melt-pool dynamics are significantly affected by Marangoni flow, vapour depression (recoil pressure), and keyhole formation, which subsequently influence spatter and porosity~\cite{King2015,Khairallah2016}.
    
    Both the Stefan problem and its phase-field analogues can be formulated to comply with the fundamental laws of thermodynamics, ensuring conservation of total energy and non-negative entropy production in the absence of external sources~\cite{GurtinFriedAnand2010}. While the Stefan model and non-isothermal phase-field equations effectively describe melting, solidification, and undercooling, they generally neglect fluid-flow effects in the liquid phase. Although this simplification is often justified~\cite{KurzFisher1998}, in additive manufacturing, melt-pool convection, surface-tension-driven flow (Marangoni effect), and recoil-pressure-driven motion can significantly affect the resulting microstructure and therefore must be incorporated~\cite{Khairallah2016,King2015}.
    
    A natural extension involves coupling a non-isothermal Allen–Cahn-type phase-field model with the incompressible Navier–Stokes equations~\cite{ANDERSON2000175}; for increased fidelity, quasi-incompressible or fully compressible formulations may be necessary. Such couplings are well established in the isothermal context for Cahn–Hilliard models and can be analogously derived for Allen–Cahn-type order parameters while preserving thermodynamic consistency~\cite{Freistuhler2016,Heida2011,Anderson1998,Alessia2014}. Achieving thermodynamic consistency requires conservation of total energy, non-negative entropy production, and conservation of linear and angular momentum. When crystalline microstructures or lattice rotations are significant at the continuum scale, the assumption of local angular-momentum symmetry may not hold, thereby motivating the adoption of Cosserat or micropolar continua~\cite{Eringen1999,Forest2009}.
    
    The non-isothermal Allen-Cahn-Navier-Stokes system  (NACNS), which we consider, is given by
    \begin{align}
        \dt\phi + \u\cdot\nabla\phi &= - \M\tfrac{\mu}{\vtheta},  \label{eq:sys1}\\
        \tfrac{\mu}{\vtheta} &= -\div(G_{\nabla\phi}(\nabla\phi)) + \partial_\phi \tfrac{f(\phi,\vtheta)}{\vtheta}, \label{eq:sys2}\\
        \dt e(\phi,\vtheta) + \u\cdot\nabla e(\phi,\vtheta) &= - \div(\K\nabla\tfrac{1}{\vtheta}) + (\eta\Du - \tilde\bsig(\nabla\phi,\vtheta)):\nabla\u + Q, \label{eq:sys3} \\
        \dt \u + (\u\cdot\nabla)\u &= \div(\eta\Du - \tilde p\mathbf{I} - \tilde\bsig(\nabla\phi,\vtheta)) + \b,\label{eq:sys4}\\
        \div(\u) &= 0. \label{eq:sys5}
    \end{align}
    
    In this system, $\phi \in [0,1]$ represents the phase-field variable, where $\phi = 0$ corresponds to the solid phase and $\phi > 0$ to the melt. The term $\tfrac{\mu}{\vtheta}$ denotes the temperature-scaled chemical potential, $\vtheta$ is the temperature, and $e(\phi,\vtheta)$ is the internal energy density. The variables $\u$ and $\tilde p$ represent the flow velocity and pressure, respectively. The temperature-dependent Korteweg stress $\tilde\bsig(\nabla\phi,\vtheta)$ accounts for Marangoni flow. The remaining functions are $\M$ (diffusive mobility), $\K$ (heat conductivity) and $\eta$ (viscosity). The term $\b$ denotes a body force, and $Q$ is the external heat source.

    Under suitable boundary conditions, in the absence of external forces and suitable assumptions on the functions, the model complies with thermodynamic laws through the total energy density $e_{\mathrm{tot}}$ and the entropy density $s$ by
    \begin{gather*}
        \int_\Omega e_{\mathrm{tot}}(\phi(t),\vtheta(t),\u(t))  = \int_\Omega e_{\mathrm{tot}}(\phi(0),\vtheta(0),\u(0)) , \\
        \int_\Omega s(\phi(t),\vtheta(t)) \geq \int_\Omega s(\phi(0),\vtheta(0)).
    \end{gather*}

    The isothermal Allen--Cahn and Cahn--Hilliard equations have been studied extensively from both analytical and numerical perspectives. Analytically, classical and modern works provide well-posedness, existence of generalised (weak or variational) solutions, regularity theory, and asymptotic behaviour; see, for instance, \cite{Elliott1996,NovickCohenPego1991,ColliGilardi1999} for Allen--Cahn-type systems. From a numerical viewpoint, a broad class of stable, convergent, and energy-dissipating time discretisations, finite element schemes, and convex-splitting or SAV-type methods has been developed; see, e.g., \cite{Eyre1998,ShenYang2010,GrunKlingbeil2016}.

    In the non-isothermal setting, the existence of generalised solutions for several variants of the non-isothermal Allen--Cahn and Cahn--Hilliard systems has been established in \cite{alt1992existence,KENMOCHI19941163,Colli_memory,Colli_Penrose,COLLI2024113461}. For diffuse-interface models coupled with fluid flow, thermodynamically consistent models and corresponding well-posedness results are studied in \cite{Eleuteri2015,Lasarzik2021,Lopes2022,WU_2017,Hossain2024}, which cover non-isothermal Navier--Stokes–phase-field systems.
    
    From a numerical perspective, energy-stable schemes for the non-isothermal systems have been proposed in the literature. Several authors have used the inverse temperature as the primary thermal variable to obtain energy-stable finite element approximations \cite{Pawlow2016,GonzalezFerreiro2014}, neglecting velocity. More recently, Brunk et al.~\cite{BrunkCMAM,BrunkPamm,Brunk2025} developed structure-preserving numerical schemes for non-isothermal phase-field–fluid systems with general Helmholtz free energies. While the latter two already incorporate the incompressible Navier--Stokes equations in a fully thermodynamically consistent manner. In a simplified thermodynamic framework, energy- and entropy-stable FEM or EIQ methods were investigated in \cite{RuedaGmez2024,Sun2020,SUN2024161}. Very recent work~\cite{Hoehn26} treats the non-isothermal Cahn--Hilliard equation using the entropy balance in place of the internal-energy equation, resulting in a fully structure-preserving discretisation.
    
    This work aims to analyse the variational structure of the NACNS system \eqref{eq:sys1}–\eqref{eq:sys5} and to develop a systematic spatial and temporal discretisation using the temperature variable rather than the inverse temperature. The principal contributions of this paper are as follows:
    \begin{itemize}
    \item Review of the inherent structures for a broad class of Helmholtz free energies
    \item A reformulation of the system using the entropy equation to highlight the underlying variational structure.
    \item A structure-preserving numerical approximation that utilises conventional discretisation techniques in both space and time, without depending on specialised structure-preserving methods.
    \end{itemize}

    The manuscript is organised as follows. \cref{sec:pre} introduces the notation and assumptions adopted throughout the paper and examines the structures and various formulations inherent in the PDE system. \cref{sec:var} presents a problem-adapted reformulation of the system in variational form while \cref{sec:dis} addresses the discretisations of the preceding weak formulation, the existence of discrete solutions, and their structure-preserving properties. The performance of the scheme is demonstrated through convergence tests in space and time, as well as application-relevant simulations in \cref{sec:num}.

\section{Preliminaries}\label{sec:pre}

    Let us briefly introduce our notation and main assumptions, and recall some basic facts.

    \subsection{Notation}
    
        The system \eqref{eq:sys1}--\eqref{eq:sys5} is investigated on a finite time interval $(0,T)$ and a spatial domain $\Omega$. By $L^p(\Omega)$, $W^{k,p}(\Omega)$, we denote the corresponding Lebesgue and Sobolev spaces with norms $\norm{\cdot}_{L^{p}}$ and $\norm{\cdot}_{W^{k,p}}$. As usual, we abbreviate $H^k(\Omega)=W^{k,2}(\Omega)$ and write $\norm{\cdot}_{H^k} = \norm{\cdot}_{W^{k,2}}$.
        %
        %
        %
        Here $\langle \cdot, \cdot\rangle$ denotes the duality product on $H^{-k}(\Omega) \times H^k(\Omega)$.
        The same symbol is also used for the scalar product on $L^2(\Omega)$, which is defined by
        \begin{align*}
        \la u, v \ra = \int_\Omega u \cdot v \qquad \forall u,v \in L^2(\Omega).    
        \end{align*}
        %
        We consider all vectors to be column vectors and gradients of vector functions correspond to the Jacobian matrix. The divergence of a matrix is applied column wise and the symmetric gradient of a vector function  $\u$ is given by $\Du:=\tfrac{1}{2}(\nabla\u+(\nabla\u)^\top)$ while the inner product between two matrices $\mathbf{A},\mathbf{B}\in\mathbb{R}^{N\times N}$ is defined by $\mathbf{A}:\mathbf{B}=\mathrm{tr}(\mathbf{A}\mathbf{B}^\top)$. We also introduce the well-known skew-symmetric formulation of $c(\u,\vv,\w)=\la(\u\cdot\nabla)\vv,\w\ra$ as
        \begin{align}\label{eq:skew}
             \cskw(\u,\vv,\w)=\tfrac{1}{2}c(\u,\vv,\w)-\tfrac{1}{2}c(\u,\w,\vv).
        \end{align}
    
    \subsection{Assumptions}\label{sec:assump}

        In this subsequent we collect the minimal assumptions for the rest of this manuscript. 

        \begin{assump}
            \item We consider $\Omega \subset \mathbb{R}^d$, $d=1,2,3$ with the following boundary conditions:
            \begin{bcs}
                \item Periodic boundary conditions: \label{bc:periodic}
                \begin{equation}
                    g(x+L_ie_i) = g(x) \text{ for } g\in\{\phi,\mu,\vtheta,\u,p\} \text{ and } i=1,\ldots, d,
                \end{equation}
                where $\Omega=[0,L_1]\times\ldots\times[0,L_d]$ is a (hyper)-cube identified with the $d$-dimensional torus $\mathcal{T}^d$.
                Moreover, other functions on $\Omega$ are also assumed to be periodic.
                \item Closed system boundary conditions: \label{bc:closed}
                \begin{align}
                    \u\vert_{\partial\Omega} &= 0, \quad G_{\nabla\phi}(\nabla\phi)\cdot\mathbf{n}\vert_{\partial\Omega} = 0, \quad  \K\nabla\tfrac{1}{\vtheta}\cdot\mathbf{n}=0,
                \end{align}
            \item Thermal boundary conditions: \label{bc:thermal}
                \begin{align}
                    \u\vert_{\partial\Omega} &= 0,  \quad G_{\nabla\phi}(\nabla\phi)\cdot\mathbf{n}\vert_{\partial\Omega} = 0,  \quad \vtheta\vert_{\partial\Omega}=\vtheta_{b},
                \end{align}
                where $\vtheta_{b}>0$ is constant on $\partial\Omega.$
            \end{bcs}
            \item The diffusive mobility $\M:=\M(\phi,\vtheta)\in\RR$ is positive and bounded, i.e. there exists positive constants $M_1,M_2>0$ such that $M_1\leq \M\leq M_2$.
            \item The heat conductivity $\K:=\K(\phi,\vtheta)\in\RR$ is positive and bounded, i.e. there exists positive constants $K_1,K_2>0$ such that $K_1\leq\K\leq K_2$.
            \item The viscosity $\eta:=\eta(\phi,\vtheta)\in\RR$ is positive and bounded, i.e. there exists positive constants $\eta_1,\eta_2>0$ such that $\eta_1\leq\eta\leq \eta_2$.
            \item The external forces/sources $\b:=\b(x,t)\in\RR^d$ and  $Q:=Q(x,t)\in\RR$ are bounded.
            \item The driving potential $f(\cdot,\cdot):\RR\times\RR_+\to \RR$ is smooth and fulfills that for every fixed $\phi$ the function $f(\phi,\cdot):\RR_+\to\RR$ is concave and goes to infinity for $\vtheta\to 0$.\label{as:pot}
            \item The gradient contribution $G(\nabla\phi):\RR^d\mapsto\RR_+\in C^2(\RR^d)$ is strictly convex with respect to $\nabla\phi$.\label{as:grad}
        \end{assump}

    \subsection{Problem inherent structures}

        To connect the state variables to the internal energy, total energy and entropy we consider an underlying thermodynamic potential given by the Helmholtz free energy, of the form
        \begin{gather}\label{eq:helmholtz}
            F(\phi,\nabla\phi,\vtheta) := \vtheta G(\nabla\phi) + f(\phi,\vtheta).
        \end{gather}
        Following the standard computations one finds that the Korteweg stress is given by
        \begin{equation}\label{eq:korteweg}
            \tilde\bsig:=\vtheta G_{\nabla\phi}(\nabla\phi)\otimes\nabla\phi.
        \end{equation}
        With that choice, we can calculate the internal energy and total energy as well as the entropy density via their relations to the Helmholtz free energy by
        \begin{align}
            s(\phi,\nabla\phi,\vtheta)&:=-\dtheta F(\phi,\nabla\phi,\vtheta) = -G(\nabla\phi) - \partial_\vtheta f(\phi,\vtheta),\label{eq:entropy}\\
            e(\phi,\vtheta)&:=F(\phi,\nabla\phi,\vtheta)+\vtheta s(\phi,\nabla\phi,\vtheta)= f(\phi,\vtheta) - \vtheta \partial_\vtheta f(\phi,\vtheta),\label{eq:energy}\\
            e_{\mathrm{tot}}(\phi,\vtheta,\u) &:= e(\phi,\vtheta) + \tfrac{1}{2}\snorm{\u}^2. \label{eq:totenergy}
        \end{align}

        \subsubsection{Reformulation of the system}\label{sec:refo}
    
            Since these relations, on the formal level, one can choose between several formulations of the temperature contribution which allows to switch between the internal energy $e$ and the entropy $s$. For simplicity we set $s\equiv s(\phi,\nabla\phi,\vtheta)$.
            By rearranging the identity $e=F+\vtheta s$, cf. \eqref{eq:energy}, in terms of the entropy, we get $s=(e-F)/\vtheta$ and compute
            \begin{align*}
                \dt s & = \tfrac{1}{\vtheta}\dt e(\phi,\vtheta) - \tfrac{1}{\vtheta}\dt F(\phi,\nabla\phi,\vtheta) - \tfrac{e(\phi,\vtheta)-F(\phi,\nabla\phi,\vtheta)}{\vtheta^2}\dt \vtheta  \notag\\
                & = \tfrac{1}{\vtheta}\dt e(\phi,\vtheta) - \tfrac{1}{\vtheta}\dphi F(\phi,\nabla\phi,\vtheta)\dt\phi - \tfrac{1}{\vtheta}\dgradphi F(\phi,\nabla\phi,\vtheta)\cdot\nabla\dt\phi\notag\\
                &\quad- \tfrac{1}{\vtheta}\dtheta F(\phi,\nabla\phi,\vtheta)\dt\vtheta - \tfrac{s(\phi,\vtheta)}{\vtheta}\dt \vtheta. \notag \\
                \intertext{Now using $s=-\dtheta F(\phi,\nabla\phi,\vtheta)$, cf. \eqref{eq:entropy}, and substituting in \eqref{eq:helmholtz} shows}
                \dt s & = \tfrac{1}{\vtheta}\dt e(\phi,\vtheta) - G_{\nabla\phi}(\nabla\phi)\cdot\nabla\dt\phi - \tfrac{1}{\vtheta}\dphi f(\phi,\vtheta)\dt\phi  \notag\\ 
                & = \tfrac{1}{\vtheta}\dt e(\phi,\vtheta) - \div\left(G_{\nabla\phi}(\nabla\phi)\dt\phi\right) + \div\left(G_{\nabla\phi}(\nabla\phi)\right)\dt\phi - \tfrac{1}{\vtheta}\dphi f(\phi,\vtheta)\dt\phi  \notag\\
                & = \tfrac{1}{\vtheta}\dt e(\phi,\vtheta) - \div(G_{\nabla\phi}(\nabla\phi)\dt\phi) - \tfrac{\mu}{\vtheta}\dt\phi, 
                \intertext{where in the last step we used the definition of $\tfrac{\mu}{\vtheta}$, c.f. \eqref{eq:sys2}. Similarly we compute}
                \u\cdot\nabla s &= \u\cdot \left( \tfrac{1}{\vtheta}\nabla e(\phi,\vtheta)  - \nabla G(\nabla\phi) - \tfrac{1}{\vtheta}\dphi f(\phi,\vtheta)\nabla\phi \right)\notag \\
                &= \u\cdot \left( \tfrac{1}{\vtheta}\nabla e(\phi,\vtheta) - (G_{\nabla\phi}(\nabla\phi)\cdot\nabla)\nabla\phi \pm \div(G_{\nabla\phi}(\nabla\phi))\nabla\phi \right.\notag\\
                &- \left.\tfrac{1}{\vtheta}\dphi f(\phi,\vtheta)\nabla\phi \right).\notag \\
                \intertext{where we use \eqref{eq:sys2} together with the identity $\div(a\otimes b)=\div(a)b + (a\cdot\nabla)b$ to get}
                \u\cdot\nabla s & = \u\cdot \left( \tfrac{1}{\vtheta}\nabla e(\phi,\vtheta) - \tfrac{\mu}{\vtheta}\nabla \phi -\div(G_{\nabla\phi}(\nabla\phi)\otimes\nabla\phi)\right)\notag\\
                & = \tfrac{1}{\vtheta}\u\cdot\nabla e(\phi,\vtheta) -  \tfrac{\mu}{\vtheta}\u\cdot\nabla \phi - \div(G_{\nabla\phi}(\nabla\phi)\otimes\nabla\phi\cdot\u)\\
                &\quad+ G_{\nabla\phi}(\nabla\phi)\otimes\nabla\phi:\nabla\u. 
            \end{align*}
            By adding these two equations together and inserting the internal energy equation cf. \eqref{eq:sys3}, we then get
            \begin{align}\label{eq:sys3alternative}
                \dt s + \u\cdot\nabla s &= - \tfrac{\mu}{\vtheta}\dt\phi - \tfrac{\mu}{\vtheta}\u\cdot\nabla\phi -\div(\K\nabla\tfrac{1}{\vtheta})\tfrac{1}{\vtheta} + \tfrac{\eta}{\vtheta}\snorm{\Du}^2\notag \\
                &\quad - \div(G_{\nabla\phi}(\nabla\phi)\dt\phi + G_{\nabla\phi}(\nabla\phi)\otimes\nabla\phi\cdot\u) + \tfrac{Q}{\vtheta}.\notag\\
                \intertext{Insertion of the phase-field equation cf. \eqref{eq:sys1} and defining $\bsig:=G_{\nabla\phi}(\nabla\phi)\otimes\nabla\phi$ then yields the first form of the entropy equation for the NACNS system}
                \dt s + \u\cdot\nabla s &= \M(\tfrac{\mu}{\vtheta})^2 - \div(\K\nabla\tfrac{1}{\vtheta})\tfrac{1}{\vtheta} + \tfrac{\eta}{\vtheta}\snorm{\Du}^2 \notag \\
                &\quad - \div(G_{\nabla\phi}(\nabla\phi)\dt\phi + \bsig\cdot\u) + \tfrac{Q}{\vtheta}.
            \end{align}
            Since from an analytical standpoint, equation \eqref{eq:sys3} and \eqref{eq:sys3alternative} are equivalent, we will from now on use the entropy equation instead of the energy equation \eqref{eq:sys3} when referring to the non-isothermal Allen-Cahn-Navier-Stokes system (NACNS). This will allow us, later on, to better discretize the system and avoid some difficulties which would arise by using the internal energy equation instead. By using the reverse product rule on this equation, one can also obtain the second form
            \begin{align}\label{eq:sys3div}
                \dt s + \u\cdot\nabla s &= \M\left(\tfrac{\mu}{\vtheta}\right)^2 + \K\snorm{\nabla\tfrac{1}{\vtheta}}^2 + \tfrac{\eta}{\vtheta}\snorm{\Du}^2 + \tfrac{Q}{\vtheta} \notag \\
                &\quad - \div\left( \tfrac{\K}{\vtheta}\nabla\tfrac{1}{\vtheta} + G_{\nabla\phi}(\nabla\phi)\dt\phi + \bsig\cdot\u\right), 
            \end{align}
            which will be useful for showing the following Lemma.
            
        \subsubsection{Thermodynamic properties}\label{sec:srtucture}
            \begin{lemma}\label{lem:sysstructure}
                Regular solutions $(\phi,\mu,\vtheta,\u,p)$ to the non-isothermal Allen-Cahn-Navier-Stokes system \eqref{eq:sys1},\eqref{eq:sys2},\eqref{eq:sys3alternative},\eqref{eq:sys4},\eqref{eq:sys5} satisfy the following balance laws
                \begin{itemize}
                    \item In the case of periodic \cref{bc:periodic} and closed system \cref{bc:closed} we have total energy balance and entropy production balance , i.e.
                    \begin{align*}
                    \la \dt e_{\mathrm{tot}},1\ra & = \la Q,1\ra + \la\b,\u\ra,\qquad \la \dt s,1 \ra  = \mathcal{D}\left(\tfrac{\mu}{\vtheta},\tfrac{1}{\vtheta},\u\right) + \la Q,\tfrac{1}{\vtheta}\ra.
                    \end{align*} 
                    \item In the case of thermal \cref{bc:thermal} we have exergy dissipation balance, i.e.
                    \begin{align*}
                    \la \dt (e_{\mathrm{tot}} - \vtheta_{b}s),1\ra & = - \vtheta_{b}\mathcal{D}\left(\tfrac{\mu}{\vtheta},\tfrac{1}{\vtheta},\u\right) + \la Q,1-\tfrac{\vtheta_{b}}{\vtheta}\ra + \la\b,\u\ra.
                    \end{align*} 
                \end{itemize}
                with the entropy production
                \begin{equation*}
                    \mathcal{D}\left(\tfrac{\mu}{\vtheta},\tfrac{1}{\vtheta},\u\right) := \int_\Omega\M\left(\tfrac{\mu}{\vtheta}\right)^2 + \K\snorm{\nabla\tfrac{1}{\vtheta}}^2 + \tfrac{\eta}{\vtheta}\snorm{\Du}^2 \geq 0.
                \end{equation*}
                Hence, in absence of external forces and sources, i.e. $Q=\b=0$, this reduces to:
                \begin{align*}
                    \la \dt e_{\mathrm{tot}}, 1 \ra &= 0, &&\text{(total energy conservation)}\qquad\qquad\qquad \\
                    \la \dt s, 1 \ra &\geq 0, &&\text{(entropy production)} \\
                    \qquad\qquad\qquad\la \dt(e_{\mathrm{tot}} - \vtheta_{b}s),1\ra & \leq 0. &&\text{(exergy dissipation)}
                \end{align*}
            \end{lemma}
            \begin{proof}
                For \labelcref{bc:periodic,bc:closed}, entropy production follows directly by first going from \eqref{eq:sys3alternative} to \eqref{eq:sys3div}, then integrating over the domain and lastly using integration by parts for the divergence term. For the total Energy we use equation \eqref{eq:sys3}, which is equivalent to \eqref{eq:sys3alternative}, in addition to \eqref{eq:sys4} and \eqref{eq:sys5} and calculate
                \begin{align}
                    \la\dt e_{\mathrm{tot}},1\ra &= \la\dt e+\dt\tfrac{1}{2}\snorm{\u}^2,1\ra =\la\dt e,1\ra + \la\dt\u,\u\ra\notag\\
                    &= -\la\u,\nabla e\ra - \la\div(\K\nabla\tfrac{1}{\vtheta}),1\ra + \la\eta\Du - \tilde\bsig,\nabla\u\ra + \la Q,1\ra\notag\\
                    &\quad- \la(\u\cdot\nabla)\u,\u\ra + \la\div(\eta\Du - \tilde p\mathbf{I} - \tilde\bsig),\u\ra + \la\b,u\ra. \label{eq:dtetot}
                \end{align}
                Here most terms cancel out using integration by parts and either of the two boundary conditions, together with \eqref{eq:sys5}, leaving only
                \begin{equation*}
                    \la\dt e_\mathrm{tot},1\ra = \la \dt e + \dt\tfrac{1}{2}\snorm{\u}^2,1 \ra = \la Q,1\ra + \la\b,u\ra.
                \end{equation*}
                In the case of \labelcref{bc:thermal} we simply subtract the integrated \eqref{eq:sys3alternative} multiplied by $\vtheta_{b}$ from \eqref{eq:dtetot} to get
                \begin{align*}
                    \la \dt(e_{\mathrm{tot}} - \vtheta_{b}s),1\ra &= -\la\nabla e-\vtheta_{b}\nabla s,\u\ra - \la\div(\K\nabla\tfrac{1}{\vtheta}),1-\tfrac{\vtheta_{b}}{\vtheta}\ra\\
                    &\quad + \la\eta\Du - \tilde\bsig,\nabla\u\ra + \la\div\left(G_{\nabla\phi}(\nabla\phi)\dt\phi + \bsig\cdot\u\right),\vtheta_{b}\ra\\
                    &\quad- \la(\u\cdot\nabla)\u,\u\ra + \la\div(\eta\Du - \tilde p\mathbf{I} - \tilde\bsig),\u\ra\\
                    &\quad- \la\M\left(\tfrac{\mu}{\vtheta}\right)^2 + \tfrac{\eta}{\vtheta}\snorm{\Du}^2,\vtheta_{b}\ra + \la Q,1-\tfrac{\vtheta_{b}}{\vtheta}\ra + \la\b,u\ra.
                \end{align*}
                Integration by parts and usage of the boundary conditions yields
                \begin{align*}
                    \la \dt e_{\mathrm{tot}} - \vtheta_{b}s,1\ra &= - \la\M\left(\tfrac{\mu}{\vtheta}\right)^2 + \K\left(\nabla\tfrac{1}{\vtheta}\right)^2 + \tfrac{\eta}{\vtheta}\snorm{\Du}^2,\vtheta_{b}\ra\\
                    &\quad - \la \K\nabla\tfrac{1}{\vtheta}\cdot\mathbf{n},1-\tfrac{\vtheta_{b}}{\vtheta}\ra_{\partial\Omega} + \la Q,1-\tfrac{\vtheta_{b}}{\vtheta}\ra + \la\b,u\ra ,
                \end{align*}
                hence the result follows as $1-\tfrac{\vtheta_{b}}{\vtheta}=0$  on the boundary.
            \end{proof}
            \begin{remark}
                Comparable properties arise when considering a full Onsager matrix, which permits cross-diffusion between the phase-field and temperature, as well as a symmetric, matrix-valued heat capacity $\K$. Cross-diffusion between the Cahn-Hilliard equation and temperature has been examined in \cite{Hoehn26}, where the cross-coupling may be matrix-valued. In the present context, the cross-coupling must be vector-valued, specifically by introducing $\mathbf{C}\cdot\nabla\tfrac{1}{\vtheta}$ in the Allen-Cahn equation and $-\div(\tfrac{\mu}{\vtheta}\mathbf{C})$ in the internal energy equation for $\mathbf{C}\in\mathbb{R}^d$.
            \end{remark}
            
            Linear momentum conservation can be verified directly under appropriate boundary conditions, whereas angular momentum conservation is considerably more complex. The balance of local angular momentum is reflected in the requirement that the stress tensor be symmetric. In contrast, global angular momentum conservation can be established under less restrictive conditions. 

\section{Variational formulation}\label{sec:var}
  
    The development of a numerical method begins with the introduction of the weak formulation of the problem, which is directly applicable to finite-element discretisation. Furthermore, energy conservation and entropy production are shown using appropriate test functions.
    \begin{lemma}\label{lem:varstruc}
        A smooth solution $(\phi,\mu,\vtheta,\u,p)$ to the NACNS system fulfills the variational formulation
        \begin{align}
            \la\dt\phi,\psi\ra + \la \u,\psi\nabla\phi \ra&= -\la\tfrac{\M}{\vtheta}\mu,\psi\ra,\label{eq:var1}\\
            \la\mu,\xi\ra&=\la G_{\nabla\phi}(\nabla\phi),\vtheta\nabla\xi+\xi\nabla \vtheta\ra+\la\dphi f,\xi\ra,\label{eq:var2}\\
            \la\dt s,\omega\ra - \la \u,s\nabla\omega \ra &= \la \tfrac{\M}{\vtheta}\mu,\tfrac{\omega}{\vtheta}\mu\ra + \la\tfrac{\K}{\vtheta^2}\nabla \vtheta,\tfrac{\omega}{\vtheta^2}\nabla \vtheta-\tfrac{1}{\vtheta}\nabla\omega\ra + \la \tfrac{\eta}{\vtheta}\snorm{\Du}^2,\omega \ra\notag \\
            &\quad +\la G_{\nabla\phi}(\nabla\phi)\dt\phi + \bsig\cdot\u,\nabla\omega\ra + \la Q,\tfrac{\omega}{\vtheta}\ra,\label{eq:var3}\\
            \la \dt\u,\vv \ra + \cskw(\u,\u,\vv) &= -\la \eta\Du,\mathrm{D}\vv \ra + \la p,\div(\vv) \ra + \la \b,\vv\ra \\
            &\quad-\la (\bsig+s\I)\nabla\vtheta-\mu\nabla\phi,\vv\ra, \label{eq:var4}\\
            0 &= \la \div(\u),q\ra,\label{eq:var5}
        \end{align}
        for smooth test functions $\psi,\xi,\omega,\vv,q$ with $p=\tilde p+F$. 
        \begin{itemize}
            \item Solutions to \eqref{eq:var1}--\eqref{eq:var5} using the \cref{bc:periodic,bc:closed} fulfil conservation of total energy as well as entropy production as in \cref{lem:sysstructure} by using only the test functions $\psi=\mu,~\xi=\dt\phi,~\omega\in\{1,\vtheta\},~\vv=\u$ and $q=p$. In the case of \labelcref{bc:closed} $\vv$ is compactly supported in space.
            \item Solutions to \eqref{eq:var1}--\eqref{eq:var5} using the \cref{bc:thermal} fulfill exergy dissipation as in \cref{lem:sysstructure} by using only the test functions $\psi=\mu,~\xi=\dt\phi,~\omega=\vtheta-\vtheta_{b},~\vv=\u$ and $q=p$. In this case, $\vv,\omega$ are compactly supported in space.
        \end{itemize}
    \end{lemma}
    \begin{proof}
    \textbf{Variational form:}
        Equation \eqref{eq:var1} and \eqref{eq:var5} follow directly from \eqref{eq:sys1} and \eqref{eq:sys5} by multiplying with $\psi$ and $q$ respectively and integrating over the domain $\Omega$. For the other equations we also multiply \eqref{eq:sys2}, \eqref{eq:sys3alternative} and \eqref{eq:sys4} by the functions $\vtheta\xi, \omega$ and $\vv$ respectively, and integrate over the domain $\Omega$ to obtain
        \begin{align*}
            \la\mu,\xi\ra&=-\la\div(G_{\nabla\phi}(\nabla\phi)),\vtheta\xi\ra+\la\dphi f,\xi\ra,\\
            \la\dt s,\omega\ra+\la\u,\omega\nabla s\ra&=\la\M\left(\tfrac{\mu}{\vtheta}\right)^2,\omega\ra - \la\div(\K\nabla\tfrac{1}{\vtheta}),\tfrac{\omega}{\vtheta}\ra + \la \tfrac{\eta}{\vtheta}\snorm{\Du}^2,\omega\ra \notag\\
            &\quad- \la\div\left(G_{\nabla\phi}(\nabla\phi)\dt\phi + \bsig\cdot\u\right),\omega\ra + \la Q,\tfrac{\omega}{\vtheta}\ra,\\
            \la\dt\u,\vv\ra+\la(\u\cdot\nabla)\u,\vv\ra &= \la\div(\eta\Du - \tilde p\mathbf{I} - \tilde\bsig),\vv\ra + \la\b,\vv\ra.
        \end{align*}
        For all the divergence terms except the Korteweg stress, we use partial integration together with the product rule to derive
        \begin{align*}
            \la\mu,\xi\ra&=\la G_{\nabla\phi}(\nabla\phi),\vtheta\nabla\xi+\xi\nabla\vtheta\ra+\la\dphi f,\xi\ra,\\
            \la\dt s,\omega\ra+\la\u,\omega\nabla s\ra &= \la\tfrac{\M}{\vtheta}\mu,\tfrac{\omega}{\vtheta}\mu\ra + \la\tfrac{\K}{\vtheta^2}\nabla \vtheta,\tfrac{\omega}{\vtheta^2}\nabla\vtheta-\tfrac{1}{\vtheta}\nabla\omega\ra + \la \tfrac{\eta}{\vtheta}\snorm{\Du}^2,\omega\ra\notag\\
            &\quad + \la G_{\nabla\phi}(\nabla\phi)\dt\phi + \bsig\cdot\u,\nabla\omega\ra + \la Q,\tfrac{\omega}{\vtheta}\ra,\\
            \la\dt\u,\vv\ra+\la(\u\cdot\nabla)\u,\vv\ra &= - \la\eta\Du,\Dv\ra + \la \tilde p ,\div(v)\ra 
            -\la\div(\tilde\bsig),\vv\ra + \la\b,\vv\ra.
        \end{align*}
        where the boundary terms vanish in all cases. While the equation for $\mu$ is already in weak formulation, for the entropy we use \eqref{eq:sys5} and calculate
        \begin{align*}
            \la\u,\omega\nabla s\ra&=\la\u,\nabla(\omega s)\ra-\la\u,s\nabla\omega\ra=-\la\div(\u),\omega s\ra-\la\u,s\nabla\omega\ra = -\la\u,s\nabla\omega\ra,
        \end{align*}
        to get \eqref{eq:var3}. Lastly \eqref{eq:var4} can be obtained by using the skew-symmetric formulation \eqref{eq:skew} for
        \begin{align*}
             \la(\u\cdot\nabla)\u,\vv\ra &= \tfrac{1}{2}\la(\u\cdot\nabla)\u,\vv\ra - \tfrac{1}{2}\la(\u\cdot\nabla)\vv,\u\ra - \la \div(\u),\u\cdot\vv \ra\\
             &=\tfrac{1}{2}c(\u,\u,\vv)-\tfrac{1}{2}c(\u,\vv,\u)= \cskw(\u,\u,\vv),
        \end{align*}
        as well as calculating
        \begin{align*}
            \la \div(\tilde \bsig),\vv \ra &= \la \div(\vtheta\bsig),\vv \ra = \la\vtheta\div(\bsig),\vv\ra + \la\bsig\nabla\vtheta,\vv\ra\\
            &= \la\vtheta\div(G_{\nabla\phi}(\nabla\phi))\nabla\phi,\vv\ra + \la\vtheta(G_{\nabla\phi}(\nabla\phi)\cdot\nabla)\nabla\phi,\vv\ra + \la\bsig\nabla\vtheta,\vv\ra.\\
            \intertext{Here we substitute in \eqref{eq:sys2} to get}
            \la \div(\tilde \bsig),\vv \ra &= -\la\mu\nabla\phi,\vv\ra + \la\dphi f\nabla\phi,\vv\ra + \la\vtheta\nabla G(\nabla\phi),\vv\ra + \la\bsig\nabla\vtheta,\vv\ra,\\
            \intertext{proceed to use the inverse product rule to derive}
            \la \div(\tilde \bsig),\vv \ra &= - \la\mu\nabla\phi,\vv\ra + \la\nabla f,\vv\ra - \la\dtheta f\nabla\vtheta,\vv\ra\\
            &\quad+ \la\nabla(\vtheta G(\nabla\phi)),\vv\ra - \la G(\nabla\phi)\nabla\vtheta,\vv\ra + \la\bsig\nabla\vtheta,\vv\ra\\
            &= \la\bsig\nabla\vtheta,\vv\ra - \la(G(\nabla\phi)+\dtheta f)\nabla\vtheta,\vv\ra\\
            &\quad - \la\mu\nabla\phi,\vv\ra + \la\nabla(\vtheta G(\nabla\phi)+f),\vv\ra\\
            \intertext{and finally use \eqref{eq:entropy} and \eqref{eq:helmholtz} followed by partial integration to reveal}
            \la \div(\tilde \bsig),\vv \ra &= \la\bsig\nabla\vtheta,\vv\ra + \la s\nabla\vtheta,\vv\ra - \la\mu\nabla\phi,\vv\ra + \la\nabla F,\vv\ra\\
            &= \la (\bsig+s\I)\nabla\vtheta,\vv \ra - \la \mu\nabla\phi,\vv \ra - \la F, \div(\vv)\ra,        \end{align*}
        again valid for all boundary conditions.
        Then, substituting these two terms into the third equation from above completes the proof of the first part of the lemma.
        
        \textbf{Simple test functions:}
        To show the second part, we assume $(\phi,\mu,\vtheta,\u,p)$ to be a solution of the variational form and start for \cref{bc:periodic,bc:closed} by testing \eqref{eq:var3} with $\omega=1$ to see:
        \begin{align*}
            \la\dt s,1\ra &= \la \u,s\nabla1 \ra + \la \tfrac{\M}{\vtheta}\mu,\tfrac{1}{\vtheta}\mu\ra + \la\tfrac{\K}{\vtheta^2}\nabla \vtheta,\tfrac{1}{\vtheta^2}\nabla \vtheta-\tfrac{1}{\vtheta}\nabla1\ra + \la \tfrac{\eta}{\vtheta}\snorm{\Du}^2,1 \ra\notag \\
            &\quad +\la G_{\nabla\phi}(\nabla\phi)\dt\phi + \bsig\cdot\u,\nabla1\ra + \la Q,\tfrac{1}{\vtheta}\ra\\
            &= \la\M\tfrac{\mu}{\vtheta},\tfrac{\mu}{\vtheta}\ra + \la\K\nabla\tfrac{1}{\vtheta},\nabla\tfrac{1}{\vtheta}\ra + \la\tfrac{\eta}{\vtheta}\Du,\Du\ra + \la Q,\tfrac{1}{\vtheta}\ra\notag\\
            &= \mathcal{D}\left(\tfrac{\mu}{\vtheta},\tfrac{1}{\vtheta},\u\right) + \la Q,\tfrac{1}{\vtheta}\ra.
        \end{align*}
        For the total energy we use the identities $e=F+\vtheta s$, $s=-\dtheta F$, cf. \eqref{eq:entropy},\eqref{eq:energy}, and \eqref{eq:helmholtz} as well as the reverse product rule to compute 
        \begin{align} \label{eq:vardtetot}
            \la \dt e_{\textrm{tot}},1\ra  &=\la\dphi F,\dt\phi\ra + \la\dtheta F,\dt \vtheta\ra + \la\dt\vtheta, s\ra  + \la\dt s,\vtheta\ra + \la\dt\u,\u\ra \notag\\
            &= \la G_{\nabla\phi}(\nabla\phi),\vtheta \nabla\dt\phi\ra + \la\dphi f, \dt \phi\ra + \la\dt s,\vtheta\ra + \la\dt\u,\u\ra.
        \end{align}
        For \labelcref{bc:periodic,bc:closed}, substituting in equation \eqref{eq:var2}, as well as equations \eqref{eq:var3} and \eqref{eq:var4}, tested with $\xi=\dt\phi$, $\omega=\vtheta$ and $\vv=\u$ respectively, shows
        \begin{align*}
            \la \dt e_{\textrm{tot}},1\ra&= \la\mu,\dt\phi\ra +\la \tfrac{\M}{\vtheta}\mu,\mu\ra + \la Q,1\ra + \la p,\div(\u)\ra + \la\b,\u\ra + \la\mu\nabla\phi,\u\ra, \notag
            \intertext{now substituting in equation \eqref{eq:var1} and \eqref{eq:var5}, tested with $\psi=\mu$ and $q=p$ leads to}
            \la \dt e_{\textrm{tot}},1\ra&= \la Q,1\ra + \la\b,\u\ra.
        \end{align*}
        In the case of \labelcref{bc:thermal} we consider \eqref{eq:vardtetot} and subtract $\la \vtheta_{b},\dt s \ra$ from both sides to obtain:
        \begin{align*}
            \la \dt \left(e_\mathrm{tot} - \vtheta_{b} s\right),1\ra  &= \la G_{\nabla\phi}(\nabla\phi),\vtheta \nabla\dt\phi\ra + \la\dphi f, \dt \phi\ra + \la\dt s,\vtheta-\vtheta_{b}\ra + \la\dt\u,\u\ra.
        \end{align*}
        By construction, $\vtheta-\vtheta_{b}$ is zero on the boundary, hence a valid test function. The rest follows from insertion of the mentioned test function and similar cancellations.
    \end{proof}
    
\section{Discrete scheme}\label{sec:dis}

    This section presents a fully discrete scheme for the NACNS system based on the variational form \eqref{eq:var1}--\eqref{eq:var5}. The scheme is shown to satisfy the entropy production and total energy dissipation conditions, and the existence of discrete solutions with positive temperature is established.
    
    As a preparatory step, the relevant notation and assumptions for the discretisation strategy are introduced.
    
    \subsection*{Time discretisation}

        The time interval $[0,T]$ is partitioned into uniformly divided sub-intervals with step size $\tau>0$, and $\Itau:=\{0=t^0,t^1=\tau,\ldots, t^{n_T}=T\}$ is introduced, where $n_T=\tfrac{T}{\tau}$ denotes the total number of time steps. The spaces $\Pi^1_c(\Itau;X)$ and $\Pi^0(\Itau;X)$ represent continuous piecewise linear and piecewise constant functions on $\Itau$, respectively, with values in a suitable Hilbert space $X$. By $g^{n+1},g^n,g^{n+1/2},g^*$ we denote the evaluation of $g\in \Pi^1_c(\Itau;X)$ or $\Pi^0(\Itau;X)$ at $t=\{t^{n+1},t^n,t^{n+1/2},t^*\}$, where $t^{n+1/2}=\tfrac{t^n+t^{n+1}}{2}$ and $t^*\in[t^n,t^{n+1}]$ is some intermediate time. Further, the time difference and the discrete time derivative are defined via
        \begin{equation*}
            d^{n+1}g = g^{n+1} - g^n, \qquad d^{n+1}_\tau g = \frac{g^{n+1}-g^n}{\tau}.
        \end{equation*}
        All notations are also extended to nonlinear functions $Y$ of $g$, i.e. $d^{n+1}Y=Y^{n+1}-Y^n=Y(g^{n+1}) - Y(g^n)$ and $Y^*=Y(g^*)$ and for the external forces $Q,\b$ we set $Q^*:=Q(x,t^*),\b^*:=\b(x,t^*)$.

    \subsection*{Space discretisation}

        A geometrically conforming partition $\Th$ of the domain $\Omega$ into simplices of maximal diameter $h$ is considered, where the spaces of continuous, piecewise linear and continuous, piecewise quadratic functions over $\Th$, as well as the subspaces containing only mean-free functions of the first type, are denoted as follows.
        \begin{align*}
            \Vh &:= \{v \in H^1(\Omega)\cap C^0(\bar\Omega) : v|_K \in P_1(K),~\forall K \in \Th\},\\
            \Xh &:= \{v \in H^1(\Omega)^d\cap C^0(\bar\Omega)^d : v|_K \in P_2(K)^d,~ \forall K \in \Th\}, \\
            \Qh &:= \{v \in \Vh : \la v, 1\ra=0\}.
        \end{align*}
        To accommodate for the different boundary conditions we then use the following approximation spaces.
        \begin{enumerate}
            \item In case of periodic \cref{bc:periodic} we choose 
            \begin{align*}
                \Vh^0 = \Vh^b = \Vh,\quad\Xh^b &:= \Xh
            \end{align*}
            and assume that $\Th$ can be extended periodically to periodic extensions of $\Omega$.
            \item In case of closed \cref{bc:closed} we choose
            \begin{align*}
                \Vh^0 = \Vh^b = \Vh,\quad\Xh^{b} = \Xh\cap \{v\vert_{\partial\Omega} = 0\}.
            \end{align*}
            \item In case of thermal \cref{bc:thermal} we choose
            \begin{align*}
                 \Vh^0 = \Vh \cap \{v\vert_{\partial\Omega} = 0\}, \quad \Vh^b = \Vh \cap \{v\vert_{\partial\Omega} = \vtheta_{b}\},\quad\Xh^{b} = \Xh\cap \{v\vert_{\partial\Omega} = 0\}.
            \end{align*}
        \end{enumerate}
        The notation $Y_h^*$ extends the notation from above to evaluations with discrete functions, i.e. $Y_h^*=Y(g_h^*)$. Lastly the time-average of $\dphi f$ is defined as
        \begin{equation}\label{eq:fphi}
            \partial_\phi f(\phi_h^{n+1},\phi_h^n,\vtheta_h^{n+1}):= \frac{1}{\tau}\int_{t^n}^{t^{n+1}} \partial_\phi f\left(\phi_h^{n}+\frac{s-t^{n}}{\tau}(\phi_h^{n+1}-\phi_h^n),\vtheta_h^{n+1}\right)\ds,
        \end{equation}
        for $\phi_h,\vtheta_h\in\Pi^1_c(\Itau;\Vh)$. With these preliminaries, the following fully discrete system is proposed.\clearpage
        \begin{problem}\label{prob:nacns}
            Given the initial data $(\phi_{h,0},\u_{h,0},\vtheta_{h,0})\in \Vh\times\Xh^b\times\Vh^b$ for the NACNS system, find $(\phi_h,\u_h,\vtheta_h)\in \Pi^1_c(\Itau;\Vh\times\Xh^b\times\Vh^b)$ and $\mu_{h},p_h\in \Pi^0(\Itau;\Vh\times\Qh)$ such that
            \begin{align}
                \la\dtau\phi_h,\psi_h\ra&=-\la\u_h^{n+1},\psi_h\nabla\phi_h^*\ra-\la\tfrac{\M_h^*}{\vtheta_h^*}\mu_h^{n+1},\psi_h\ra,\label{eq:dis1}\\
                \la\mu_h^{n+1},\xi_h\ra&=\la G_{\nabla\phi}(\nabla\phi_h^{n+1}),\vtheta_h^{n+1}\nabla\xi_h\ra+\la G_{\nabla\phi}(\nabla\phi_h^*),\xi_h\nabla \vtheta_h^{n+1}\ra\notag\\
                &\quad+ \la \partial_\phi f(\phi_h^{n+1},\phi_h^n,\vtheta_h^{n+1}),\xi_h\ra,\label{eq:dis2}\\
                \la\dtau s_h,\omega_h\ra&=\la\tfrac{\M_h^*}{\vtheta_h^*}\mu_h^{n+1},\tfrac{\omega_h}{\vtheta_h^{n+1}}\mu_h^{n+1}\ra+\la\tfrac{\eta_h^*}{\vtheta_h^{n+1}}\snorm{\Du_h^{n+1}}^2,\omega_h\ra\notag\\
                &\quad+\la\tfrac{\K_h^*}{(\vtheta_h^*)^2}\nabla \vtheta_h^{n+1},\tfrac{\omega_h}{\vtheta_h^{n+1}\cdot \vtheta_h^*}\nabla \vtheta_h^{n+1}-\tfrac{1}{\vtheta_h^*}\nabla\omega_h\ra + \la Q^*,\tfrac{\omega_h}{\vtheta_h^{n+1}}\ra \notag\\
                &\quad+\la G_{\nabla\phi}(\nabla\phi_h^*)\dtau\phi_h + (\bsig_h^*+s_h^*\mathbf{I})\u_h^{n+1},\nabla\omega_h\ra,\label{eq:dis3}\\
                \la\dtau\u,\vv_h\ra&=-~\cskw(\u_h^*,\u_h^{n+1},\vv_h)-\la\eta_h^*\Du_h^{n+1},\Dv_h\ra+\la p_{h}^{n+1},\div(\vv_h)\ra\notag\\
                &\quad-\la(\bsig_h^*+s_h^*\mathbf{I})\nabla\vtheta_h^{n+1}-\mu_h^{n+1}\nabla\phi_h^*,\vv_h\ra + \la\b^*,\vv_h\ra,\label{eq:dis4}\\
                0&=\la\div(\u_h^{n+1}),q_h\ra,\label{eq:dis5}
            \end{align}
            holds for $(\psi_h,\xi_h,w_h,\vv_h,q_h)\in\Vh\times\Vh\times\Vh^0\times \Xh^b\times\Qh$ and all $0\leq n\leq n_T-1$.
        \end{problem}
        \begin{theorem}\label{thm:discstruc}
             A discrete solution $(\phi_h,\mu_{h},\vtheta_h,\u_h,p_h)$ with positive temperature $\vtheta_h$ of \cref{prob:nacns} has the following structure preserving properties
            \begin{enumerate}
                \item In case of \cref{bc:periodic,bc:closed} 
                \begin{align*}
                    \la \dtau s_h,1\ra &=  \mathcal{D}_h^{n+1} + \la Q^*,\tfrac{1}{\vtheta_h^{n+1}}\ra,\\
                    \la \dtau e_{\mathrm{tot},h},1\ra &=  \la \dtau (e_h+\tfrac{1}{2}\snorm{\u_h}^2),1\ra = \mathcal{D}_\mathrm{num}^{n+1} + \la Q^*,1\ra + \la\b^*,\u_h^{n+1}\ra.
                \end{align*}
                \item In case of \cref{bc:thermal}
                \begin{align*}
                    \la \dtau (e_{\mathrm{tot},h}-\vtheta_{b}s_h),1\ra &= - \vtheta_{b}\mathcal{D}_h^{n+1} + \mathcal{D}_\mathrm{num}^{n+1} + \la\b^*,\u_h^{n+1}\ra + \la Q^*,1-\tfrac{\vtheta_{b}}{\vtheta_h^{n+1}}\ra. 
                \end{align*}
            \end{enumerate}
            with the entropy production
            \begin{equation*}
                \mathcal{D}_h^{n+1}=\la\tfrac{\M_h^*}{\vtheta_h^{n+1}\cdot \vtheta_h^*}\snorm{\mu_h^{n+1}}^2 + \tfrac{\K_h^*}{\vtheta_h^{n+1}\cdot(\vtheta_h^*)^3}\snorm{\nabla \vtheta_h^{n+1}}^2 + \tfrac{\eta_h^*}{\vtheta_h^{n+1}}\snorm{\Du_h^{n+1}}^2,1\ra \geq 0
            \end{equation*}
            and the numerical dissipation 
            \begin{align*}
                \mathcal{D}_\mathrm{num}^{n+1} = \tfrac{\tau}{2}&\left(\la\partial_{\vtheta\vtheta}f(\phi_h^{n},\zeta_1^{n+1}),(\dtau\vtheta_h)^2\ra - \la\snorm{\dtau\u_h}^2,1\ra\right.\\
                &~\left. - \la G_{\nabla\phi\nabla\phi}(\nabla\zeta_2^{n+1})\nabla\dtau\phi_h,\vtheta_h^{n+1}\nabla\dtau\phi_h\ra\right) \leq 0,
            \end{align*}
            for some $\zeta_1^{n+1}\in[\vtheta_h^{n},\vtheta_h^{n+1}],\zeta_2^{n+1}\in[\phi_h^{n},\phi_h^{n+1}]$.
        \end{theorem}
        \begin{proof}
            For the entropy production balance with \labelcref{bc:periodic,bc:closed} the choice $\omega=1$ yields 
            \begin{align*}
                \la\dtau s_h,1\ra&=\la\tfrac{\M_h^*}{\vtheta_h^*}\mu_h^{n+1},\tfrac{1}{\vtheta_h^{n+1}}\mu_h^{n+1}\ra+\la\tfrac{\eta_h^*}{\vtheta_h^{n+1}}\snorm{\Du_h^{n+1}}^2,1\ra\\
                &\quad + \la\tfrac{\K_h^*}{(\vtheta_h^*)^2}\nabla \vtheta_h^{n+1},\tfrac{1}{\vtheta_h^{n+1}\cdot \vtheta_h^*}\nabla \vtheta_h^{n+1}\ra + \la Q^*,\tfrac{1}{\vtheta_h^{n+1}}\ra\\
                &=\mathcal{D}_h^{n+1}  + \la Q^*,\tfrac{1}{\vtheta_h^{n+1}}\ra.
            \end{align*}
            To get the equation for the total energy as well as the exergy, we first use the identity from above for $e$, cf. \eqref{eq:energy}, and calculate as follows.
            \begin{align}\label{eq:disdtetot}
                \la\dtau e_{\mathrm{tot},h},1\ra&=\la\tfrac{1}{\tau}(F_h^{n+1}+\vtheta_h^{n+1}s_h^{n+1}-F_h^n-\vtheta_h^{n}s_h^{n})+\tfrac{1}{2\tau}(\snorm{\u^{n+1}_h}^2-\snorm{\u^{n}_h}^2),1\ra\notag\\
                &=\la\dtau f,1\ra+\la\tfrac{ \vtheta_h^{n+1}}{\tau}G(\nabla\phi_h^{n+1})-\tfrac{\vtheta_h^{n}}{\tau}G(\nabla\phi_h^n),1\ra\notag\\
                &\quad+\la\dtau s_h,\vtheta_h^{n+1}\ra+\la s_h^n,\dtau\vtheta_h\ra\notag\\
                &\quad+\la\dtau \u_h,\u_h^{n+1}\ra-\tfrac{\tau}{2}\la\snorm{\dtau\u_h}^2,1\ra.\notag\\
                \intertext{Then using $s=-\dtheta F=-(\dtheta f+G(\nabla\phi))$, cf. \eqref{eq:entropy}, we find}
                \la\dtau e_{\mathrm{tot},h},1\ra&=\la\dtau f,1\ra - \la\dtheta f_h^n,\dtau\vtheta_h\ra + \la \dtau G(\nabla\phi_h),\vtheta_h^{n+1}\ra\notag\\
                &\quad + \la\dtau s_h,\vtheta_h^{n+1}\ra + \la\dtau \u_h,\u_h^{n+1}\ra - \tfrac{\tau}{2}\la\snorm{\dtau\u_h}^2,1\ra.
            \end{align}
            Now for \labelcref{bc:periodic,bc:closed} we use \eqref{eq:dis3} with $\omega_h=\vtheta_h^{n+1}$, where the first term cancels out to get
            \begin{align*}
                \la\dtau e_{\mathrm{tot},h},1\ra&=\la\dtau f,1\ra - \la\dtheta f_h^{n},\dtau \vtheta_h\ra - \tfrac{\tau}{2}\la\snorm{\dtau\u_h}^2,1\ra\\
                &\quad+\la\tfrac{\M_h^*}{\vtheta_h^*}\mu_h^{n+1},\mu_h^{n+1}\ra + \la\eta_h^*\snorm{\Du_h^{n+1}}^2,1\ra + \la Q^*,1\ra\\
                &\quad + \la G_{\nabla\phi}(\nabla\phi_h^*)\dtau\phi_h + (\bsig_h^*+s_h^*\mathbf{I})\u_h^{n+1},\nabla\vtheta_h^{n+1}\ra\\
                &\quad + \la \dtau G(\nabla\phi_h),\vtheta_h^{n+1}\ra + \la\dtau \u_h,\u_h^{n+1}\ra.\\
                \intertext{By using \eqref{eq:dis2} with $\xi_h=\dtau\phi_h$ we find}
                \la\dtau e_{\mathrm{tot},h},1\ra&=\la\dtau f,1\ra - \la\dtheta f_h^{n},\dtau \vtheta_h\ra - \tfrac{\tau}{2}\la\snorm{\dtau\u_h}^2,1\ra\\
                &\quad+\la\tfrac{\M_h^*}{\vtheta_h^*}\mu_h^{n+1},\mu_h^{n+1}\ra + \la\eta_h^*\snorm{\Du_h^{n+1}}^2,1\ra + \la Q^*,1\ra\\
                &\quad+\la\mu_h^{n+1},\dtau\phi_h\ra - \la G_{\nabla\phi}(\nabla\phi_h^{n+1}),\vtheta_h^{n+1}\nabla\dtau\phi_h\ra\\
                &\quad - \la\dphi f(\phi_h^{n+1},\phi_h^n,\vtheta_h^{n+1}),\dtau\phi_h\ra + \la(\bsig_h^*+s_h^*\mathbf{I})\u_h^{n+1},\nabla\vtheta_h^{n+1}\ra\\
                &\quad + \la \dtau G(\nabla\phi_h),\vtheta_h^{n+1}\ra + \la\dtau \u_h,\u_h^{n+1}\ra.\\
                \intertext{Two final insertion of \eqref{eq:dis1} with $\psi_h=\mu_h^{n+1}$ and \eqref{eq:dis4} with $\vv_h=\u_h^{n+1}$ yield}
                \la\dtau e_{\mathrm{tot},h},1\ra&=\la\dtau f,1\ra - \la\dtheta f_h^{n},\dtau \vtheta_h\ra - \la\dphi f(\phi_h^{n+1},\phi_h^n,\vtheta_h^{n+1}),\dtau\phi_h\ra\\
                &\quad + \la \dtau G(\nabla\phi_h),\vtheta_h^{n+1}\ra - \la G_{\nabla\phi}(\nabla\phi_h^{n+1}),\vtheta_h^{n+1}\nabla\dtau\phi_h\ra  \\
                &\quad + \la\rho_{h}^{n+1},\div(\u_h^{n+1})\ra - \tfrac{\tau}{2}\la\snorm{\dtau\u_h}^2,1\ra + \la Q^*,1\ra + \la\b^*,\u_h^{n+1}\ra,\\
                \intertext{then using \eqref{eq:dis5} with $q=\rho_h^{n+1}$ and adding $\pm f(\phi_h^{n},\vtheta_h^{n+1})$ shows}
                \la\dtau e_{\mathrm{tot},h},1\ra&=\tfrac{1}{\tau}\la f(\phi_h^{n+1},\vtheta_h^{n+1})-f(\phi_h^{n},\vtheta_h^{n+1}),1\ra - \la\dphi f(\phi_h^{n+1},\phi_h^n,\vtheta_h^{n+1}),\dtau\phi_h\ra\\
                &\quad + \tfrac{1}{\tau}\la f(\phi_h^{n},\vtheta_h^{n+1})-f(\phi_h^{n},\vtheta_h^{n}),1\ra-\la\dtheta f(\phi_h^n,\vtheta_h^n),\dtau \vtheta_h\ra\\
                &\quad + \la \dtau G(\nabla\phi_h),\vtheta_h^{n+1}\ra - \la G_{\nabla\phi}(\nabla\phi_h^{n+1}),\vtheta_h^{n+1}\nabla\dtau\phi_h\ra  \\
                &\quad - \tfrac{\tau}{2}\la\snorm{\dtau\u_h}^2,1\ra + \la Q^*,1\ra + \la\b^*,\u_h^{n+1}\ra\\
                \intertext{and by using the Taylor-expansion of $f$ and $G$ in $\phi,\vtheta$ and $\nabla\phi$ one can see}
                \la\dtau e_{\mathrm{tot},h},1\ra&=\tfrac{\tau}{2}\Big(\la\partial_{\vtheta\vtheta}f(\phi_h^{n},\zeta_1^{n+1}),(d_\tau^{n+1} \vtheta_h)^2\ra - \la\snorm{d_\tau^{n+1}\u_h}^2,1\ra\\
                &\qquad  - \la G_{\nabla\phi\nabla\phi}(\nabla\zeta_2^{n+1})\nabla\dtau\phi_h,\vtheta_h^{n+1}\nabla\dtau\phi_h\ra\Big)\\
                &\quad+ \la Q^*,1\ra + \la\b^*,\u_h^{n+1}\ra\\
                &=\mathcal{D}_\mathrm{num}^{n+1} + \la Q^*,1\ra + \la\b^*,\u_h^{n+1}\ra,
            \end{align*}
            for some $\zeta_1^{n+1}\in[\vtheta_h^{n},\vtheta_h^{n+1}],\zeta_2^{n+1}\in[\phi_h^{n},\phi_h^{n+1}]$. For \labelcref{bc:thermal} we subtract $\vtheta_b\dtau s_h$ from both sides of \eqref{eq:disdtetot} to get
            \begin{align*}
                \la \dtau (e_h+\tfrac{1}{2}\snorm{\u_h}^2-\vtheta_{b}s_h),1\ra &= \la\dtau f,1\ra + \la \dtau G(\nabla\phi_h),\vtheta_h^{n+1}\ra\notag\\
                &\quad + \la\dtau s_h,\vtheta_h^{n+1}-\vtheta_b\ra - \la\dtheta f_h^n,\dtau\vtheta_h\ra\notag\\
                &\quad+\la\dtau \u_h,\u_h^{n+1}\ra - \tfrac{\tau}{2}\la\snorm{\dtau\u_h}^2,1\ra.
            \end{align*}
            then calculate as before, but with $\omega_h=\vtheta_h^{n+1}-\vtheta_b$. The negativity of $\mathcal{D}_\mathrm{num}^{k}$ follows from the properties of $f$ and $G$, cf. \labelcref{as:grad,as:pot}, as well as the positivity of $\vtheta_h$.
        \end{proof}

    \subsection{Existence of discrete solutions}
    
        To establish the existence of discrete solutions, the strategy from \cite{Hoehn26} is adopted, with the following restrictions imposed on the potential $f$.
        \begin{assump}
            \item We extend \labelcref{as:pot} and assume that $f$ is of the form
            \begin{equation*}
                f(\phi,\vtheta) = -b\vtheta\log(\vtheta) + \vtheta f_1(\phi) + f_2(\phi)
            \end{equation*}
            with $f_2(\phi)\geq -c$ for some constants $b,c>0$ independent of $\phi$ and $\vtheta$. Furthermore, we assume that there exists another constant $c_f>0$ independent of $\phi$ such that
            \begin{equation*}
                f_2(\phi) + c_ff_1(\phi) \geq -c.
            \end{equation*}
            By construction $F(\phi,\cdot):\RR_+\to\RR$ is strictly concave.\label{as:exergy}
            \item We assume that
            \begin{equation*}
                |f_\phi(\phi_h^{n+1},\phi_h^n,\vtheta_h^{n+1})| \leq C_f(|\vtheta_h^{n+1}|+1)(\snorm{\phi_h^{n+1}}^6 + \snorm{\phi_h^n}^6 + 1)
            \end{equation*}
            for some constant $C_f>0$.
            \label{as:fphi}
            \item We restrict the proof to $G(\nabla\phi)=\tfrac{\gamma}{2}\snorm{\nabla\phi}^2$. \label{as:Gphi}
            \item The mesh $\Th$ is quasi-uniform.\label{as:mesh}
        \end{assump}
        \begin{theorem}\label{thm:existence}
            Under the additional \cref{as:exergy,as:fphi,as:Gphi,as:mesh}, there exists at least one discrete solution $(\phi_h,\mu_{h},\vtheta_h,\u_h,p_h)$ of \cref{prob:nacns} with uniform positive temperature, i.e. $\vtheta_h^k>\underaccent{\bar}{\vtheta}>0$ for all $0\leq k\leq n_T$.
        \end{theorem}
        \begin{proof}
            The proof is a direct extension of the argument in \cite{Hoehn26}, with important details provided in the appendix. For simplicity, only the fully implicit case is considered, that is, $t^*=t^{n+1}$  and external forces are neglected.
        \end{proof}
    
\section{Benchmarks and numerical examples}\label{sec:num}

    In this section, we present the experimental order of convergence in both space and time, and demonstrate applications to melting and solidification processes. All simulations are conducted in two dimensions, where the gradient-free part of the potential $F$ is chosen following \cite{potential} as
    \begin{equation}\label{eq:f}
    f(\phi,\vtheta)=H_\vtheta(\vtheta)\phi^2(1-\phi)^2 - \mathcal{L}H_\phi(\phi)(\tfrac{\vtheta}{\vtheta_m}-1) - C_\mathrm{vsh}\left(\vtheta\log(\tfrac{\vtheta}{\vtheta_m})-(\vtheta-\vtheta_m)\right).   
    \end{equation}
    The potential consists of a local part, $f_\mathrm{loc} = H_\vtheta(\vtheta)\phi^2(1-\phi)^2$, which represents the local thermodynamic stability of a phase at a given temperature, and a heat part, $f_\mathrm{ht} = - \mathcal{L}H_\phi(\phi)(\tfrac{\vtheta}{\vtheta_m}-1) - C_\mathrm{vsh}\left(\vtheta\log(\tfrac{\vtheta}{\vtheta_m})-(\vtheta-\vtheta_m)\right)$, which modifies the double-well potential due to local temperature variations and incorporates both the latent heat term and the effects of heat absorption or release. The function $H_\vtheta(\vtheta) = H_\mathrm{pt}-H_\mathrm{cf}(\vtheta-\vtheta_m)$ defines the temperature-dependent barrier height, scaling the double-well potential by the potential factor $H_\mathrm{pt}$ at the melting temperature $\vtheta_m$, and is adjusted by the configurational factor $H_\mathrm{cf}$. Similarly, $H_\phi(\phi)=\phi^3(6\phi^2-15\phi+10)$ serves as an interpolation function in $\phi$, which, together with the constant $\mathcal{L}$, scales the latent heat contribution. Finally, $C_\mathrm{vsh}$ denotes the volumetric specific heat. The viscosity $\eta$ is defined as
    \begin{equation}
        \eta(\phi)=\frac{\eta_l\eta_s}{H_\phi(\phi)(\eta_s-\eta_l)+\eta_l}
    \end{equation}
    where $\eta_l$ and $\eta_s$ denote the viscosities of the liquid and solid phase, respectively, while the gradient contribution $G(\nabla\phi)$ in the potential $F$ will be specified later according to the requirements of each application.
    \begin{remark}
        As $H_\phi$ is considered to be an interpolation function, it will be truncated to $[0,1]$ with a constant continuation on $\RR$. Therefore $f$ from \eqref{eq:f} fulfills \cref{as:exergy,as:fphi} with $c_f=\vtheta_m$ and $c=-C_\mathrm{vsh}\vtheta_m\min\left(-1,\log(\vtheta_m)\right)$.
    \end{remark}
    At each timestep, the nonlinear system is solved using Newton's method with an increment tolerance of $10^{-8}$, where the resulting linear system is solved directly. The time average \eqref{eq:fphi} is computed using a Gauss-Legendre quadrature rule with five nodes, which is exact in this context. Also all $g^*$ are taken explicitly, i.e. $g^*=g^n$.
    
    \subsection{Experimental order of convergence}
    
        In this section, we present numerical results on the convergence of \cref{prob:nacns}. The convergence study is conducted for $G(\nabla\phi)=\tfrac{\gamma^2}{2}\snorm{\nabla\phi}^2$, where $\gamma$ is the interface parameter. The mobility $\M$, heat conductivity $\K$, and external forces $Q$ and $\b$ are set as constants according to the following parameters:\\
        \begin{minipage}[t]{0.32\textwidth}
            \vspace{0.1cm}
            \begin{itemize}
                \item $\M=10$
                \item $\K=0.01$
                \item $Q=0$
                \item $\b=\mathbf{0}$
            \end{itemize}
            \vspace{0.2cm}
        \end{minipage}\hfill
        \begin{minipage}[t]{0.32\textwidth}
            \vspace{0.1cm}
            \begin{itemize}
                \item $\gamma=0.05$
                \item $\eta_s=1$
                \item $\eta_l=0.001$
                \item $\vtheta_m=1$
            \end{itemize}
            \vspace{0.2cm}
        \end{minipage}\hfill
        \begin{minipage}[t]{0.32\textwidth}
            \vspace{0.1cm}
            \begin{itemize}
                \item $\mathcal{L}=1$
                \item $H_\mathrm{pt}=1$
                \item $H_\mathrm{cf}=0.1$
                \item $C_\mathrm{vsh}=1$
            \end{itemize}
            \vspace{0.2cm}
        \end{minipage}
        
        We also restrict our analysis to \cref{bc:periodic}, as we expect similar results for all choices. Since different convergence orders are expected in space and time, we first examine spatial convergence by fixing the time step size to $\tau=2^{-11}\cdot10^{-1}\approx 4.8\cdot 10^{-5}$ with an end time of $T=0.05$, and consider spatial resolutions of $h_k=2^{-k}$ for $k=4,5,6,7$ on the domain $\Omega=[0,1]^2\subset\RR^2$. The initial conditions are chosen as follows:
        \begin{align*}
            \phi_0(x,y)&:=0.5\left(\operatorname{trans}_{[0,1]}\left(\frac{\sqrt{(x-0.25)^2+(y-0.25)^2}-0.15)}{3.25\cdot 10^{-2}}\right)\right.\\
            &\qquad-\left.\operatorname{trans}_{[0,1]}\left(\frac{\sqrt{(x-0.75)^2+(y-0.75)^2}-0.15)}{3.25\cdot 10^{-2}}\right)+1\right),\\
            \vtheta_0(x,y)&:=\vtheta_m\cdot\exp\left(\ln(0.5)\cdot(\sin(4\pi x)\sin(4\pi y)+1)(\sin(2\pi x)+\sin(2\pi y))\cdot0.5\right),
        \end{align*}
        where
        \begin{equation*}
            \operatorname{trans}_{[0,1]}(z)=\frac{1-\tanh(z)}{2}\in[0,1].
        \end{equation*}
        The temporal evolution for this setup is considered in \cref{sec:ex_melt}.
        The errors under consideration are divided into two groups determined by their used time norm, and since no analytical solution is available, they are computed with respect to the corresponding refined solution. The errors of the first group, being the ones using the discrete $L^\infty$ norm in time, are defined as
        \begin{align*}
            \mathrm{err}(\nabla\phi,h_k)&:= \max_{n\in\Itau}\norm{\phi^n_{h_k}-\phi^n_{h_{k+1}}}_{H^1},
            &
            \mathrm{err}(\vtheta,h_k)&:=\max_{n\in\Itau}\norm{\vtheta^n_{h_k}-\vtheta^n_{h_{k+1}}}_{L^2},
            \\
            \mathrm{err}(\u,h_k)&:=\max_{n\in\Itau}\norm{\u^n_{h_k}-\u^n_{h_{k+1}}}_{L^2},
        \end{align*}
        while the errors of the second group, using the discrete $L^2$ norm in time, are defined as
        \begin{align*}
            \mathrm{err}(\mu,h_k)&:=\sqrt{\tau\sum_{n=1}^{n_\tau}\norm{\mu^n_{h_k}-\mu^n_{h_{k+1}}}_{L^2}^2},
            &
            \mathrm{err}(\nabla\vtheta,h_k)&:=\sqrt{\tau\sum_{n=1}^{n_\tau}\norm{\vtheta^n_{h_k}-\vtheta^n_{h_{k+1}}}_{H^1}^2}
            \\
            \mathrm{err}(\nabla\u,h_k)&:=\sqrt{\tau\sum_{n=1}^{n_\tau}\norm{\u^n_{h_k}-\u^n_{h_{k+1}}}_{H^1}^2},
            &
            \mathrm{err}(p,h_k)&:=\sqrt{\tau\sum_{n=1}^{n_\tau}\norm{p^n_{h_k}-p^n_{h_{k+1}}}_{L^2}^2}.
        \end{align*}
        The experimental order of convergence for the above errors is then computed as
        \begin{equation*}
            \mathrm{eoc}_{k}:=\log_2\left(\frac{\mathrm{err}(a,h_{k-1})}{\mathrm{err}(a,h_{k})}\right),
        \end{equation*}
        for $a\in\{\nabla\phi,\mu,\vtheta,\nabla\vtheta,\u,\nabla\u,p\}$. In \cref{tab:sprates1}, we present the errors and rates in the discrete $L^\infty$ time norms for the spatial convergence test. The data indicates first order accuracy for the $L^\infty(0,T;H^1(\Omega))$-norm of $\phi_h$ and second order accuracy for the $L^\infty(0,T;L^2(\Omega))$-norm of $\vtheta_h$, as suspected from the linear elements in space. For $\u_h$ we observe a better order convergence in the $L^\infty(0,T;L^2(\Omega))$-norm in \cref{tab:sprates1} and 2nd order in the $L^2(0,T;H^1(\Omega))$-norm in \cref{tab:sprates2}, as expected for $P_2$-elements for $\u$ in space. Coherently in \cref{tab:sprates2} we see the expected 2nd order convergence for the $L^2(0,T;L^2(\Omega)$-norm of $p_h$. Lastly the $L^2(0,T;L^2(\Omega)$-norm of $\mu_h$ and $L^2(0,T;H^1(\Omega)$-norm of $\vtheta_h$ in \cref{tab:sprates2} indicate second and first order convergence in space as predicted by the use of linear elements.
        \begin{table}[htbp!]
            \centering
            \caption{Errors and rates in $L^\infty$ time norms for spatial convergence.\label{tab:sprates1}} 
            \begin{tabular}{c|l|c|l|c|l|c}
                $ k $ & $\mathrm{err}(\nabla\phi,h_k)$ &  $\mathrm{eoc}_k$ & $\mathrm{err}(\vtheta,h_k)$ & $\mathrm{eoc}_k$ & $\mathrm{err}(\u,h_k)$ & $\mathrm{eoc}_k$\\
                \hline
                0 & $8.94\cdot 10^{-1}$ & -- & $2.59\cdot 10^{-2}$ & -- & $9.55\cdot 10^{-3}$ & --\\
                1 & $4.98\cdot 10^{-1}$ & $0.85$ & $7.43\cdot 10^{-3}$ & $1.80$ & $1.74\cdot 10^{-3}$ & $2.46$\\
                2 & $2.63\cdot 10^{-1}$ & $0.92$ & $1.77\cdot 10^{-3}$ & $2.07$ & $1.60\cdot 10^{-4}$ & $3.44$\\
                3 & $1.34\cdot 10^{-1}$ & $0.97$ & $4.62\cdot 10^{-4}$ & $1.94$ & $3.09\cdot 10^{-5}$ & $2.37$\\
            \end{tabular}
        \end{table}
        \begin{table}[htbp!]
            \centering
            \caption{Errors and rates in $L^2$ time norms for spatial convergence.\label{tab:sprates2}} 
            \begin{tabular}{c|l|c|l|c|l|c|l|c|l|c}
                $ k $ & $\mathrm{err}(\mu,h_k)$ &  $\mathrm{eoc}_k$ & $\mathrm{err}(\nabla\vtheta,h_k)$ & $\mathrm{eoc}_k$ & $\mathrm{err}(\nabla\u,h_k)$ & $\mathrm{eoc}_k$ & $\mathrm{err}(p,h_k)$ & $\mathrm{eoc}_k$\\
                \hline
                0 & $9.84\cdot 10^{-3}$ & -- & $3.70\cdot 10^{-1}$ & -- & $1.39\cdot 10^{-1}$ & -- & $4.92\cdot 10^{-3}$ & --\\
                1 & $3.12\cdot 10^{-3}$ & $1.66$ & $2.05\cdot 10^{-1}$ & $0.85$ & $6.10\cdot 10^{-2}$ & $1.19$ & $1.68\cdot 10^{-3}$ & $1.55$\\
                2 & $7.85\cdot 10^{-4}$ & $1.99$ & $1.06\cdot 10^{-1}$ & $0.95$ & $1.14\cdot 10^{-2}$ & $2.42$ & $3.74\cdot 10^{-4}$ & $2.16$\\
                3 & $2.26\cdot 10^{-4}$ & $1.80$ & $5.35\cdot 10^{-2}$ & $0.99$ & $1.75\cdot 10^{-3}$ & $2.70$ & $9.70\cdot 10^{-5}$ & $1.95$\\
            \end{tabular}
        \end{table}
        
        For the time convergence study we use the same set of parameters as before but this time fixing the spatial resolution to a maximum mesh diameter of $h=2^{-6}$ while considering time steps sizes of $\tau=2^{-k}\cdot10^{-1}$ for $k=1,...,11$. The errors under consideration are again divided into the discrete $L^\infty$ time errors
        \begin{align*}
            \mathrm{err}(\nabla\phi,\tau_k)&:= \max_{n\in\Itauk}\norm{\phi^n_{h,\tau_k}-\phi^{2n}_{h,\tau_{k+1}}}_{H^1},
            &
            \mathrm{err}(\vtheta,\tau_k)&:=\max_{n\in\Itauk}\norm{\vtheta^n_{h,\tau_k}-\vtheta^{2n}_{h,\tau_{k+1}}}_{L^2},
            \\
            \mathrm{err}(\u,\tau_k)&:=\max_{n\in\Itauk}\norm{\u^n_{h,\tau_k}-\u^{2n}_{h,\tau_{k+1}}}_{L^2},
        \end{align*}
        and the discrete $L^2$ time errors
        \begin{align*}
            \mathrm{err}(\mu,\tau_k)&:=\sqrt{\tau_k\sum_{n=1}^{n_{\tau_k}}\norm{\mu^n_{h,\tau_k}-\mu^{2n}_{h,\tau_{k+1}}}_{L^2}^2},
            &
            \mathrm{err}(\nabla\vtheta,\tau_k)&:=\sqrt{\tau_k\sum_{n=1}^{n_{\tau_k}}\norm{\vtheta^n_{h,\tau_k}-\vtheta^{2n}_{h,\tau_{k+1}}}_{H^1}^2},
            \\
            \mathrm{err}(\nabla\u,\tau_k)&:=\sqrt{\tau_k\sum_{n=1}^{n_{\tau_k}}\norm{\u^n_{h,\tau_k}-\u^{2n}_{h,\tau_{k+1}}}_{H^1}^2},
            &
            \mathrm{err}(p,\tau_k)&:=\sqrt{\tau_k\sum_{n=1}^{n_{\tau_k}}\norm{p^n_{h,\tau_k}-p^{2n}_{h,\tau_{k+1}}}_{L^2}^2}.
        \end{align*}
        Note, that we only use the evaluations on the respective corse time grid of the two compared functions, therefore avoid porjections and interpolations between different grids in time.
        The experimental order of convergence (eoc) in time is defined similar to above as 
        \begin{equation*}
            \mathrm{eoc}_k:=\log_2\left(\frac{\mathrm{err}(a,\tau_{k-1})}{\mathrm{err}(a,\tau_{k})}\right),
        \end{equation*}
        for $a\in\{\nabla\phi,\mu,\vtheta,\nabla\vtheta,\u,\nabla\u,p\}$. The errors and rates for the time convergence test can be found in \cref{tab:tirates1,tab:tirates2}. As expected, optimal first-order convergence in time is achieved in every norm.
        \begin{table}[htbp!]
            \centering
            \caption{Errors and rates in $L^\infty$-norm for convergence in time.\label{tab:tirates1}} 
            \begin{tabular}{c|l|c|l|c|l|c}
                $ k $ & $\mathrm{err}(\nabla\phi,\tau_k)$ &  $\mathrm{eoc}_k$ & $\mathrm{err}(\vtheta,\tau_k)$ & $\mathrm{eoc}_k$ & $\mathrm{err}(\u,\tau_k)$ &  $\mathrm{eoc}_k$ \\
                \hline
                7 & $4.34\cdot 10^{-3}$ & $0.99$ & $3.15\cdot 10^{-4}$ & $0.99$ & $6.20\cdot 10^{-5}$ & $0.99$\\
                8 & $2.18\cdot 10^{-3}$ & $0.99$ & $1.58\cdot 10^{-4}$ & $0.99$ & $3.11\cdot 10^{-5}$ & $0.99$\\
                9 & $1.09\cdot 10^{-3}$ & $1.00$ & $7.93\cdot 10^{-5}$ & $1.00$ & $1.56\cdot 10^{-5}$ & $1.00$\\
                10 & $5.46\cdot 10^{-4}$ & $1.00$ & $3.97\cdot 10^{-5}$ & $1.00$ & $7.81\cdot 10^{-6}$ & $1.00$\\
            \end{tabular}
        \end{table}
        \begin{table}[htbp!]
            \centering
            \caption{Errors and rates in $L^2$-norm for convergence in time.\label{tab:tirates2}} 
            \begin{tabular}{c|l|c|l|c|l|c|l|c|l|c}
                $ k $ & $\mathrm{err}(\mu,\tau_k)$ &  $\mathrm{eoc}_k$ & $\mathrm{err}(\nabla\vtheta,\tau_k)$ & $\mathrm{eoc}_k$ & $\mathrm{err}(\nabla\u,\tau_k)$ &  $\mathrm{eoc}_k$ & $\mathrm{err}(p,\tau_k)$ &  $\mathrm{eoc}_k$  \\
                \hline
                7 & $2.32\cdot 10^{-4}$ & $0.98$ & $2.58\cdot 10^{-3}$ & $0.98$ & $3.03\cdot 10^{-4}$ & $0.99$ & $1.18\cdot 10^{-4}$ & $0.98$\\
                8 & $1.17\cdot 10^{-4}$ & $0.99$ & $1.30\cdot 10^{-3}$ & $0.99$ & $1.52\cdot 10^{-4}$ & $1.00$ & $5.96\cdot 10^{-5}$ & $0.99$\\
                9 & $5.89\cdot 10^{-5}$ & $0.99$ & $6.52\cdot 10^{-4}$ & $0.99$ & $7.60\cdot 10^{-5}$ & $1.00$ & $2.99\cdot 10^{-5}$ & $1.00$\\
                10 & $2.96\cdot 10^{-5}$ & $0.99$ & $3.27\cdot 10^{-4}$ & $1.00$ & $3.80\cdot 10^{-5}$ & $1.00$ & $1.50\cdot 10^{-5}$ & $1.00$\\
            \end{tabular}
        \end{table}
    
    \subsection{Simultaneous melting and solidification example }\label{sec:ex_melt}

        With the first example we will demonstrate the structure preserving properties of the discretisation while also showing both melting and solidification side by side. This is done by considering only \cref{bc:periodic} on the domain $\Omega=[0,1]^2\subset\RR^2$ without external sources and forces, i.e. $Q=\b=0$, together with the following initial condition:
        \begin{align*}
            \phi_0(x,y)&:=0.5\left(\operatorname{trans}_{[0,1]}\left(\frac{\sqrt{(x-0.25)^2+(y-0.25)^2}-0.15)}{3.25\cdot 10^{-2}}\right)\right.\\
            &\qquad-\left.\operatorname{trans}_{[0,1]}\left(\frac{\sqrt{(x-0.75)^2+(y-0.75)^2}-0.15)}{3.25\cdot 10^{-2}}\right)+1\right),\\
            \vtheta_0(x,y)&:=\vtheta_m\exp\left(\ln(0.5)\cdot(\sin(4\pi x)\sin(4\pi y)+1)(\sin(2\pi x)+\sin(2\pi y))\cdot0.5\right),\\
            \u(x,y)&:=(0,0)^\top.
        \end{align*}
         Both initial conditions for $\phi$ and $\vtheta$ are depicted in \cref{fig:ex_melt_init}. 
         \begin{figure}[htbp!]
            \centering
            \includegraphics[trim={16cm 0cm 9cm 0cm},clip,scale=0.14]{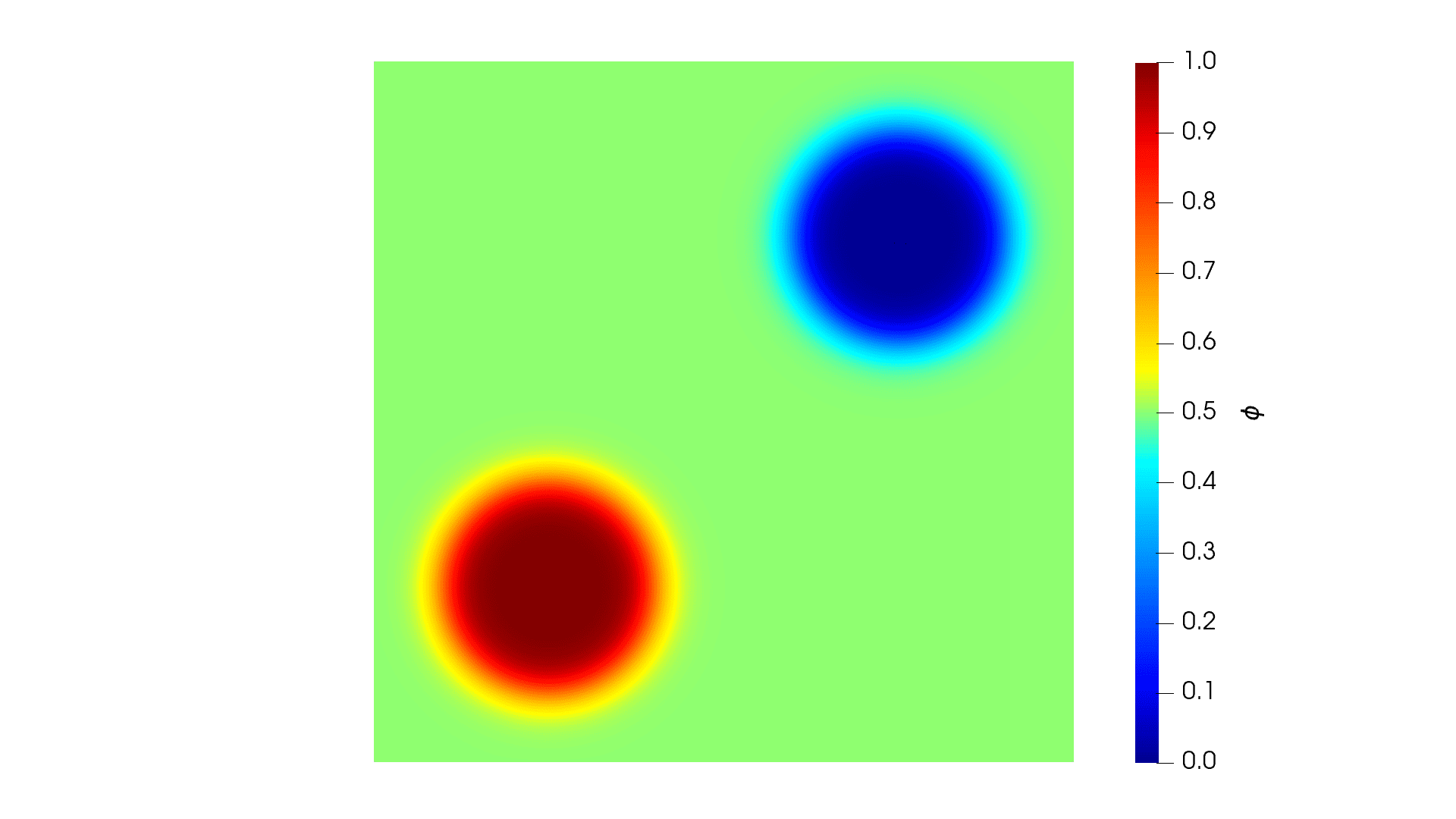}
            \includegraphics[trim={16cm 0cm 9cm 0cm},clip,scale=0.14]{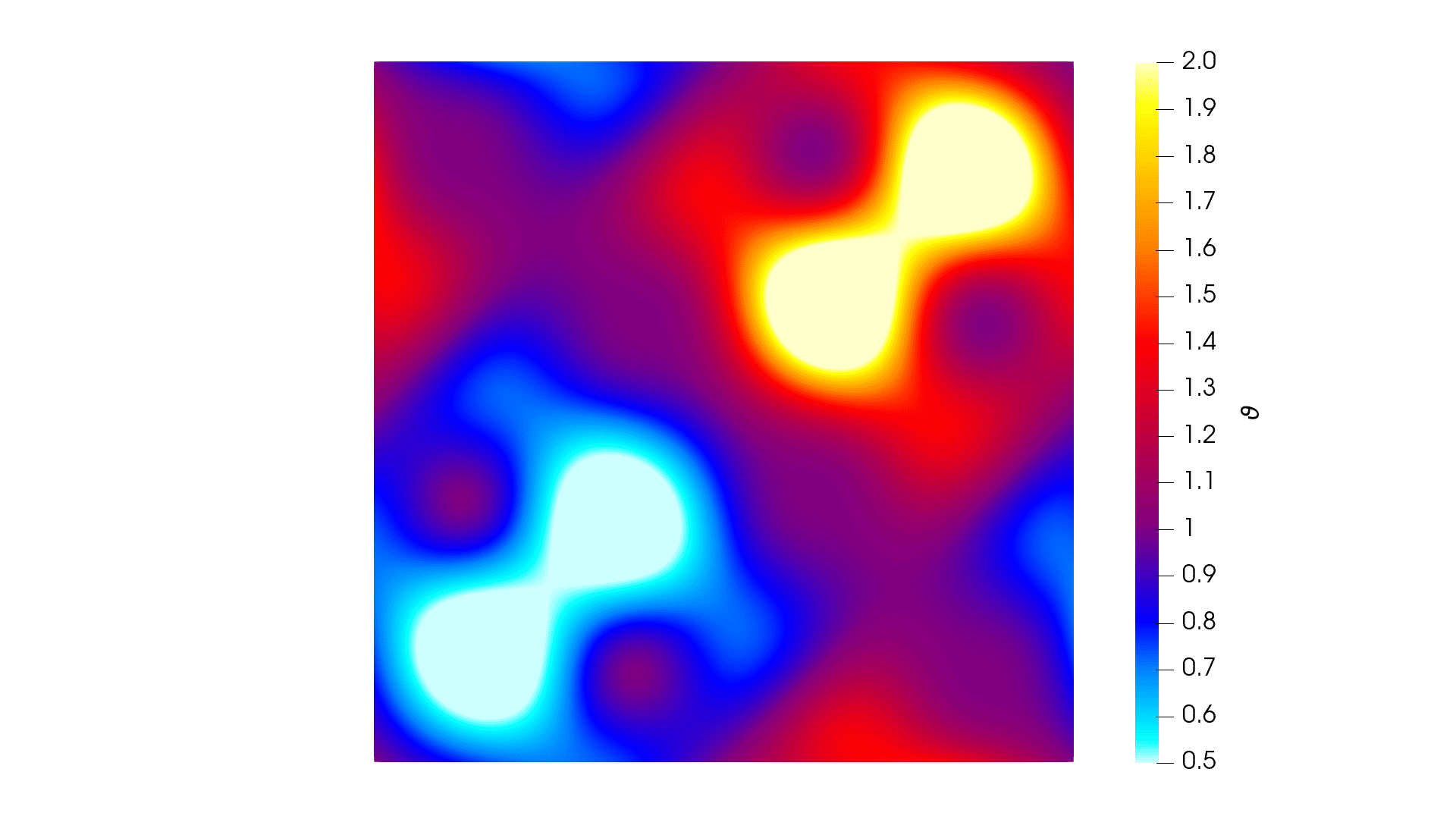} 
            \caption{Initial condition for the first example of the phase variable $\phi$ (left) and the temperature $\vtheta$ (right).}
            \label{fig:ex_melt_init}
        \end{figure}
        
        For the phase variable $\phi$ we consider a mixture of solid and melt around $\phi=0.5$ which contains a completely solidified phase in the right upper corner, i.e. $\phi=0$, and a completely melted region in the left lower corner, i.e. $\phi=1$. For $\vtheta$ we impose a profile which allows simultaneous overheating and undercooling in both pure regions, with values around the melting temperature $\vtheta_m=1$ for the rest of the domain. The remaining parameters are set similar to the convergence study, only sharpening the interface parameter $\gamma$ to $0.025$.
        \begin{figure}[htbp!]
            \centering
            \hspace*{-1.5em}
            \begin{tabular}{c@{}c@{}c@{}c@{}c@{}c@{}}
                & $t=0.5$ & $t=2$ & $t=3$ & $t=5$ & 
                \\[-0.5em]
                \includegraphics[trim={8.5cm 0cm 52.2cm 0cm},clip,scale=0.09]{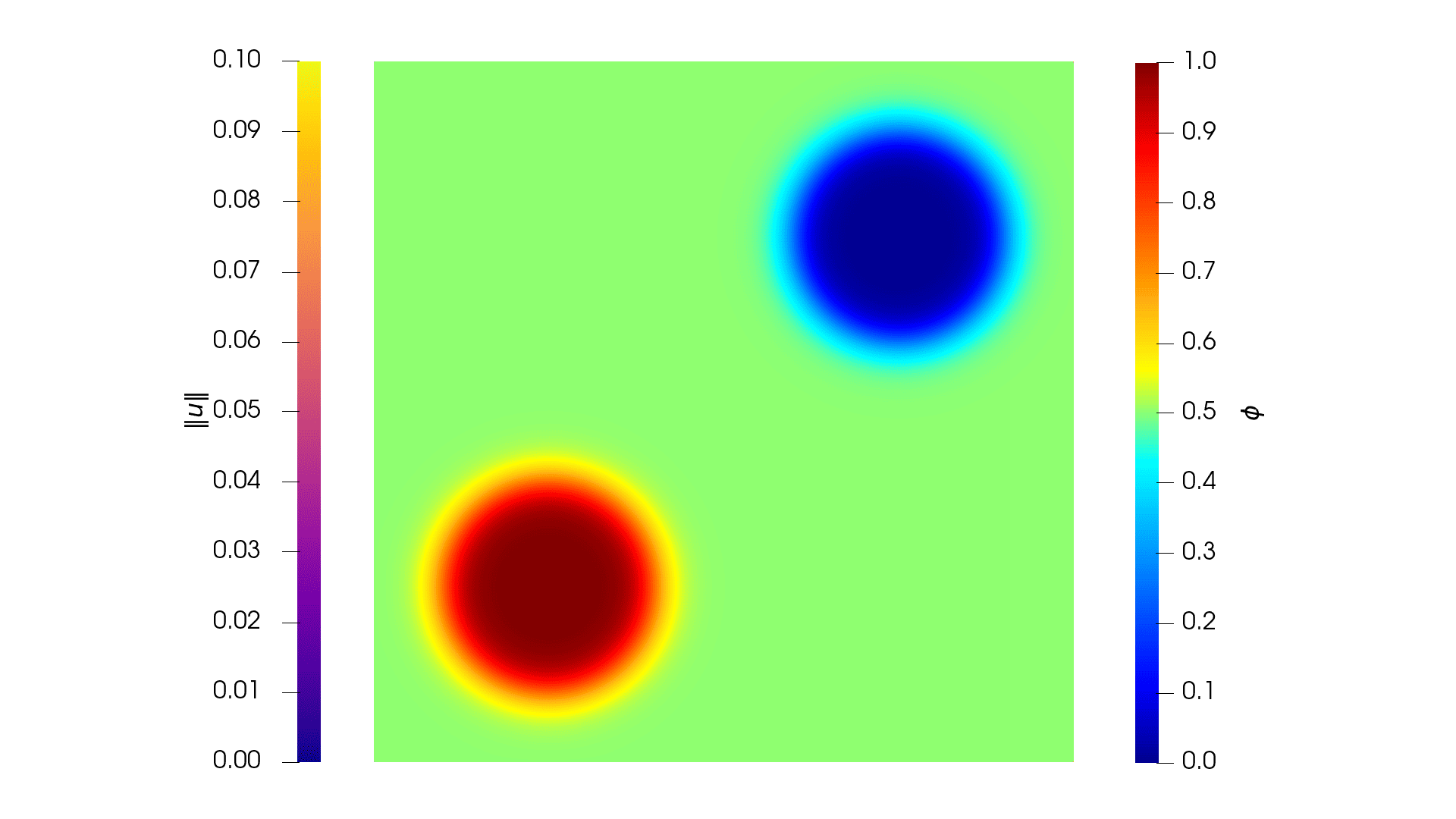}
                &
                \includegraphics[trim={17cm 0cm 17cm 0cm},clip,scale=0.09]{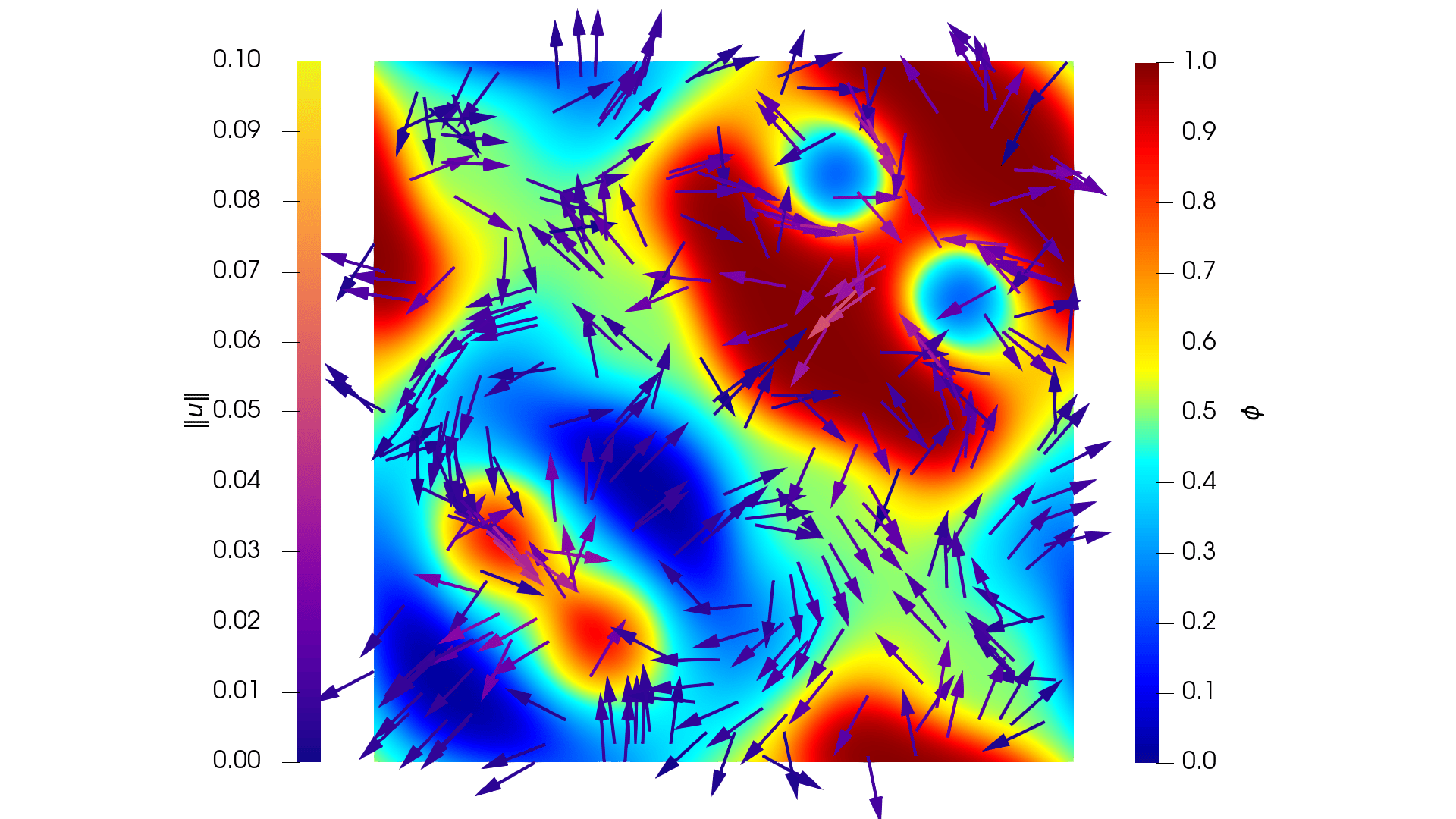} 
                &
                \includegraphics[trim={17cm 0cm 17cm 0cm},clip,scale=0.09]{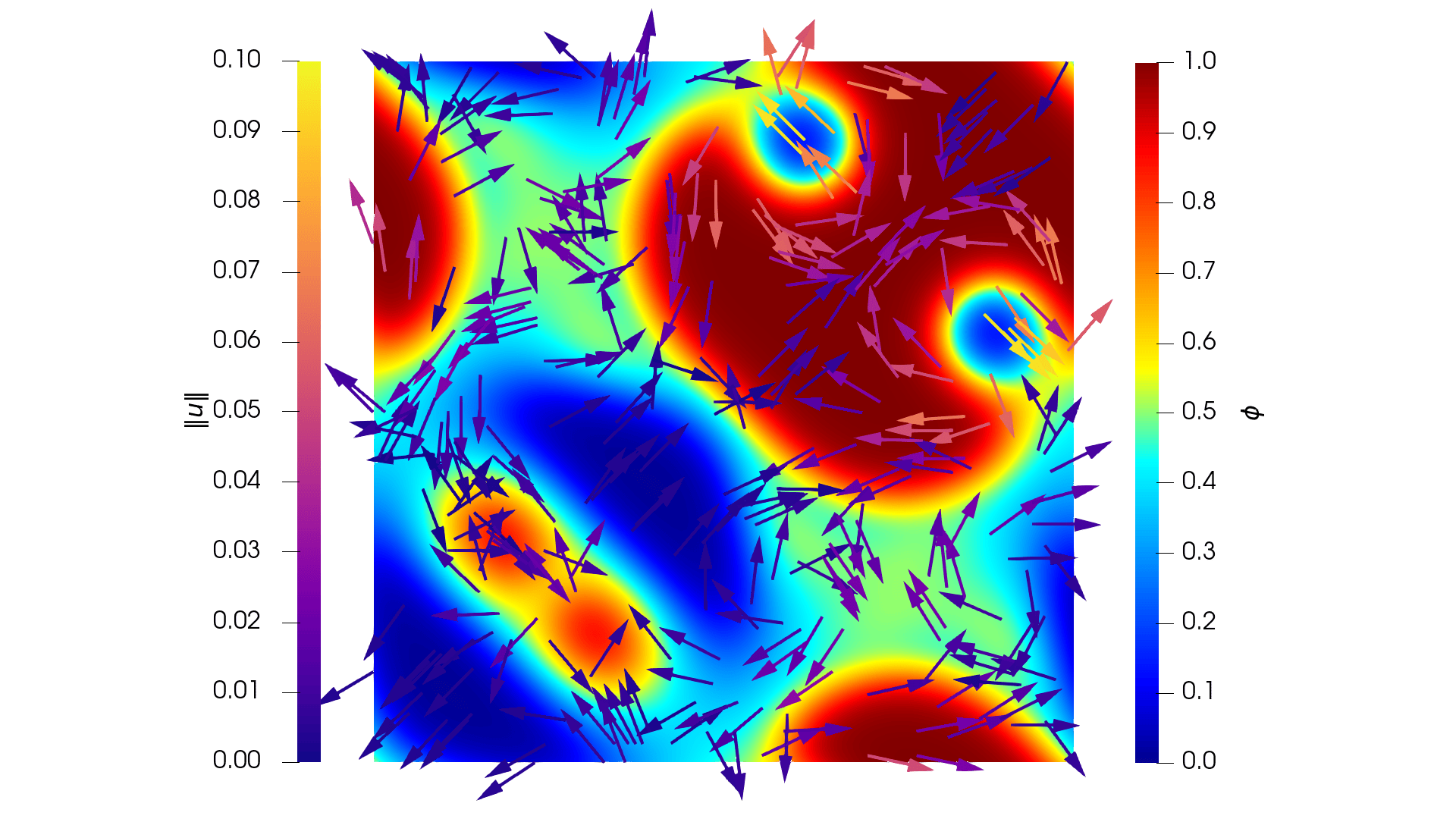}
                &
                \includegraphics[trim={17cm 0cm 17cm 0cm},clip,scale=0.09]{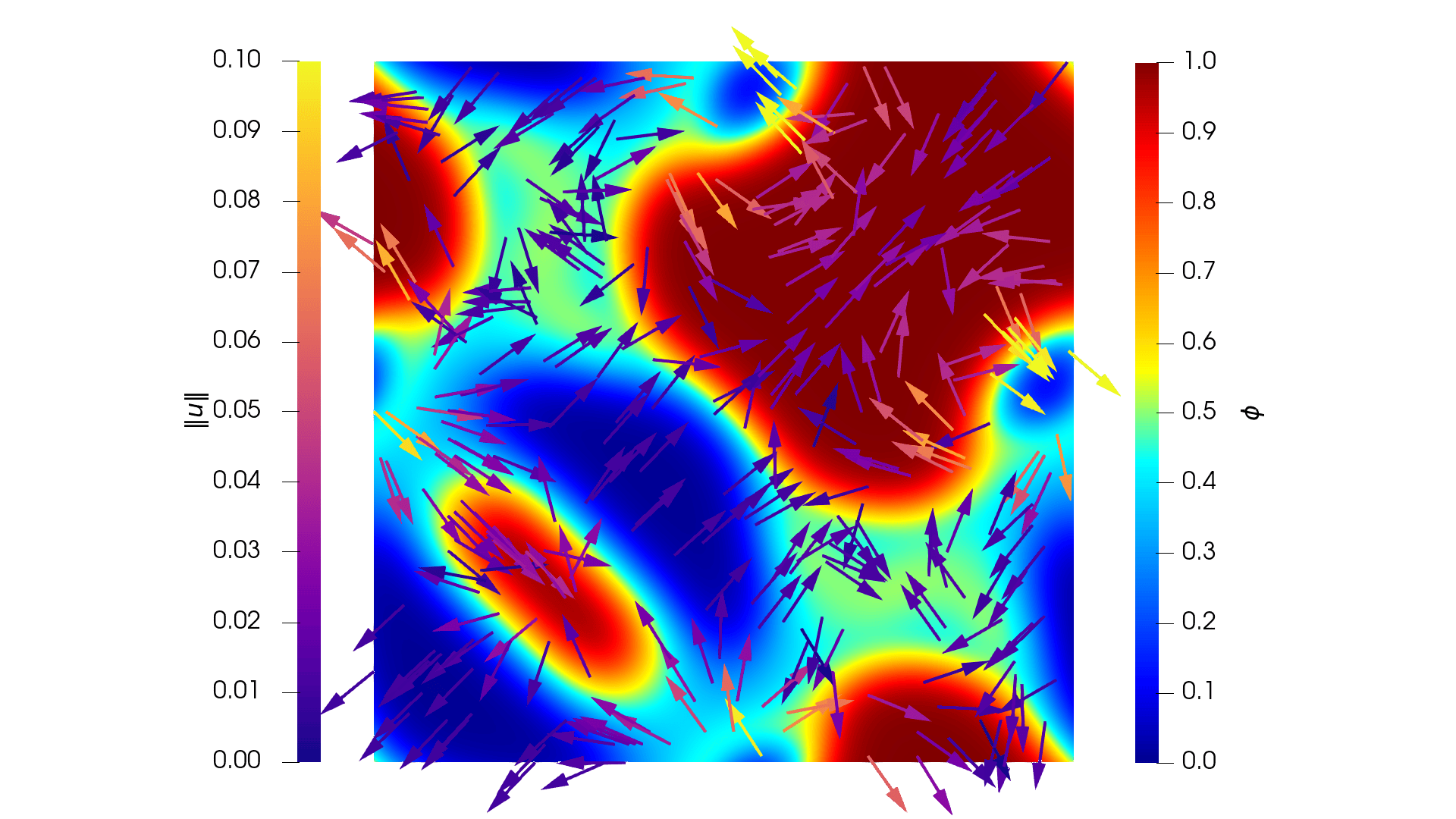} 
                &
                \includegraphics[trim={17cm 0cm 17cm 0cm},clip,scale=0.09]{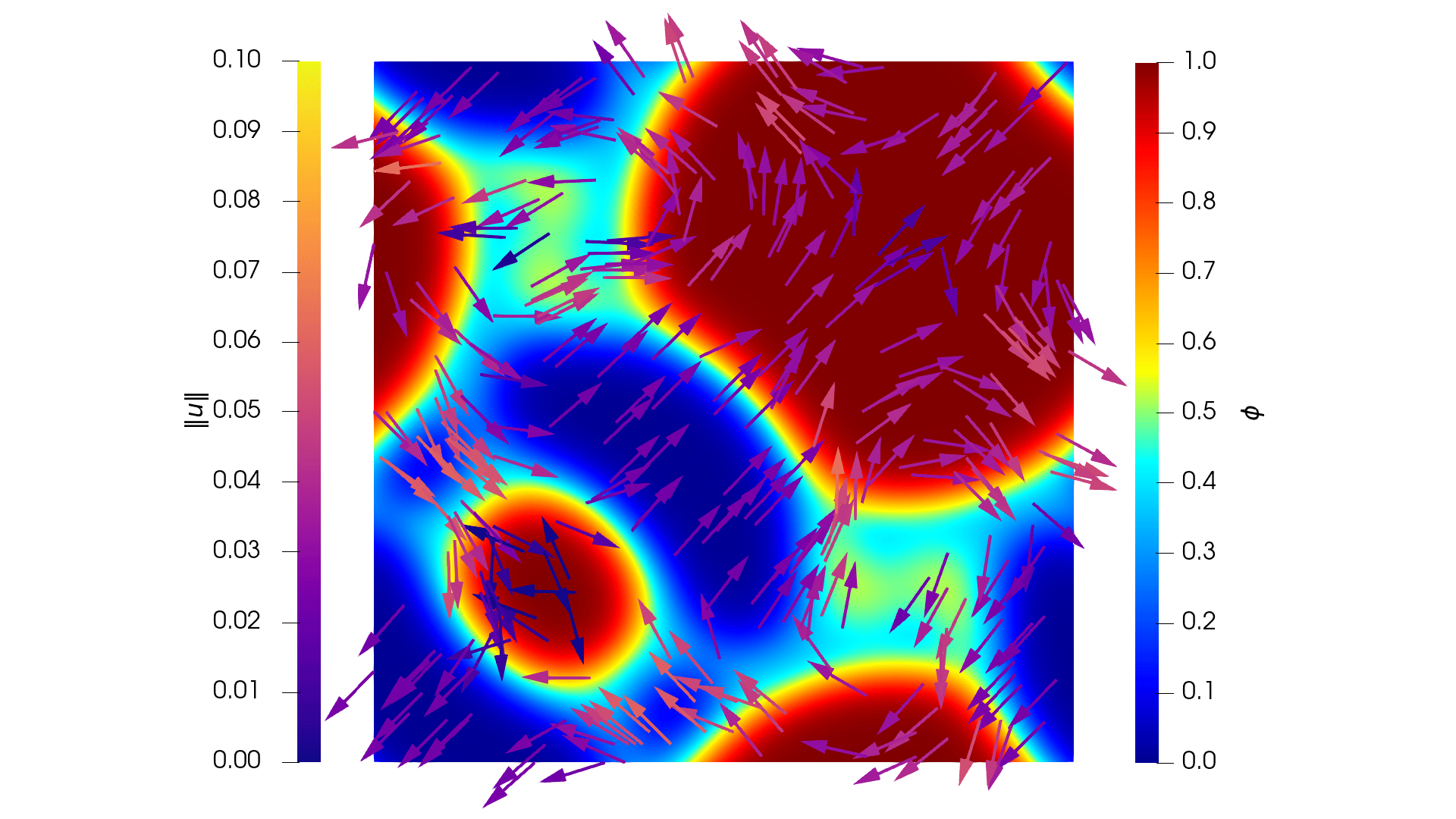}
                & 
                \includegraphics[trim={52.5cm 0cm 8.5cm 0cm},clip,scale=0.09]{Bilder/Melt/With_Flow/Phi_U_0.png}\\[-1em]
                \includegraphics[trim={8.5cm 0cm 52.2cm 0cm},clip,scale=0.09]{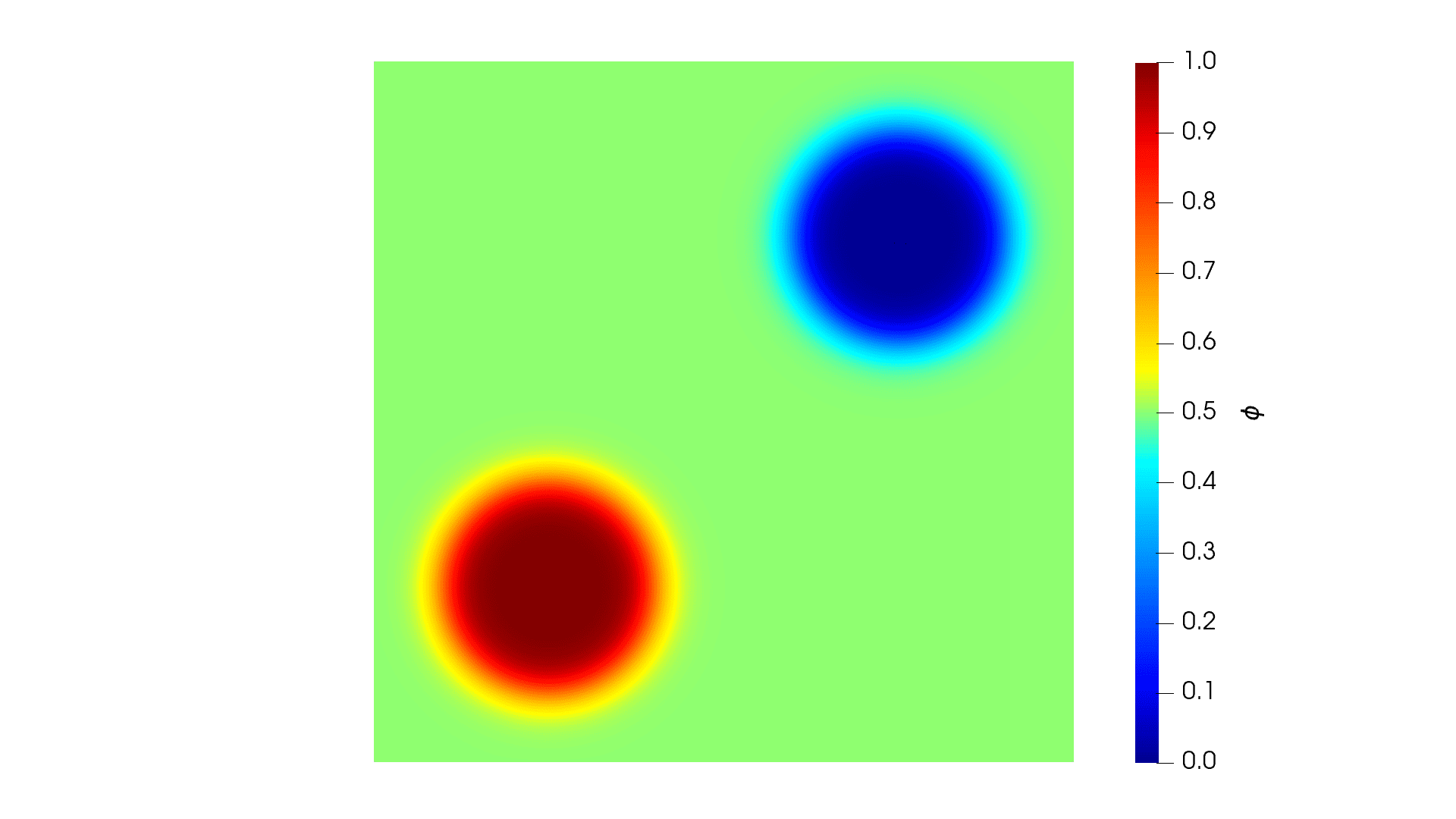}
                &
                \includegraphics[trim={17cm 0cm 17cm 0cm},clip,scale=0.09]{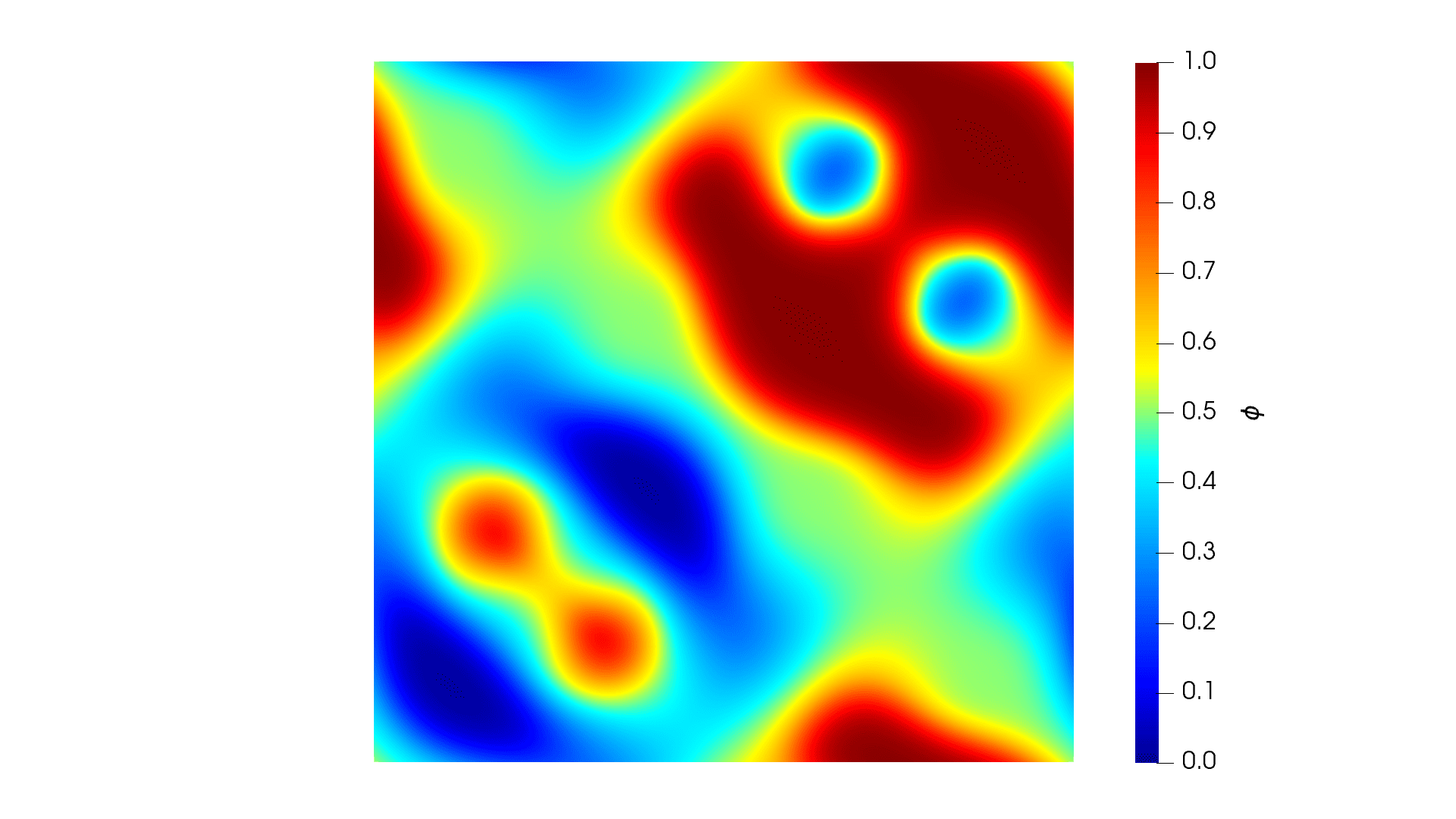} 
                &
                \includegraphics[trim={17cm 0cm 17cm 0cm},clip,scale=0.09]{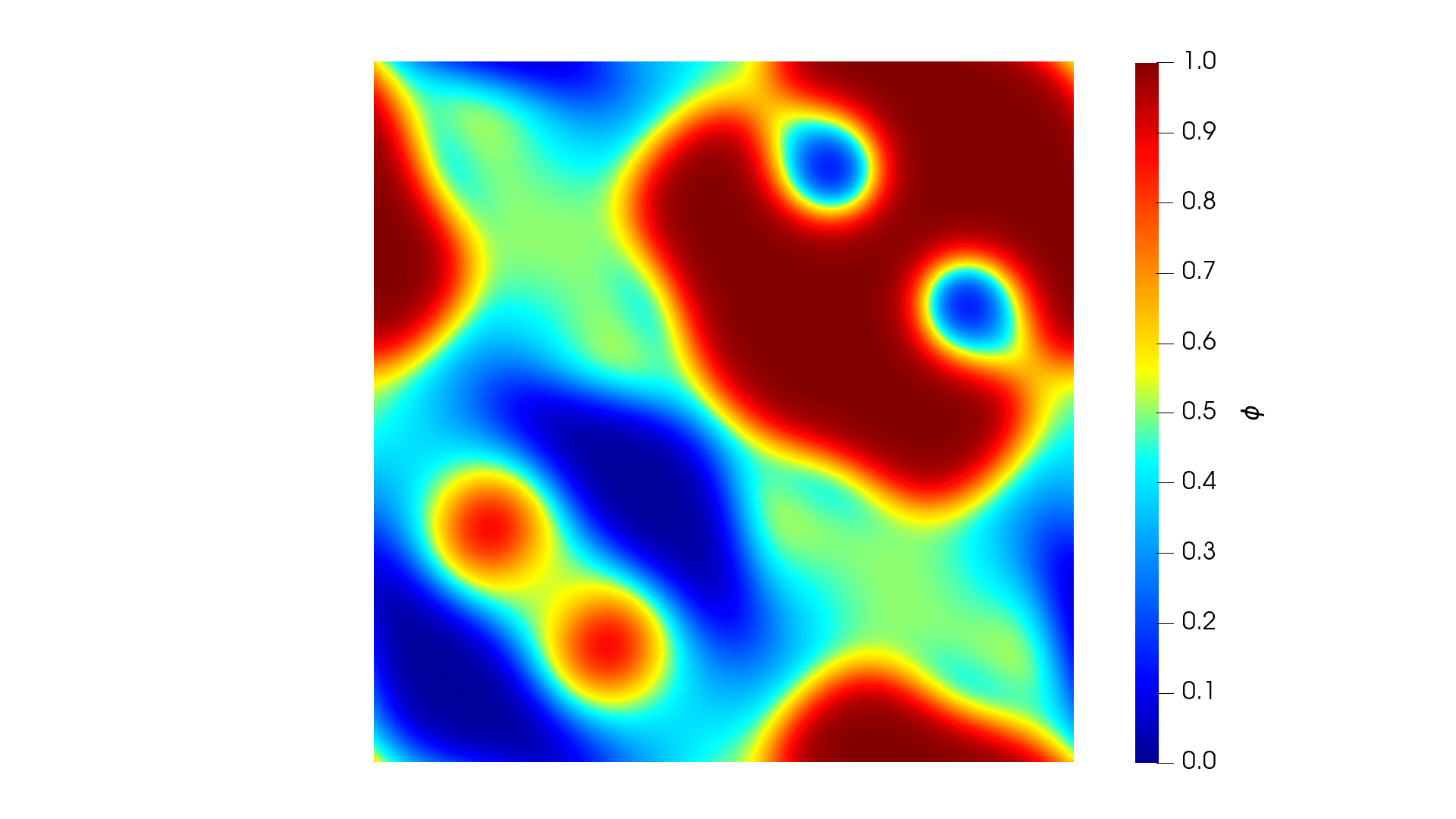}
                &
                \includegraphics[trim={17cm 0cm 17cm 0cm},clip,scale=0.09]{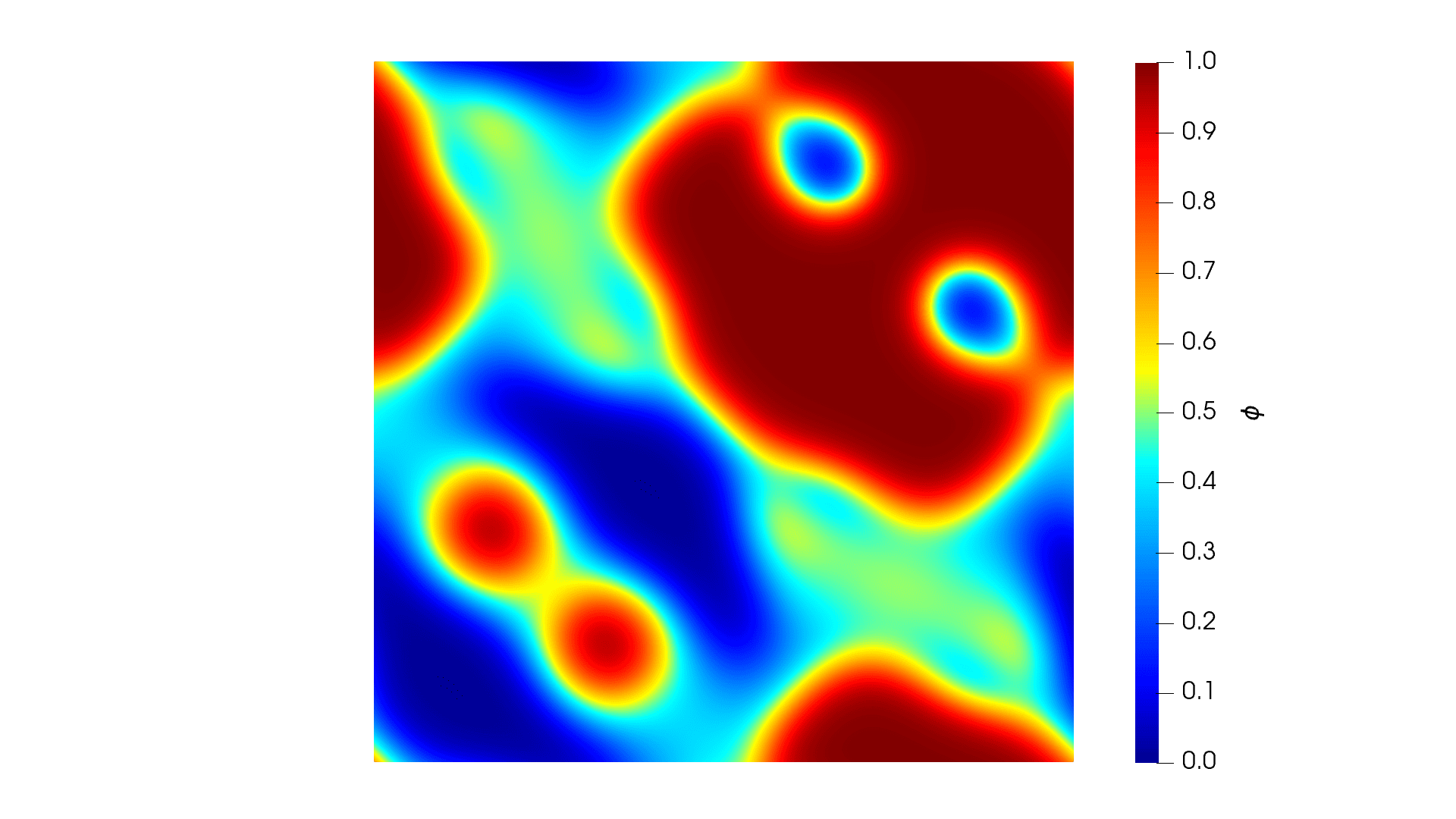} 
                &
                \includegraphics[trim={17cm 0cm 17cm 0cm},clip,scale=0.09]{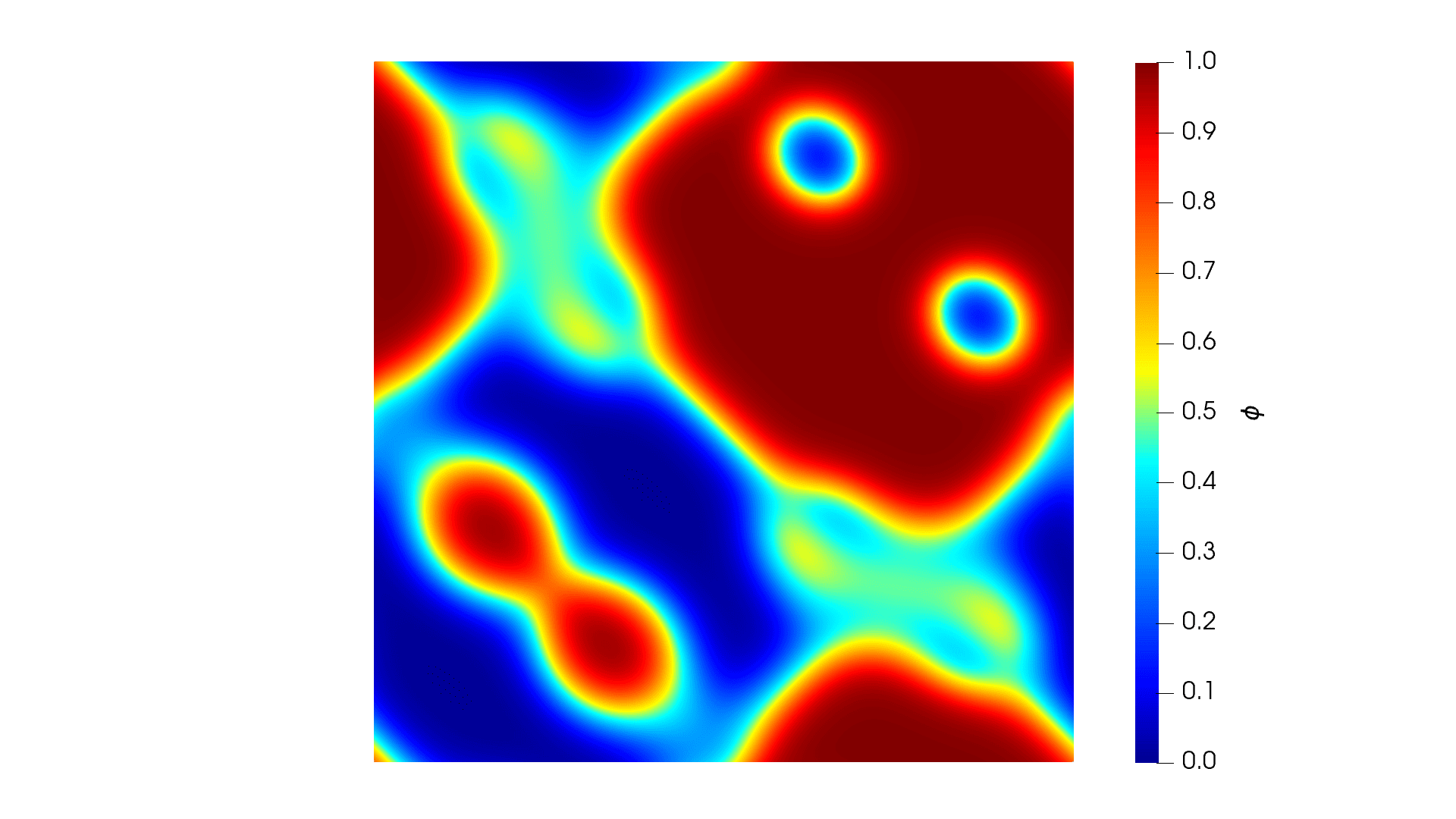}
                & 
                \includegraphics[trim={52.5cm 0cm 8.5cm 0cm},clip,scale=0.09]{Bilder/Melt/Without_Flow/Phi_0.png}
                \\[-0.5em]
                \\
            \end{tabular}
            \caption{Snapshots of the temporal evolution of $\phi_h$ together with $\u$ (top) compared with snapshots of $\phi_h$ from the simulation without consideration of velocity (bottom).}
            \label{fig:ex_melt_phi}
        \end{figure}
        
        Snapshots of the simulation can be seen in \cref{fig:ex_melt_phi,fig:ex_melt_theta} showing the phase variable $\phi_h$ together with the velocity field $\u_h$ as well as the temperature $\vtheta_h$ compared with snapshots of the same simulation without consideration of a velocity field respectively. For the simulation with velocity (top), we see that because of the initial temperature, the upper right corner starts melting on the diagonal, while staying solid on the regions with low initial temperature, leaving two solid grains in the otherwise molten region. Same happens in the opposite corner of the domain, creating two off diagonal melt pools in an otherwise solid region. Similar behavior can be seen without consideration of velocity (bottom), but as time evolves, for the simulation with velocity, because of the different viscosities of the liquid and solid phase, the grains in the upper right corner flow to the sides of the surrounding melt pool introducing some melt flow, while in the lower left corner the melt pools combine in the middle to one bigger melt pool. This shows the influence of the Navier-Stokes equation, since in the simulation without a velocity field, we do not see the flow effects for the solid grains in the liquid melt pool and also the two melt pools in the solid region do not combine because of lacking flow.
        Because of that, we also see some differences in the temperature distribution. Since the colder solid grains in the top right corner leave the melt pool under consideration of velocity, they dissipate away the associated cold spots, while remaining stationary otherwise. Also for the lower right corner the center of the melt pool is more concentrated showing again a corresponding difference in temperature. Note that the temperature in the beginning starts decreasing in the melting regions and rises for the solidifying parts. This is due to the latent heat storing additional energy in the melt phase, therefore releasing energy in form of temperature while solidifying and vice versa.
        \begin{figure}[htbp!]
            \centering
            \hspace*{-2em}
            \begin{tabular}{c@{}c@{}c@{}c@{}c@{}c@{}}
                & $t=0.5$ & $t=2$ & $t=3$ & $t=5$ & 
                \\[-0.5em]
                \includegraphics[trim={8.5cm 0cm 53cm 0cm},clip,scale=0.09]{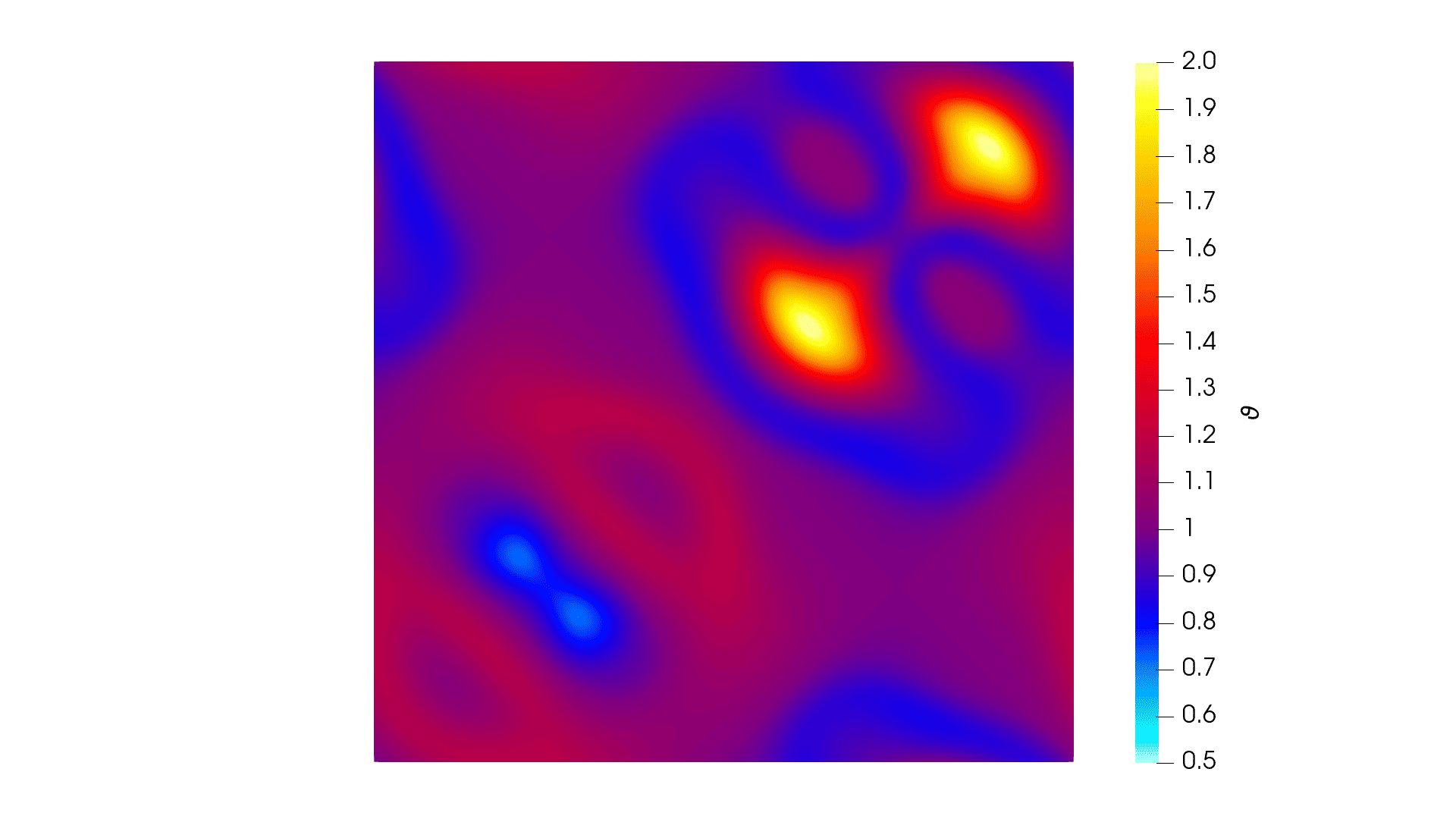}
                &
                \includegraphics[trim={17cm 0cm 17cm 0cm},clip,scale=0.09]{Bilder/Melt/With_Flow/Theta_50.png} 
                &
                \includegraphics[trim={17cm 0cm 17cm 0cm},clip,scale=0.09]{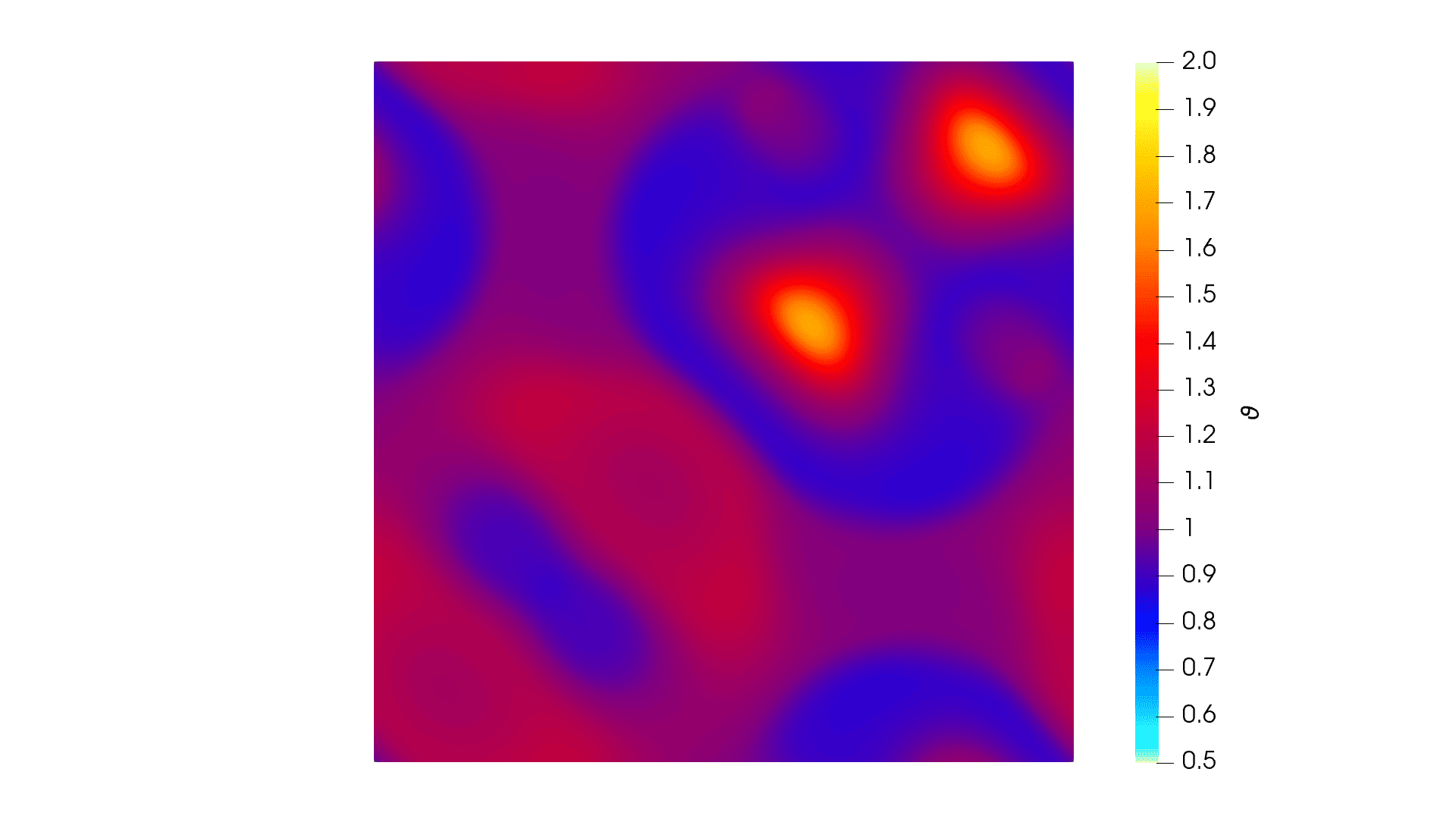}
                &
                \includegraphics[trim={17cm 0cm 17cm 0cm},clip,scale=0.09]{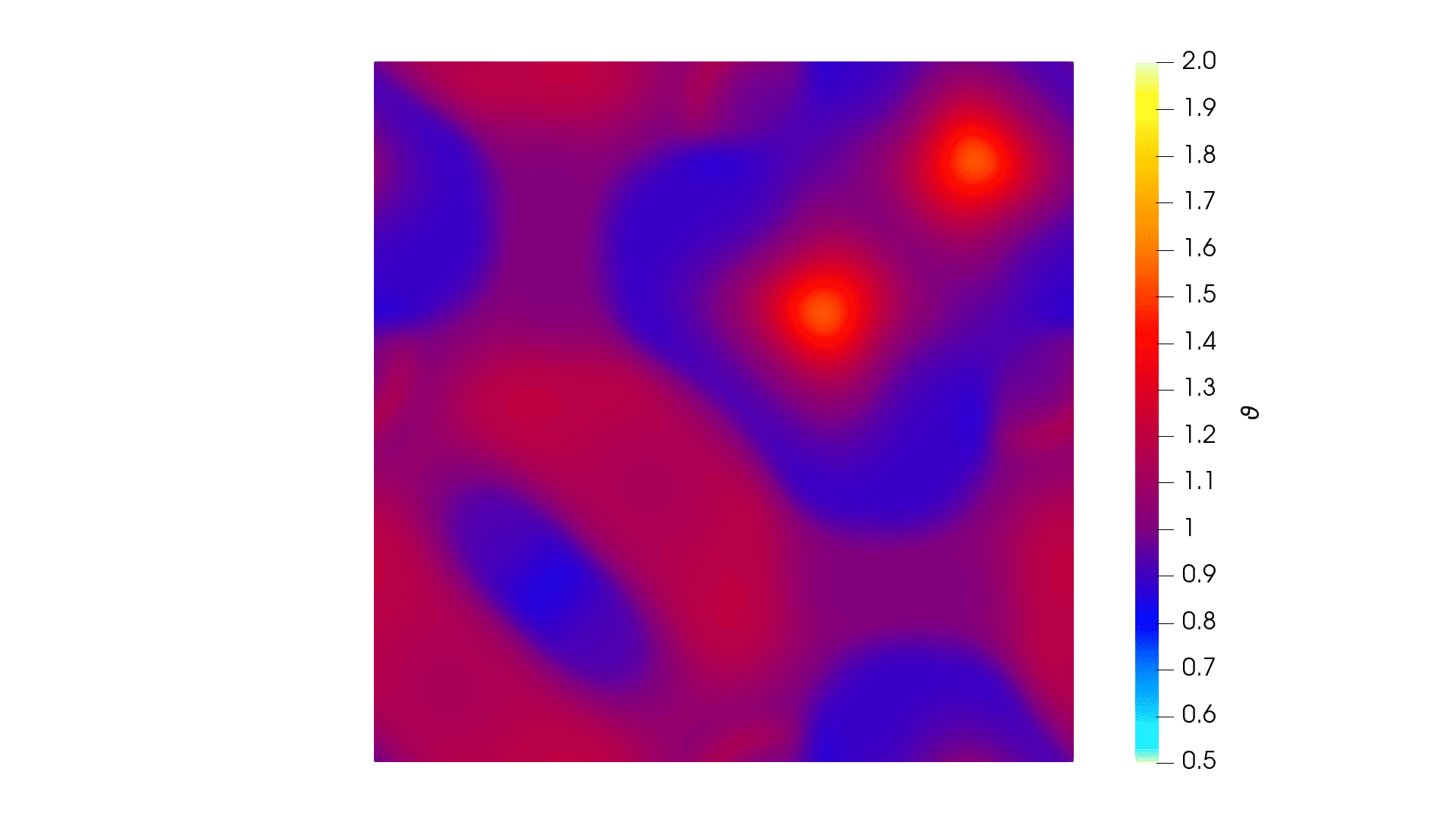} 
                &
                \includegraphics[trim={17cm 0cm 17cm 0cm},clip,scale=0.09]{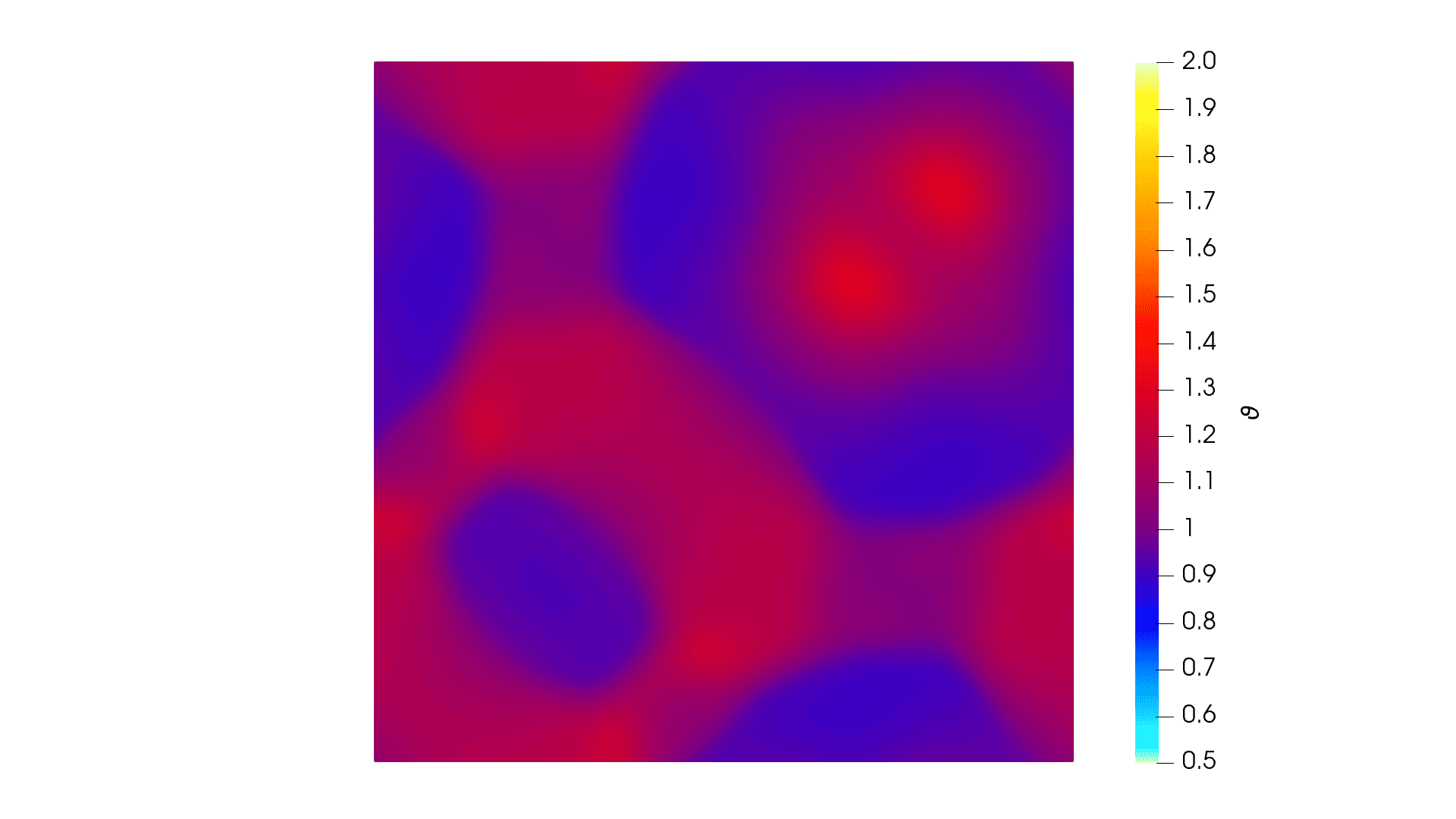}
                & 
                \includegraphics[trim={52.5cm 0cm 8.5cm 0cm},clip,scale=0.09]{Bilder/Melt/With_Flow/Theta_50.png}\\[-1em]
                \includegraphics[trim={8.5cm 0cm 53cm 0cm},clip,scale=0.09]{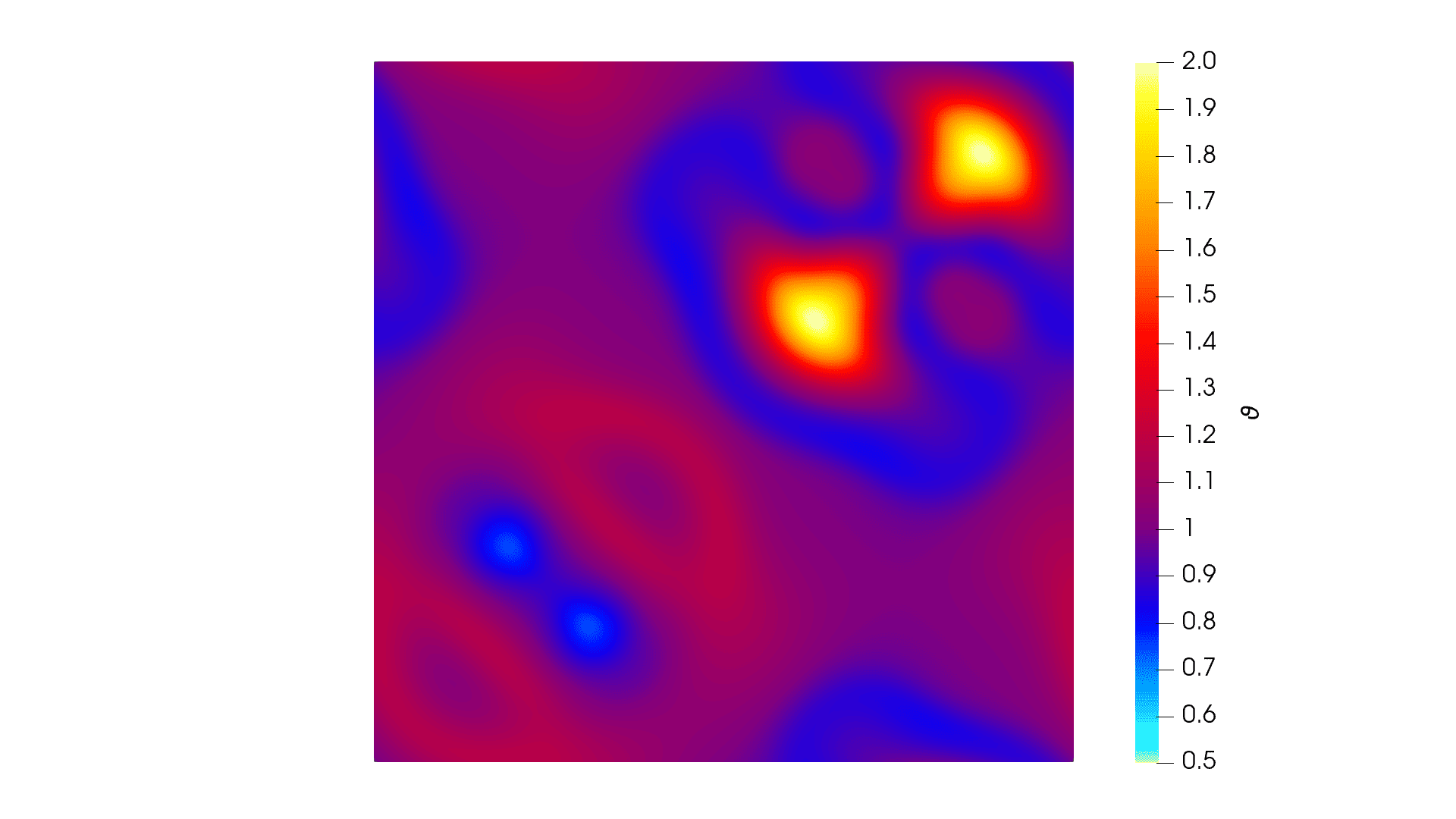}
                &
                \includegraphics[trim={17cm 0cm 17cm 0cm},clip,scale=0.09]{Bilder/Melt/Without_Flow/Theta_50.png} 
                &
                \includegraphics[trim={17cm 0cm 17cm 0cm},clip,scale=0.09]{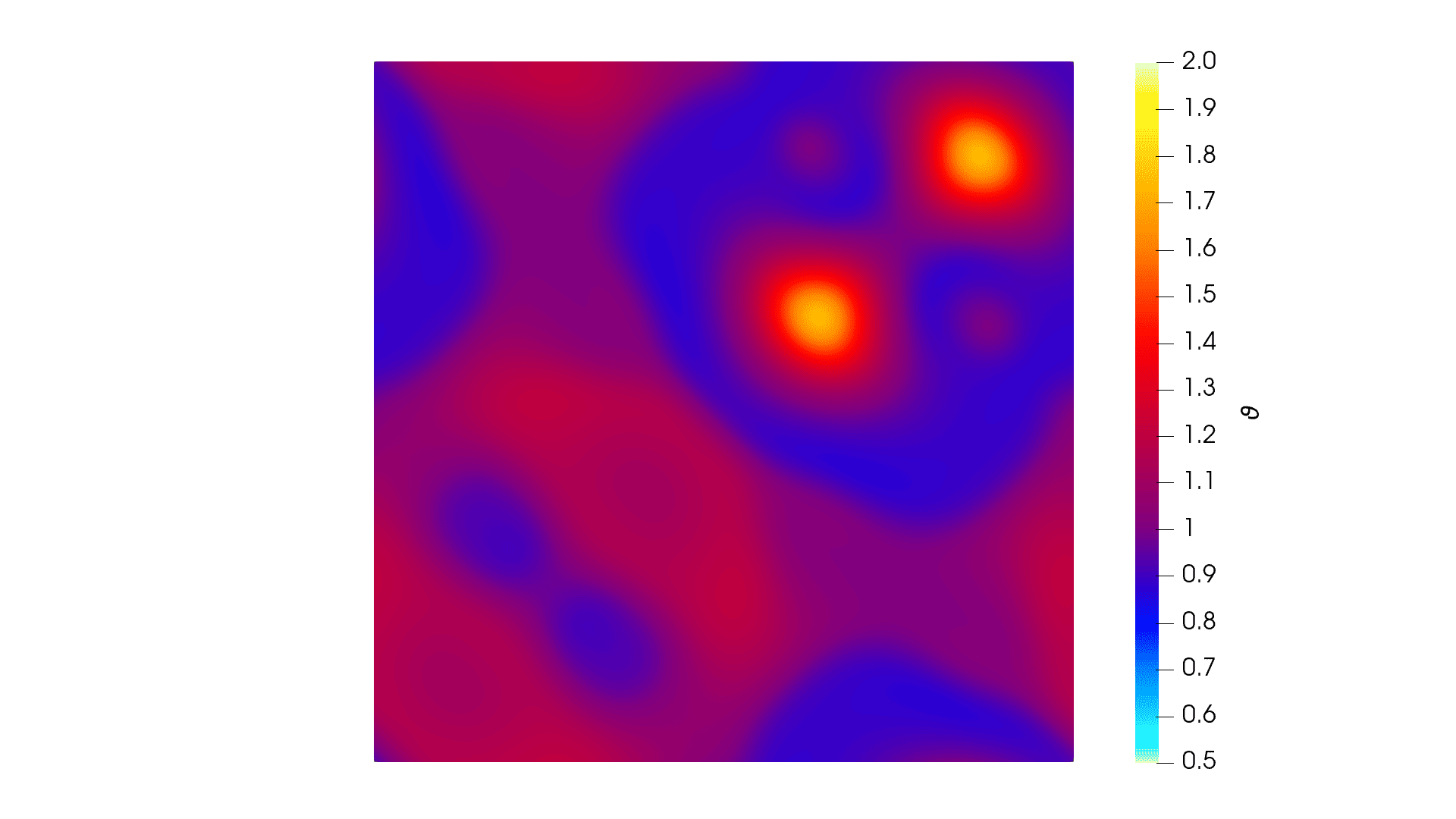}
                &
                \includegraphics[trim={17cm 0cm 17cm 0cm},clip,scale=0.09]{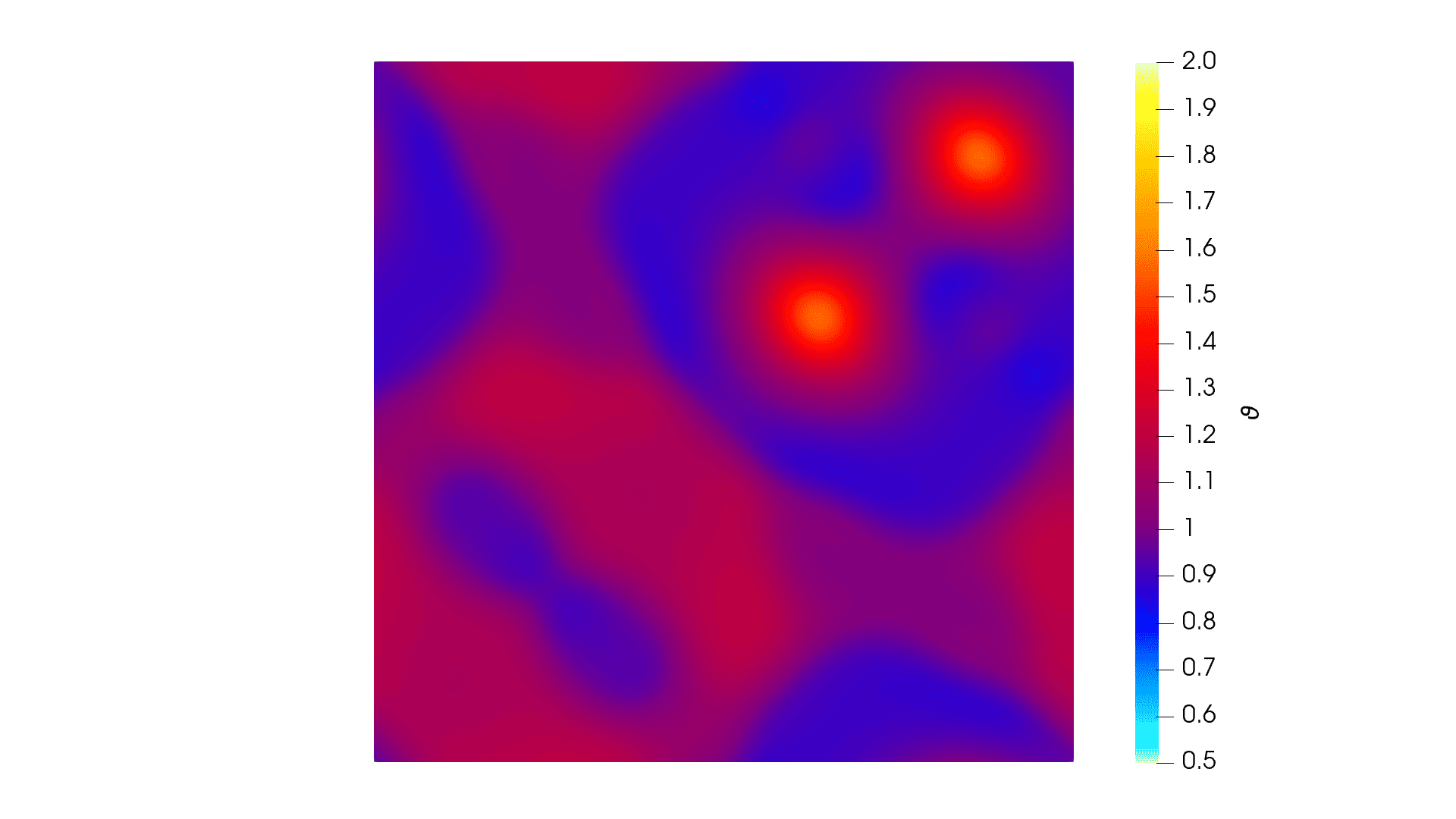} 
                &
                \includegraphics[trim={17cm 0cm 17cm 0cm},clip,scale=0.09]{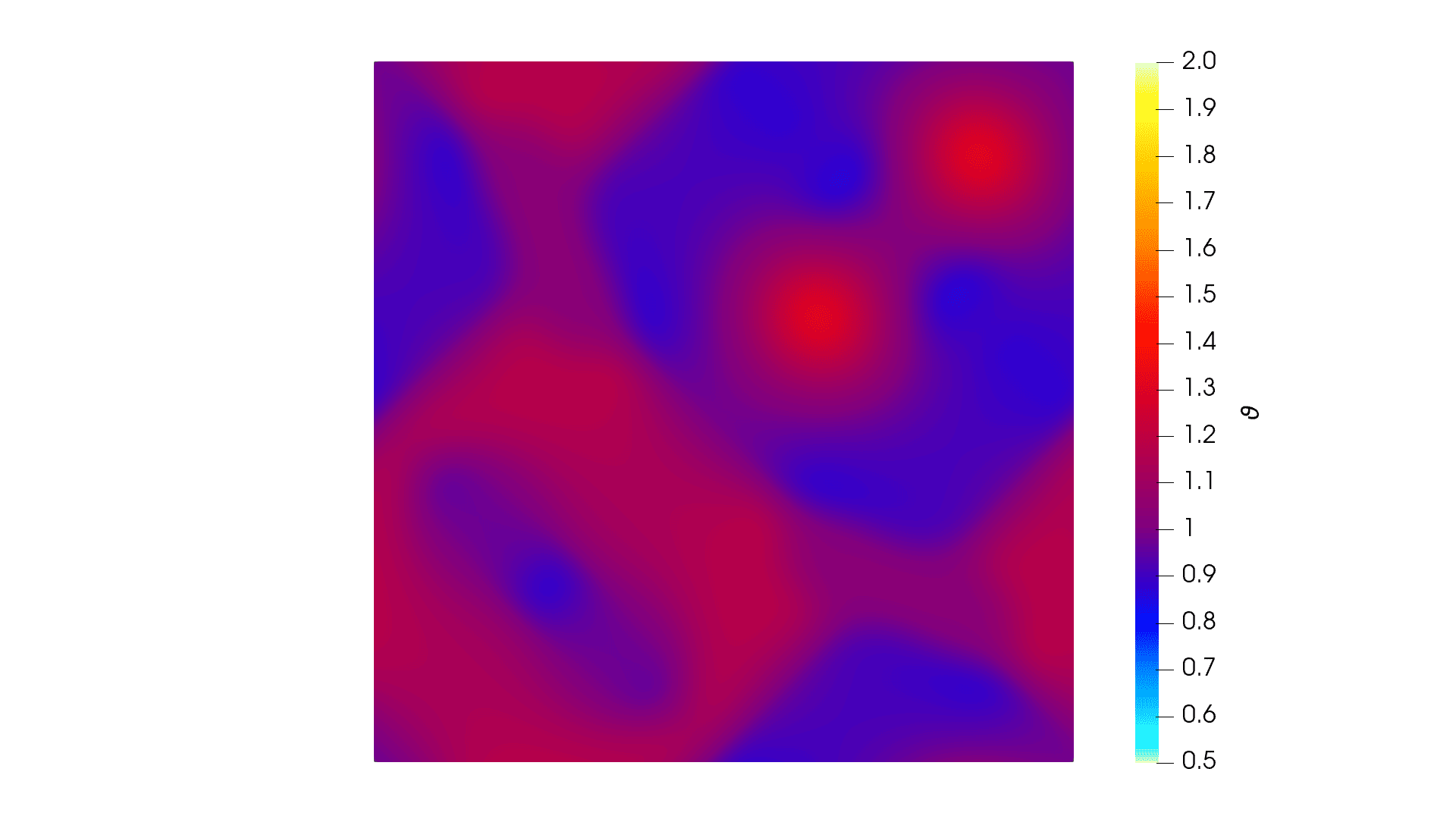}
                & 
                \includegraphics[trim={52.5cm 0cm 8.5cm 0cm},clip,scale=0.09]{Bilder/Melt/Without_Flow/Theta_50.png}
                \\[-0.5em]
                \\
            \end{tabular}
            \caption{Snapshots of the temperature $\vtheta_h$ with (top) and without (bottom) consideration of velocity.}
            \label{fig:ex_melt_theta}
        \end{figure}
        In \cref{fig:ex_melt_struct} the temporal evolution of the integral over the entropy density $s_h$ and the total energy density $e_{\mathrm{tot},h}$, for both simulations with and without consideration of velocity, are shown respectively. Due to periodic boundary conditions and lack of external forces and sources, the system is isolated showing a small numerical decrease for the total energy and an increasing entropy as predicted by \cref{thm:discstruc}. This is true for both simulations, since we used the same structure preserving algorithm in both cases. At approximately half the simulation time a notable difference between the simulation with and without consideration of velocity can be seen,  corresponding to the above mentioned flow effects.
        \begin{figure}[htbp!]
            \centering
            \begin{tabular}{c@{}c@{}}
                \includegraphics[width=0.47\linewidth]{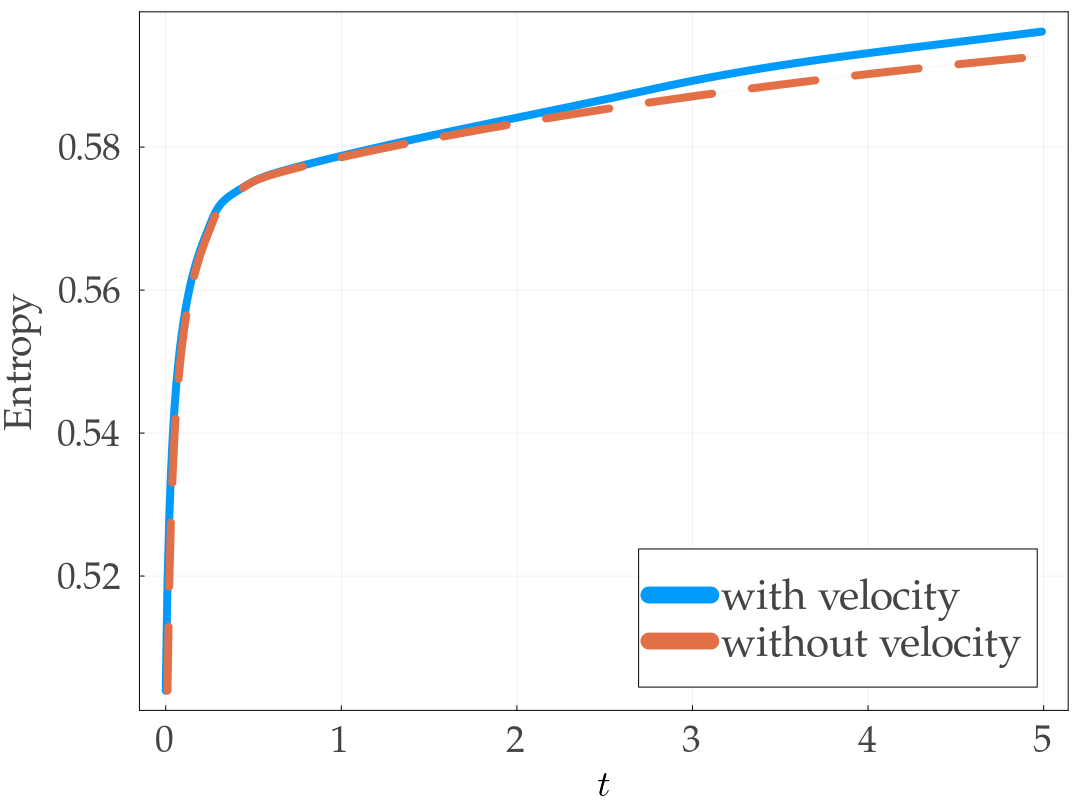}
                &
                \includegraphics[width=0.47\linewidth]{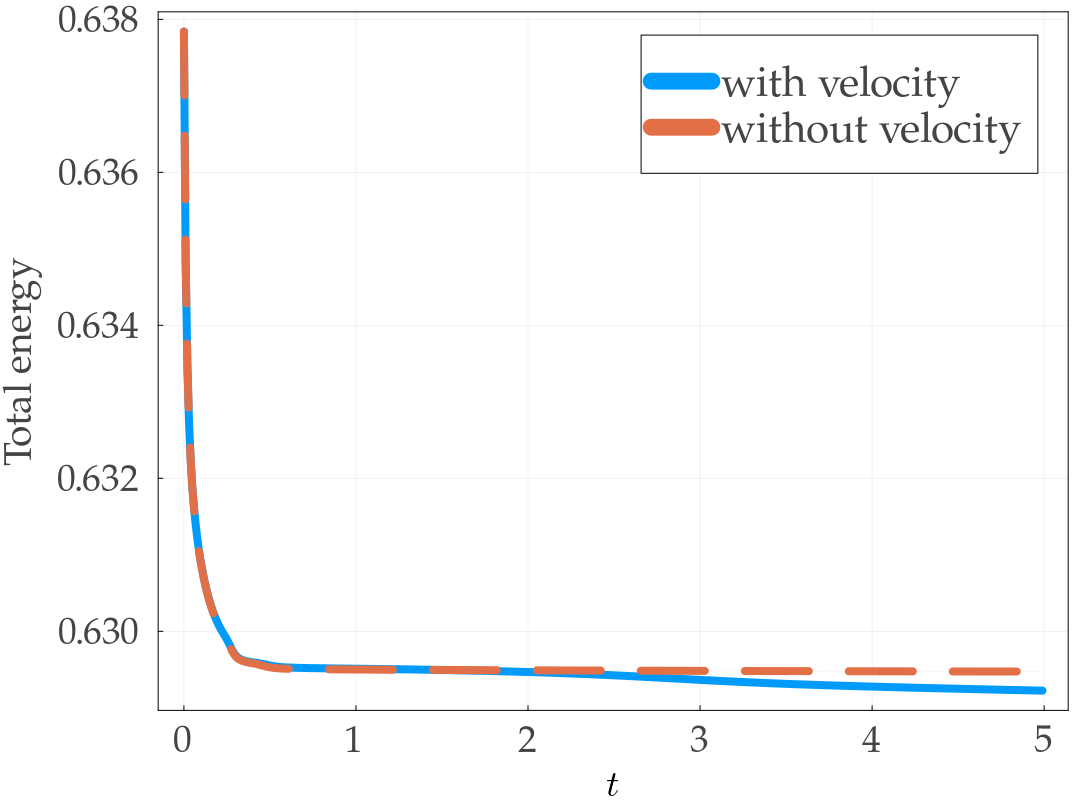}
            \end{tabular}
            \caption{Temporal evolution of the integrated entropy density $s_h$ (left) and integrated total energy density $e_{\mathrm{tot},h}$ (right)}
            \label{fig:ex_melt_struct}
        \end{figure}
        
    \subsection{Laser melt example}
    
        For the second example we simulate a Laser melting metal by introducing a heat source $Q$ of the form
        \begin{align*}
            Q = \begin{cases}
                200\cdot Q_\mathrm{lsr}(t)&,t\leq1\\
                Q_\mathrm{cool}&,1.5\leq t\leq2\\
                0&,\text{else}
            \end{cases}
        \end{align*}
        with
        \begin{align*}
            Q_\mathrm{lsr}(t)&=\operatorname{trans}_{[0,1]}\left(\frac{\sqrt{2.5(x-\mathbf{p}_1(t))^2+2.5(y-\mathbf{p}_2(t))^2}-0.2)}{0.1}\right), \\
            \mathbf{p}(t) &= 5\cdot(0.3\cos(2\pi t)+0.5,0.3\sin(2\pi t)+0.5)^\top,
        \end{align*}
        being a circle of diameter $r\approx0.25$ with a peak heat in the center of $200$ moving on a circular path $\mathbf{p}$ over the domain $\Omega=[0,5]^2$. After some time we turn of the laser and later on we cool the whole domain via $Q_\mathrm{cool}=-1$ so that the melt path re-solidifies again. Afterwards to external source is turned off. For the gradient contribution we use again $G(\nabla\phi)=\tfrac{\gamma^2}{2}\snorm{\nabla\phi}^2$, while the remaining parameters are set as\\
        \begin{minipage}[t]{0.32\textwidth}
            \vspace{0.1cm}
            \begin{itemize}
                \item $\M=100$
                \item $\K=2$
                \item $\tau=0.001$
                \item $\b=\mathbf{0}$
            \end{itemize}
            \vspace{0.2cm}
        \end{minipage}
        \begin{minipage}[t]{0.32\textwidth}
            \vspace{0.1cm}
            \begin{itemize}
                \item $\gamma=0.05$
                \item $\eta_s=1$
                \item $\eta_l=0.001$
                \item $\vtheta_m=1$
            \end{itemize}
            \vspace{0.2cm}
        \end{minipage}
        \begin{minipage}[t]{0.32\textwidth}
            \vspace{0.1cm}
            \begin{itemize}
                \item $\mathcal{L}=1$
                \item $H_\mathrm{pt}=1$
                \item $H_\mathrm{cf}=0.1$
                \item $C_\mathrm{vsh}=1$
            \end{itemize}
            \vspace{0.2cm}
        \end{minipage}
        As for the initial values we consider a preheated solid with an already molten starting point for the laser and no flow, given by
        \begin{align*}
            \phi_0(x,y)&:=\operatorname{trans}_{[0,1]}\left(\frac{\sqrt{5(x-4)^2+5(y-2.5)^2}-0.2)}{\sqrt{0.05}}\right),\\\vtheta_0(x,y)&:=\vtheta_m(0.95+0.55\phi_0),\\
            \u(x,y)&:=(0,0)^\top.
        \end{align*}
        as depicted in \cref{fig:ex_laser_init}. We again only consider \cref{bc:periodic}, simulating a infinite solid metal surface.
        \begin{figure}[htbp!]
            \centering
            \includegraphics[trim={16cm 0cm 9cm 0cm},clip,scale=0.14]{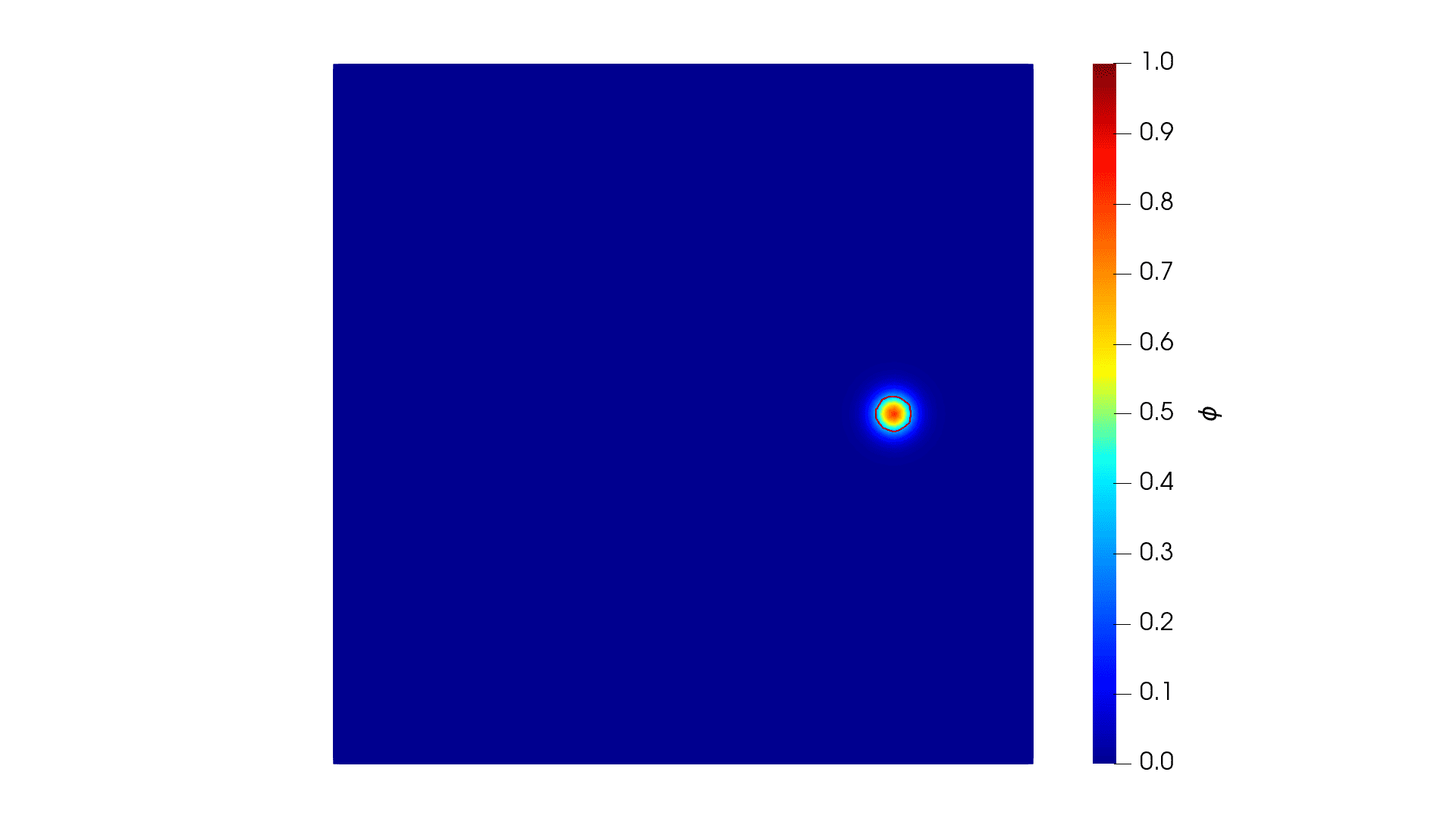}
            \includegraphics[trim={16cm 0cm 9cm 0cm},clip,scale=0.14]{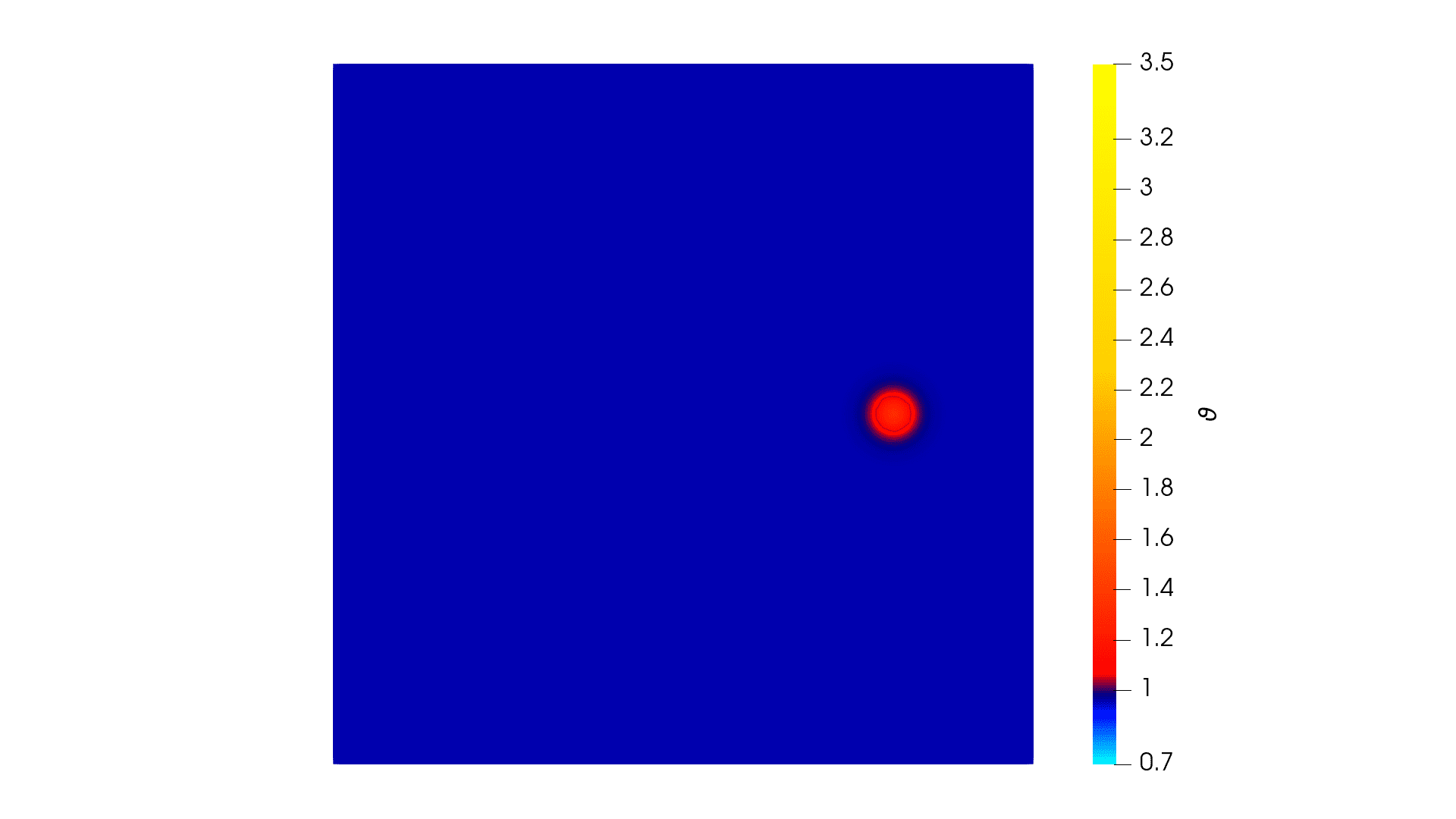} 
            \caption{Initial condition for second example of the phase variable $\phi$ (left) and the temperature $\vtheta$ (right).}
            \label{fig:ex_laser_init}
        \end{figure}
        
        The snapshots of the simulation can be found in \cref{fig:ex_laser_phi} and \cref{fig:ex_laser_theta}, showing the phase variable $\phi_h$ together with the melt flow velocity $\u_h$ as well as the temperature $\vtheta_h$, respectively. Also the contour of the external heat source $Q_\mathrm{lsr}$, representing the laser, is outlined in both figures in red.
        \begin{figure}[htbp!]
            \centering
            \hspace*{-1em}
            \begin{tabular}{c@{}c@{}c@{}c@{}c@{}}
                \multirow{4}{*}{\includegraphics[trim={6cm 0cm 54cm 0cm},clip,scale=0.22]{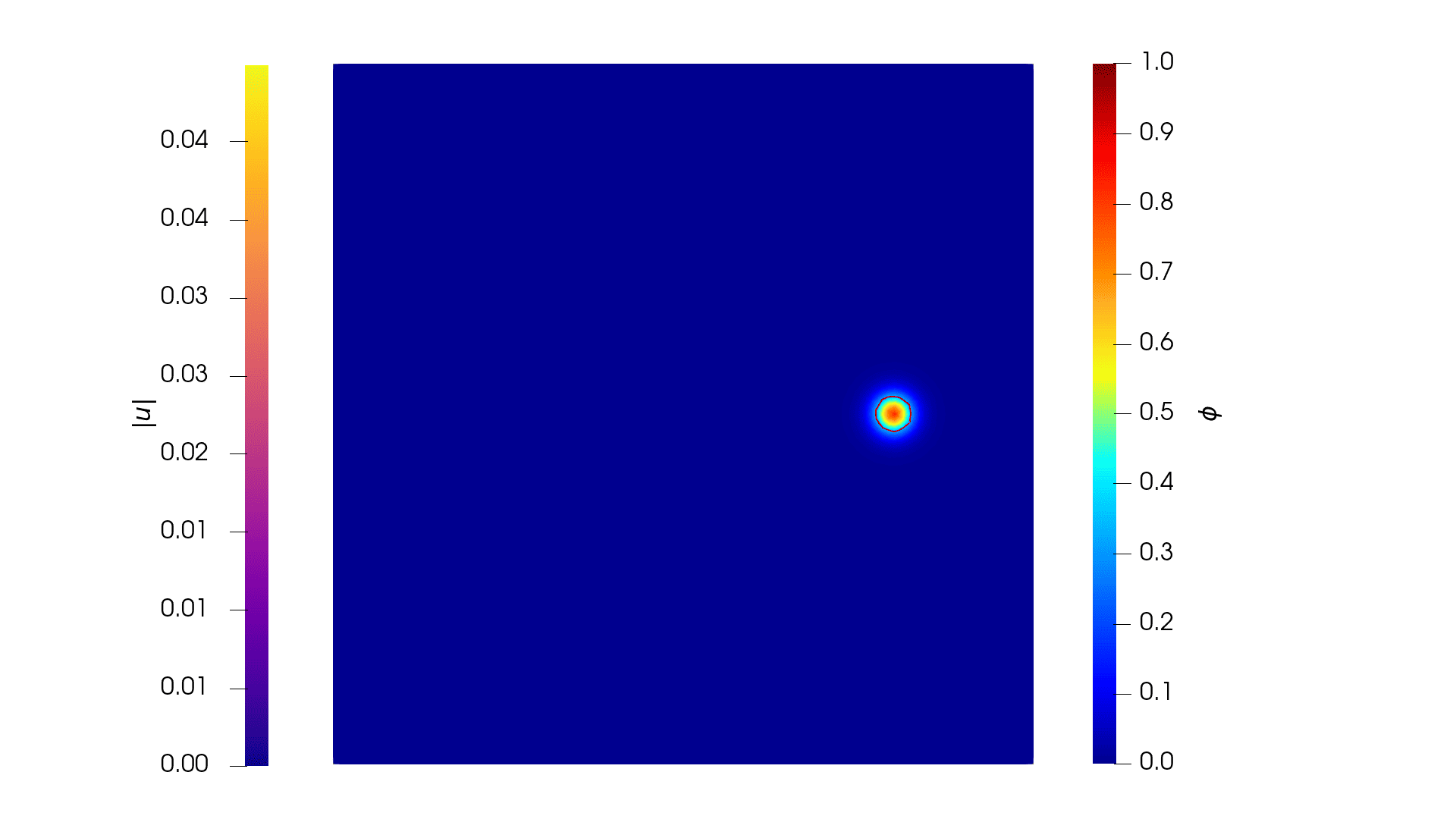}}
                & $t=0.1$ & $t=0.5$ & $t=1$ & 
                \multirow{4}{*}{\includegraphics[trim={51cm 0cm 9cm 0cm},clip,scale=0.22]{Bilder/Laser/Phi_u_0.png}}
                \\[-0.5em]
                &
                \includegraphics[trim={17cm 0cm 17cm 0cm},clip,scale=0.1]{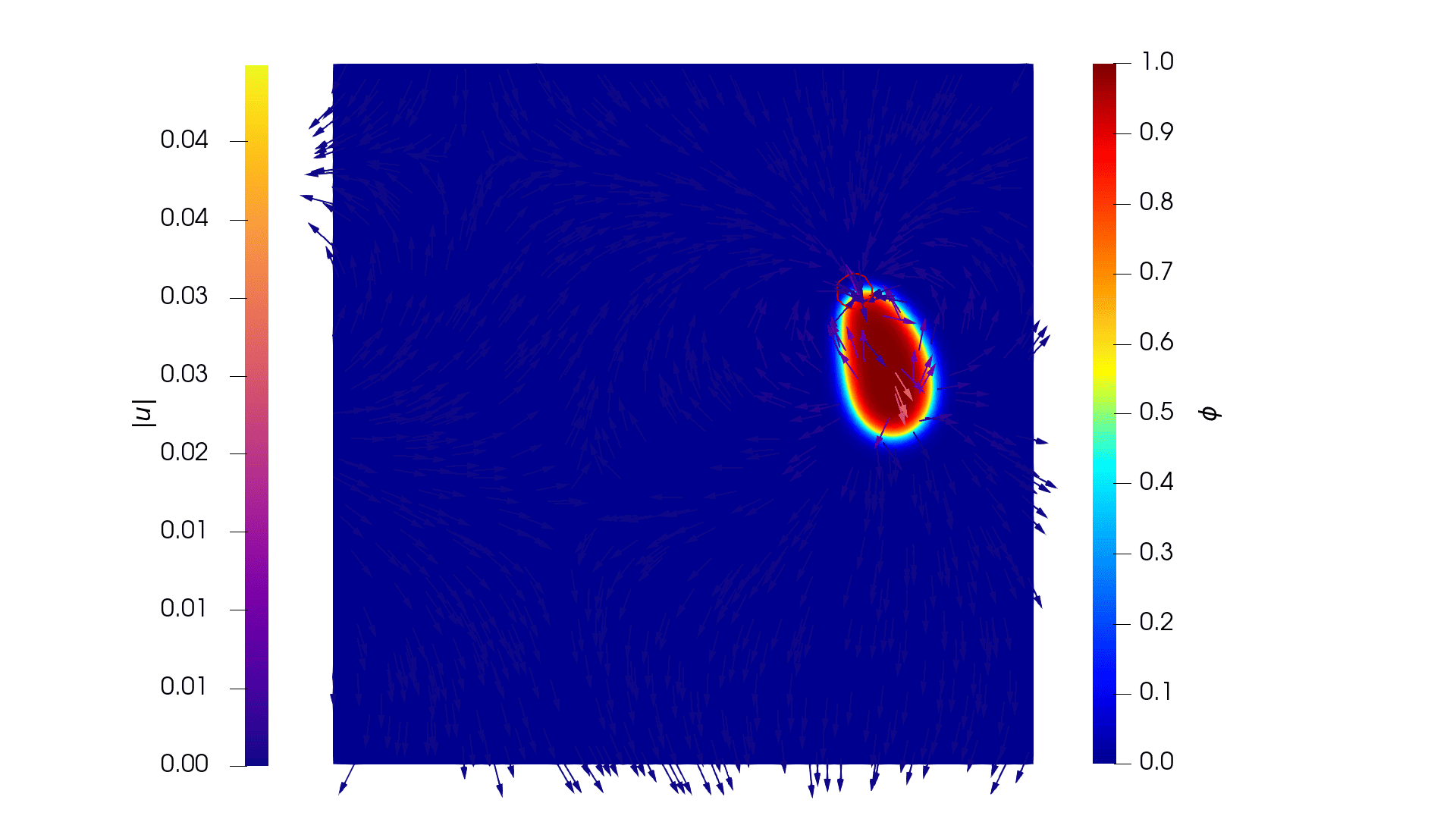} 
                &
                \includegraphics[trim={17cm 0cm 17cm 0cm},clip,scale=0.1]{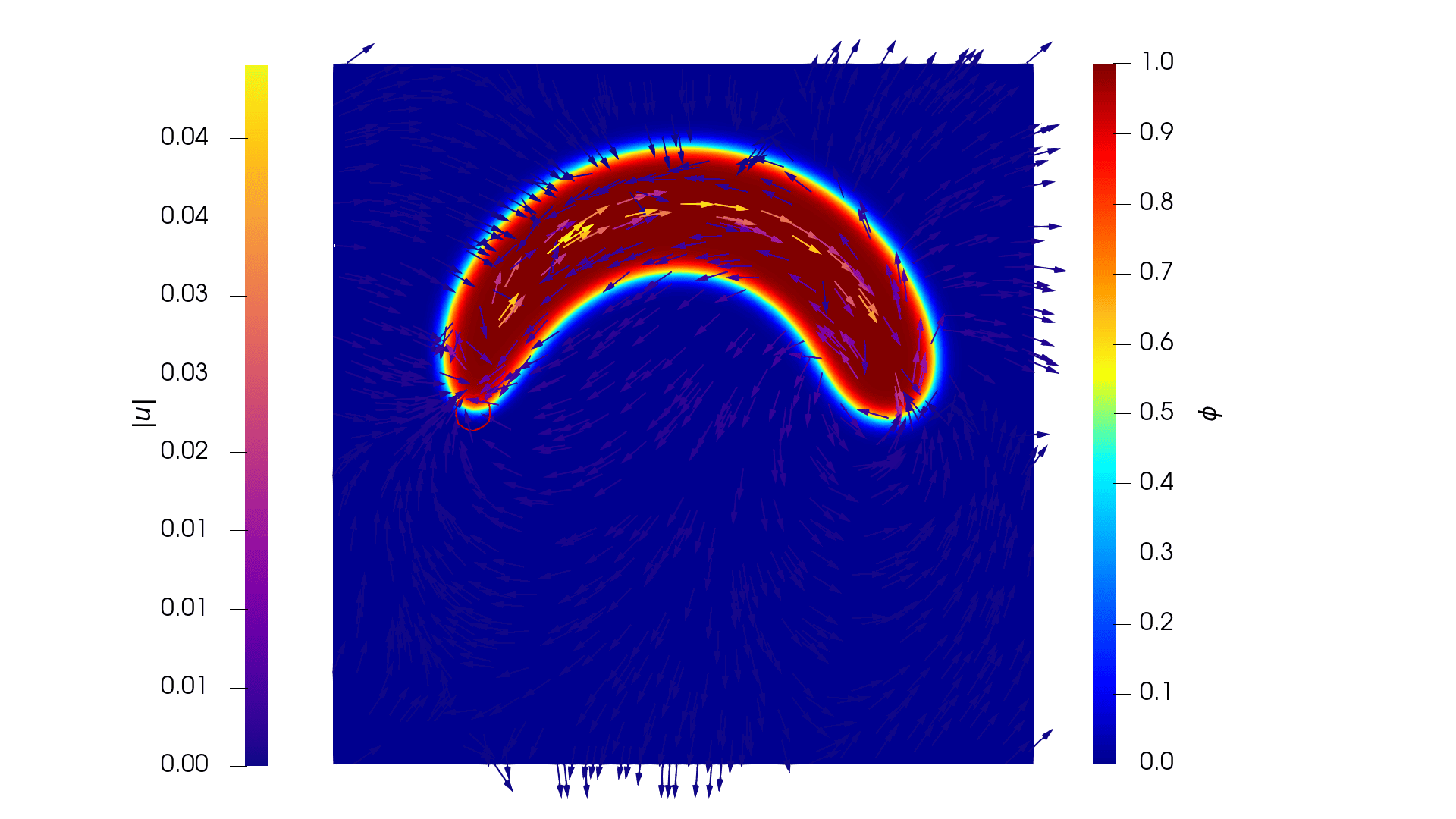}
                &
                \includegraphics[trim={17cm 0cm 17cm 0cm},clip,scale=0.1]{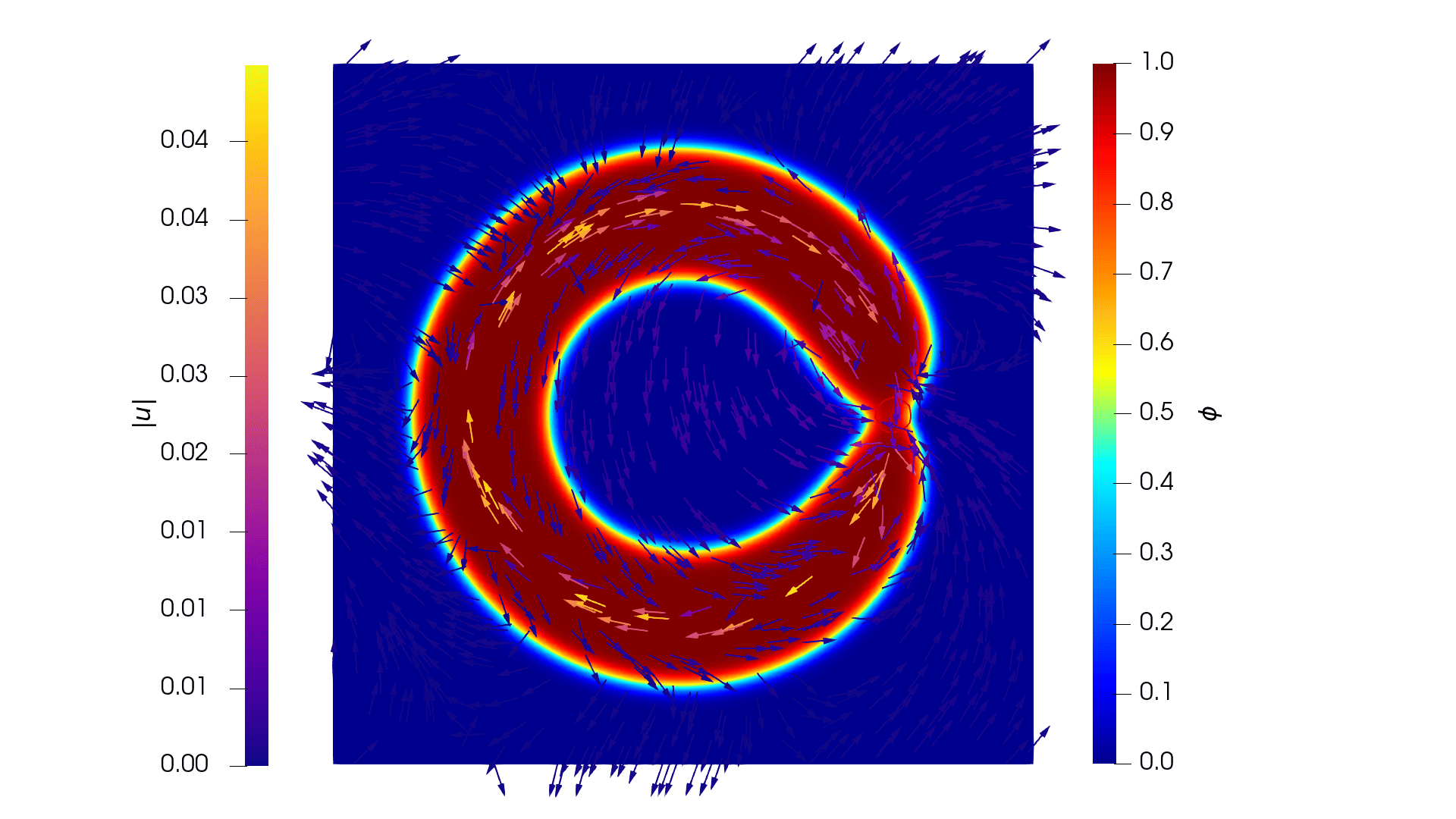}
                & 
                \\[-0.5em]& $t=1.5$ & $t=2$ & $t=2.25$ &\\[-0.5em]
                &
                \includegraphics[trim={17cm 0cm 17cm 0cm},clip,scale=0.1]{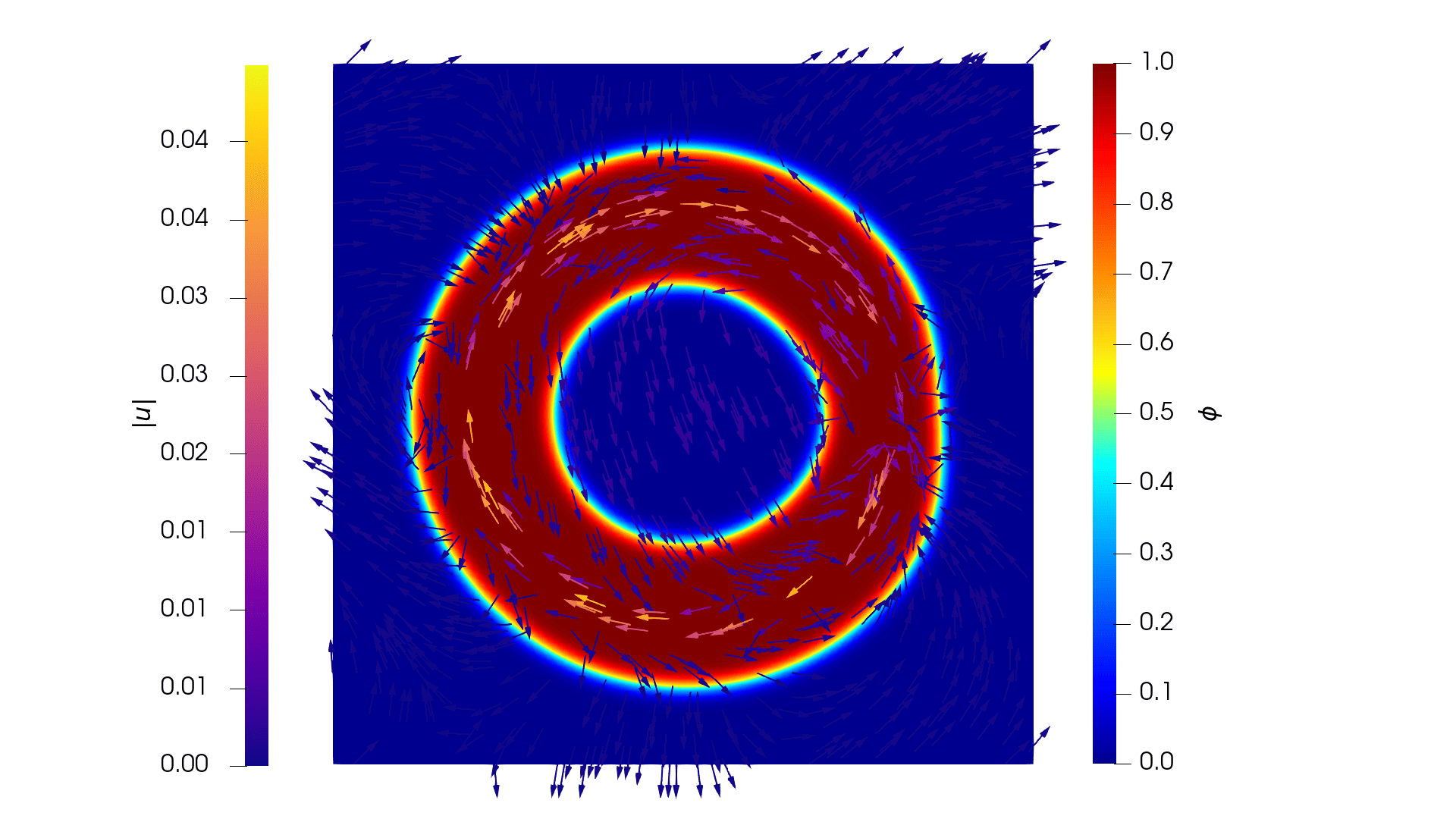} 
                &
                \includegraphics[trim={17cm 0cm 17cm 0cm},clip,scale=0.1]{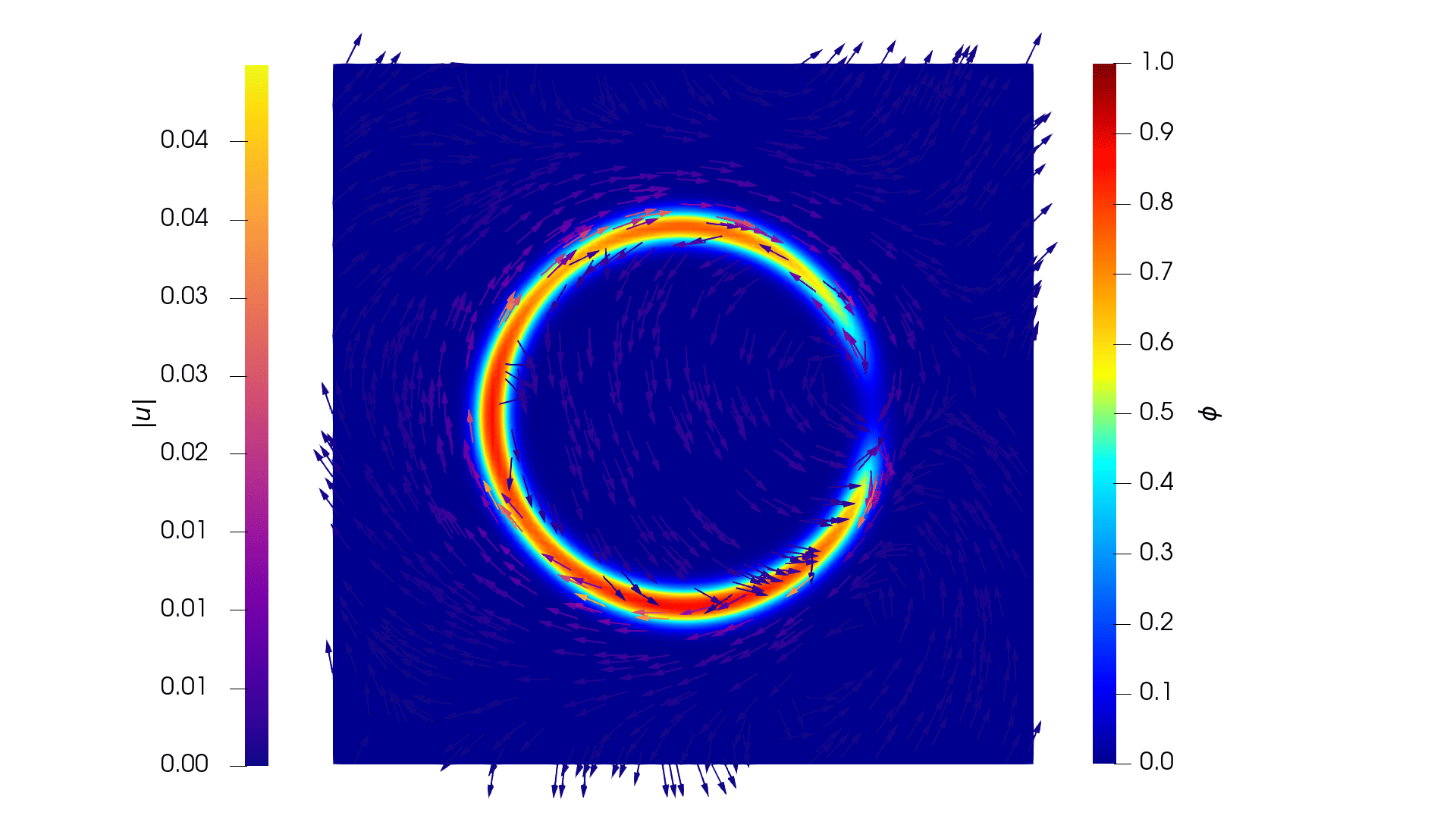}
                &
                \includegraphics[trim={17cm 0cm 17cm 0cm},clip,scale=0.1]{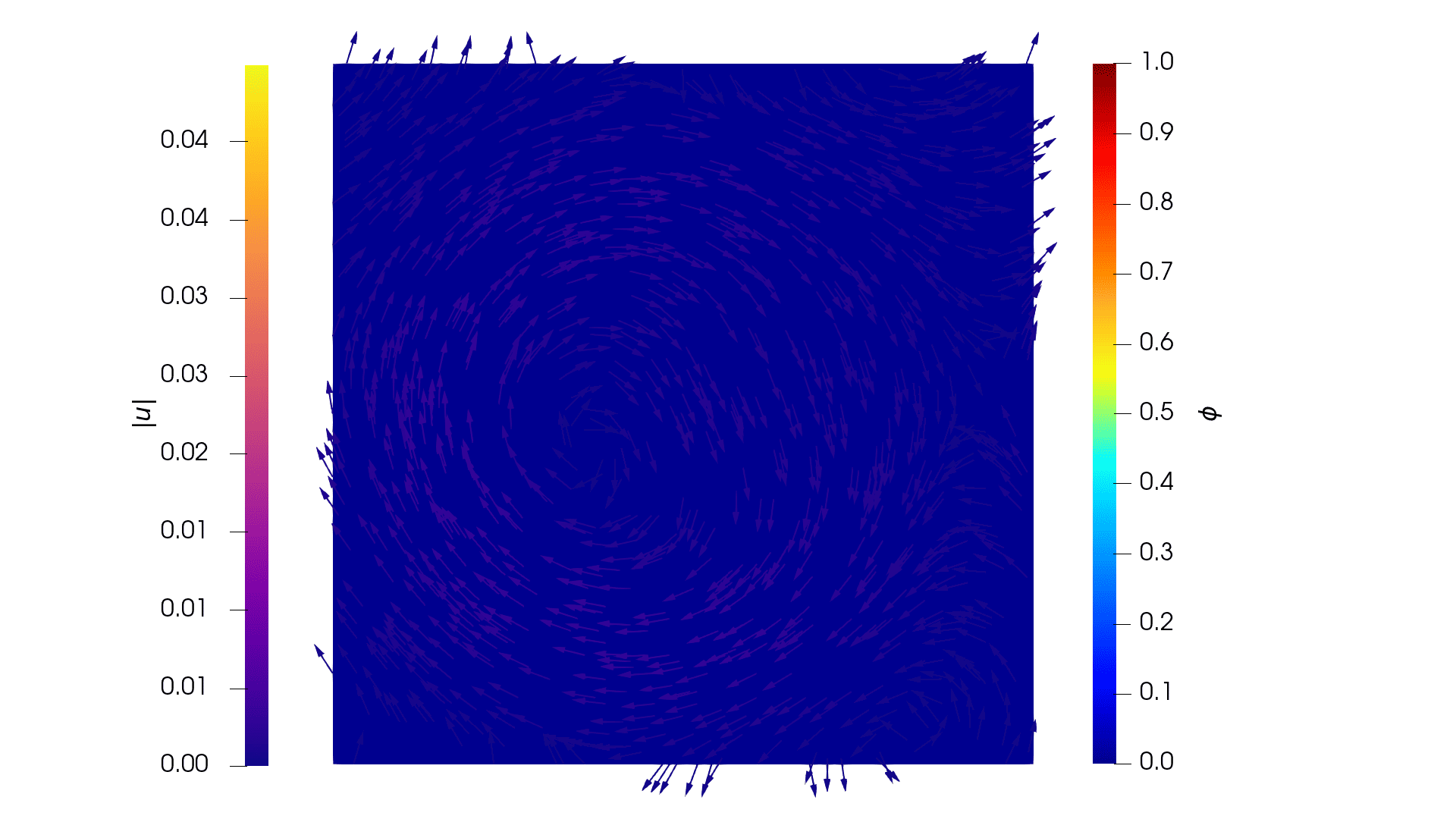}
                & 
                \\[-0.5em]
            \end{tabular}
            \caption{Snapshots of the phase variable $\phi_h$ together with the melt flow velocity $\u_h$ with the red circle outlining the heat source $Q_l$.}
            \label{fig:ex_laser_phi}
        \end{figure}
        \begin{figure}[htbp!]
            \centering
            \begin{tabular}{c@{}c@{}c@{}c@{}c@{}}
                $t=0.1$ & $t=0.5$ & $t=1$ & 
                \multirow{4}{*}{\includegraphics[trim={51cm 0cm 9cm 0cm},clip,scale=0.24]{Bilder/Laser/Theta_0.png}}
                \\[-0.5em]
                \includegraphics[trim={17cm 0cm 17cm 0cm},clip,scale=0.11]{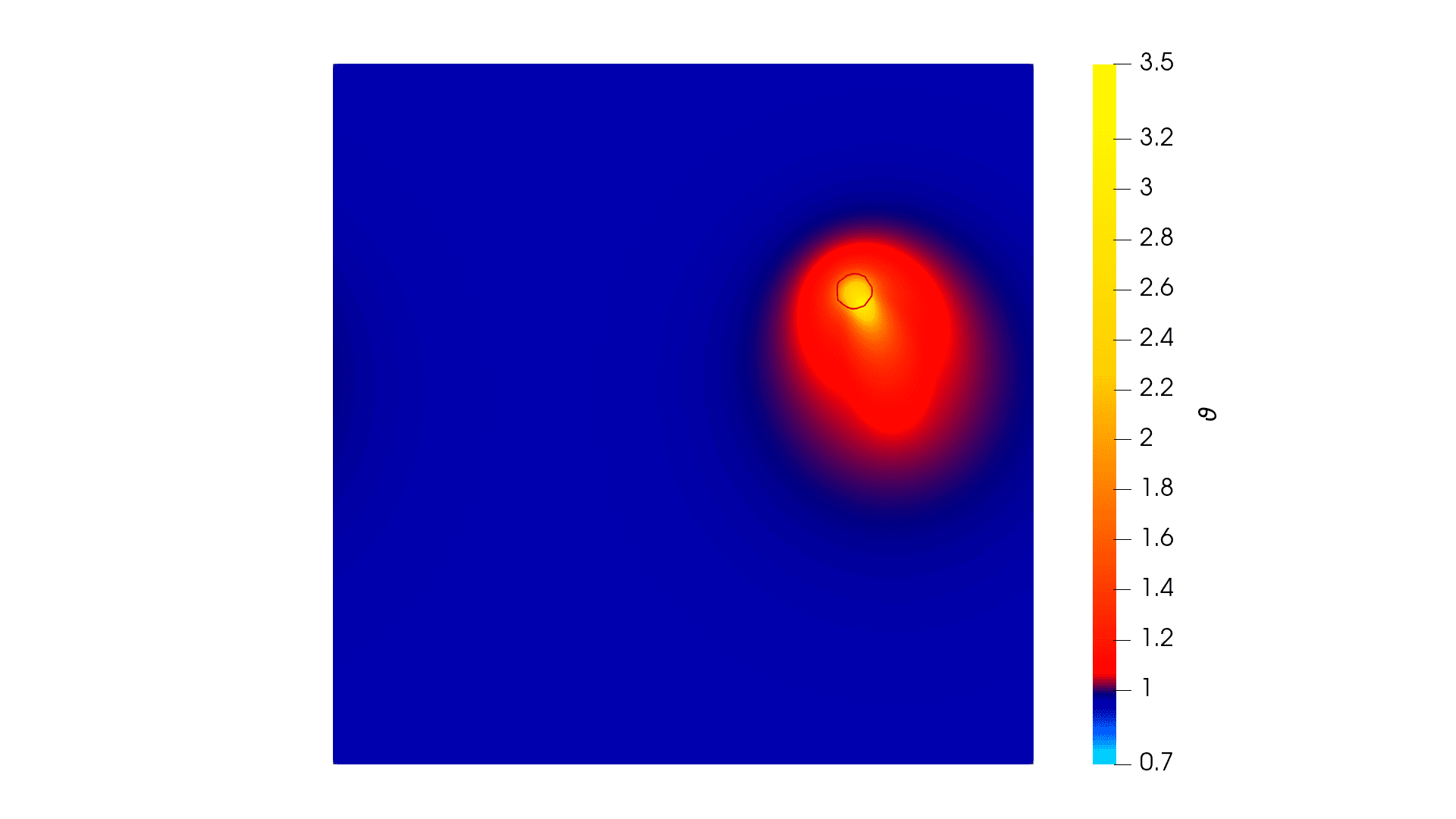} 
                &
                \includegraphics[trim={17cm 0cm 17cm 0cm},clip,scale=0.11]{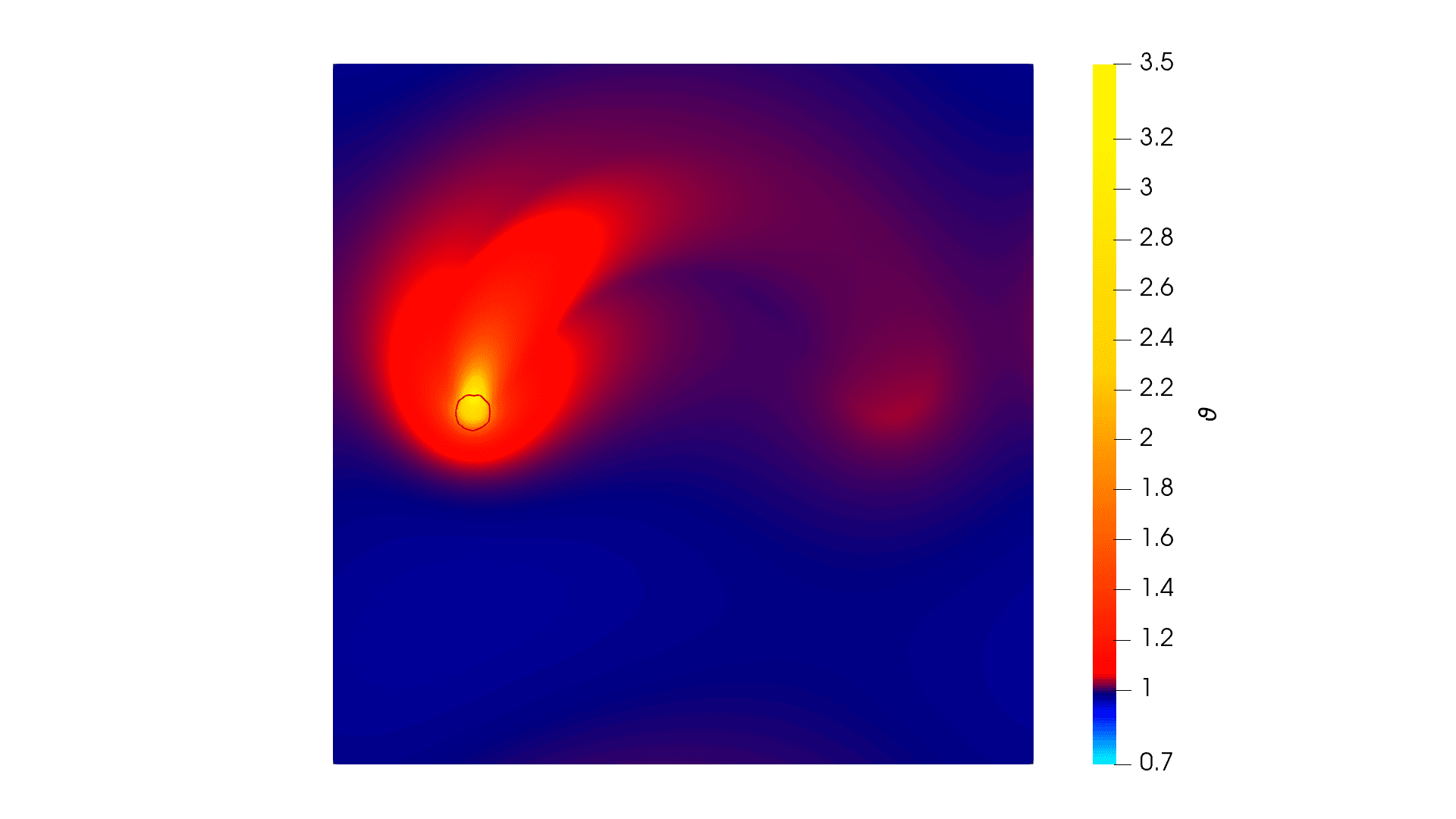}
                &
                \includegraphics[trim={17cm 0cm 17cm 0cm},clip,scale=0.11]{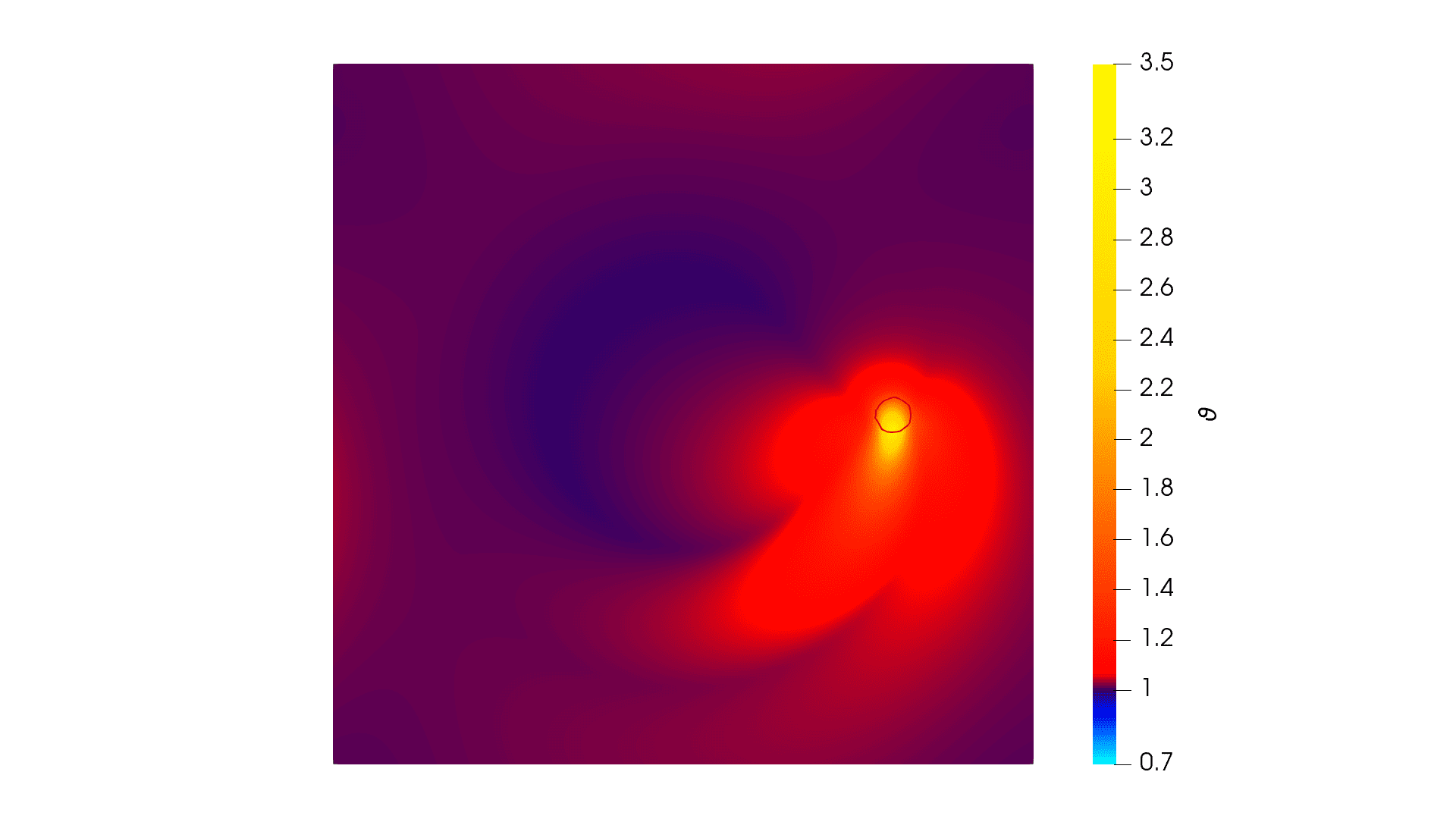}
                & 
                \\[-0.5em] $t=1.5$ & $t=2$ & $t=2.25$ &\\[-0.5em]
                \includegraphics[trim={17cm 0cm 17cm 0cm},clip,scale=0.11]{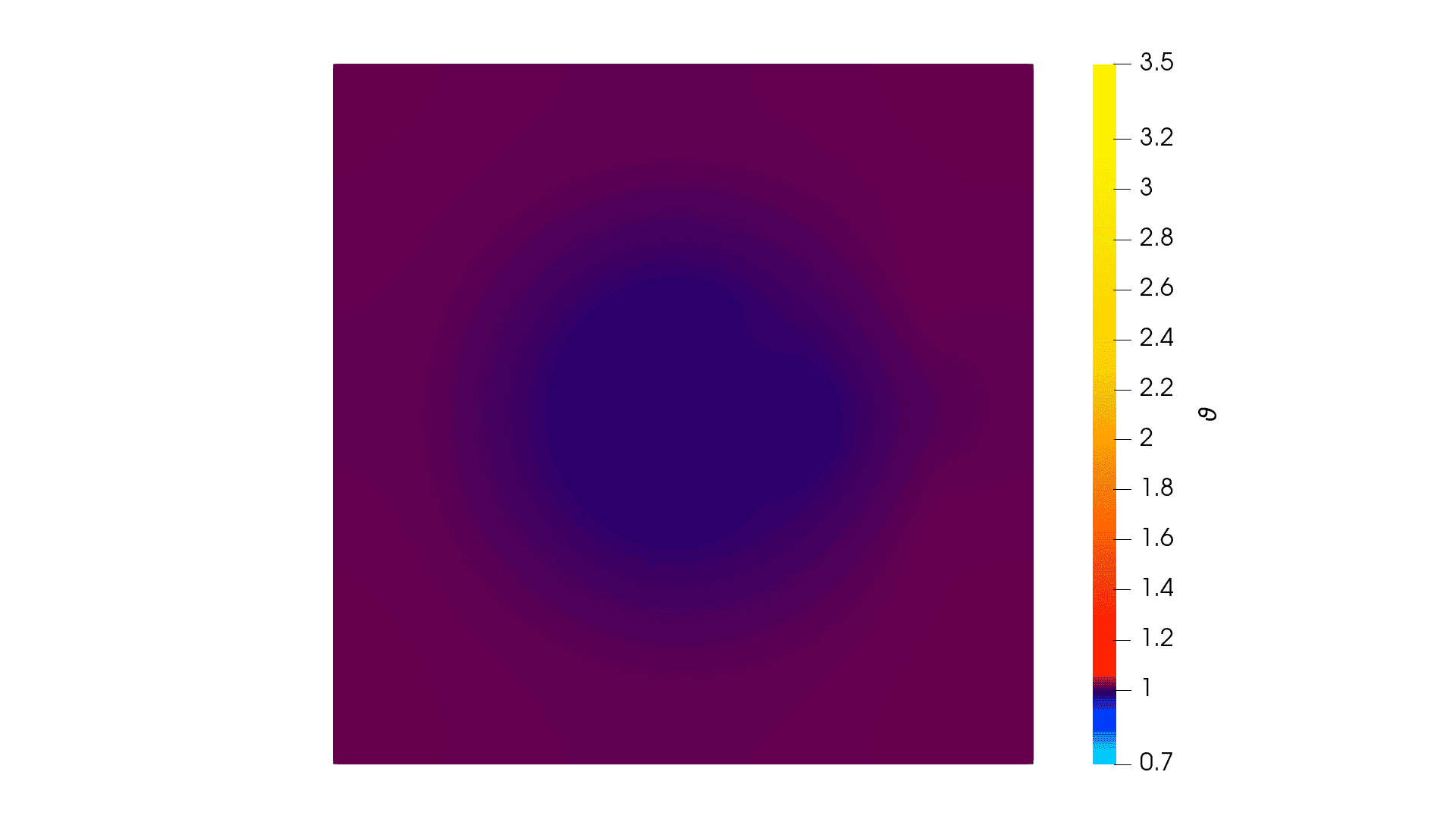} 
                &
                \includegraphics[trim={17cm 0cm 17cm 0cm},clip,scale=0.11]{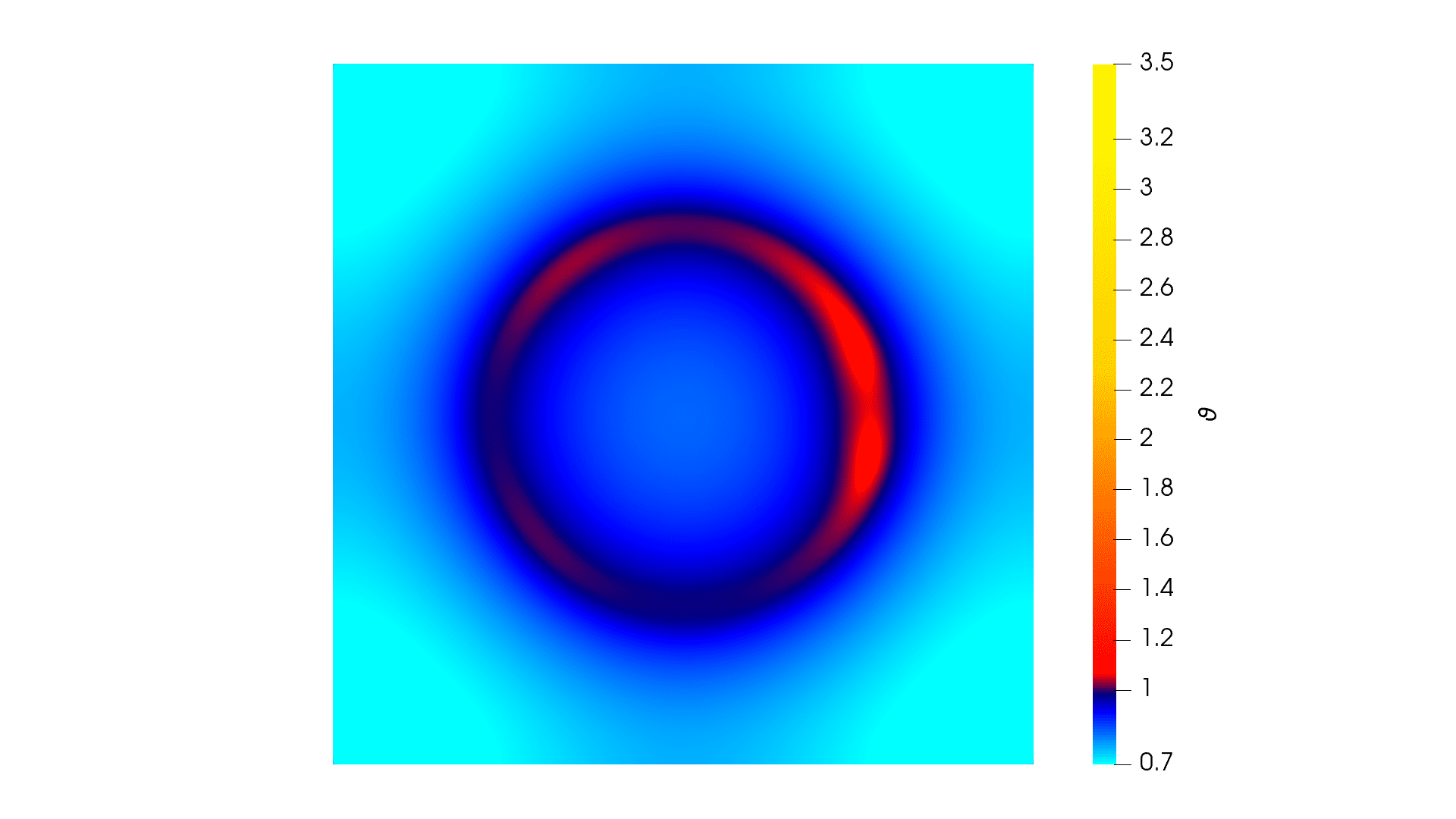}
                &
                \includegraphics[trim={17cm 0cm 17cm 0cm},clip,scale=0.11]{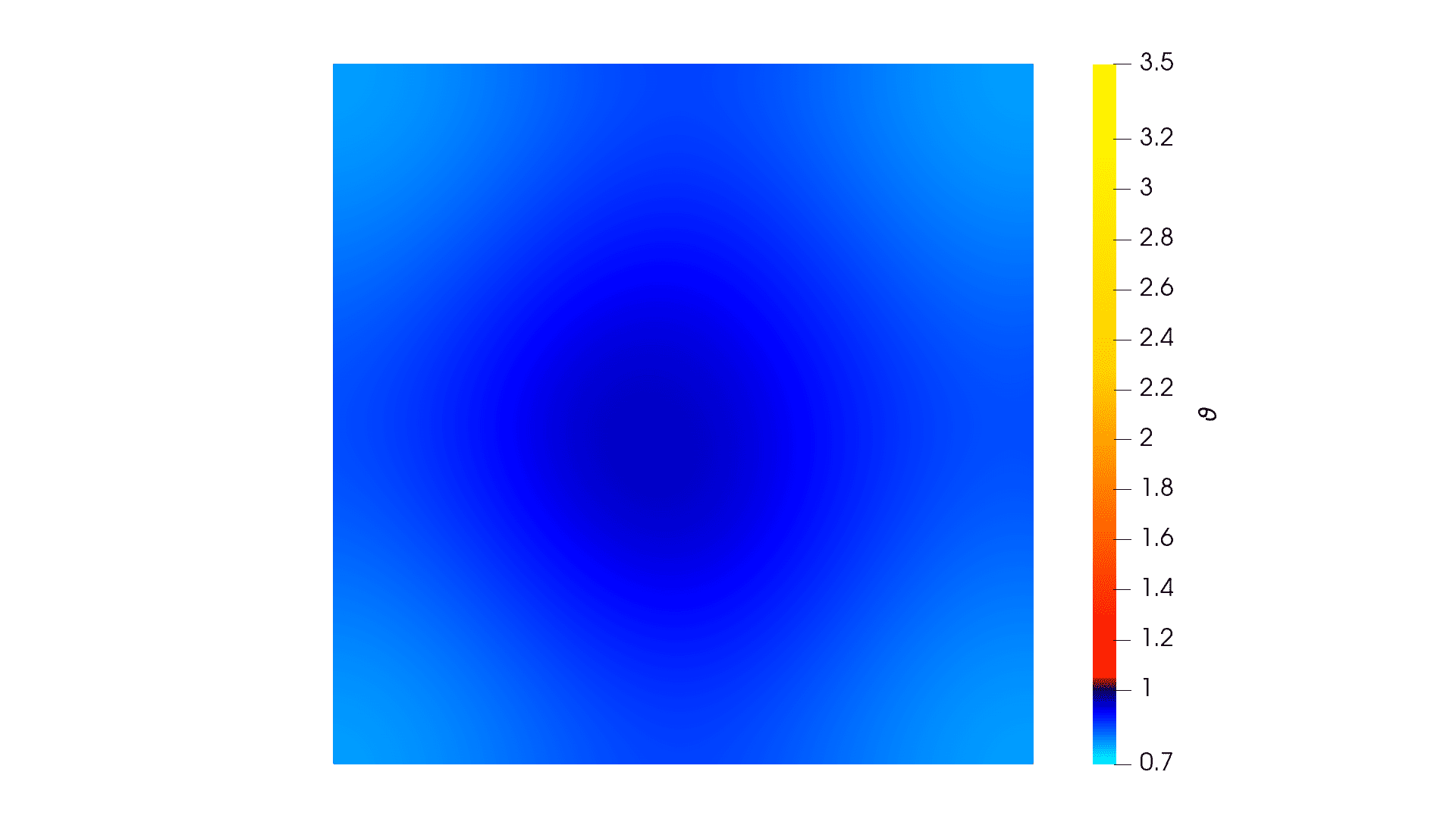}
                & 
                \\[-0.5em]
            \end{tabular}
            \caption{Snapshots of the temperature $\vtheta_h$ with the red circle outlining the heat source $Q_l$}
            \label{fig:ex_laser_theta}
        \end{figure}

        When we look at \cref{fig:ex_laser_struct}, showing the entropy $s_h$ and total energy $e_\mathrm{tot,h}$, because of the external heat source, we see a rise of energy at the beginning due to the laser constantly heating up the domain, followed by the delayed cooling between $t=1.5$ and $t=2$. Similar behavior can be seen for the entropy. If we subtract the respective source terms, as they appear in \cref{thm:discstruc}, as well as the initial values at $t=0$, we find that the entropy production is indeed positive as well as the negative dissipation in the energy, as shown in \cref{fig:ex_laser_adjstruct}.
        \begin{figure}[htbp!]
            \centering
            \hspace*{-2em}
            \begin{tabular}{c@{}c@{}}
                \includegraphics[width=0.49\linewidth]{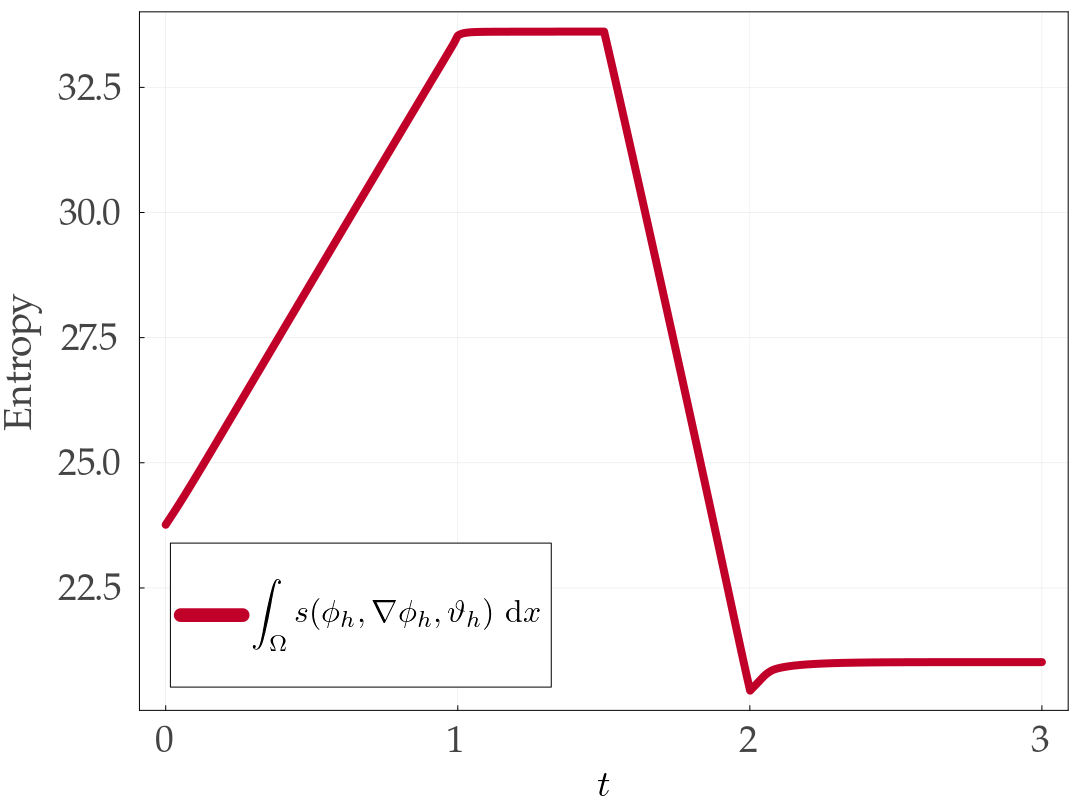}
                &
                \includegraphics[width=0.49\linewidth]{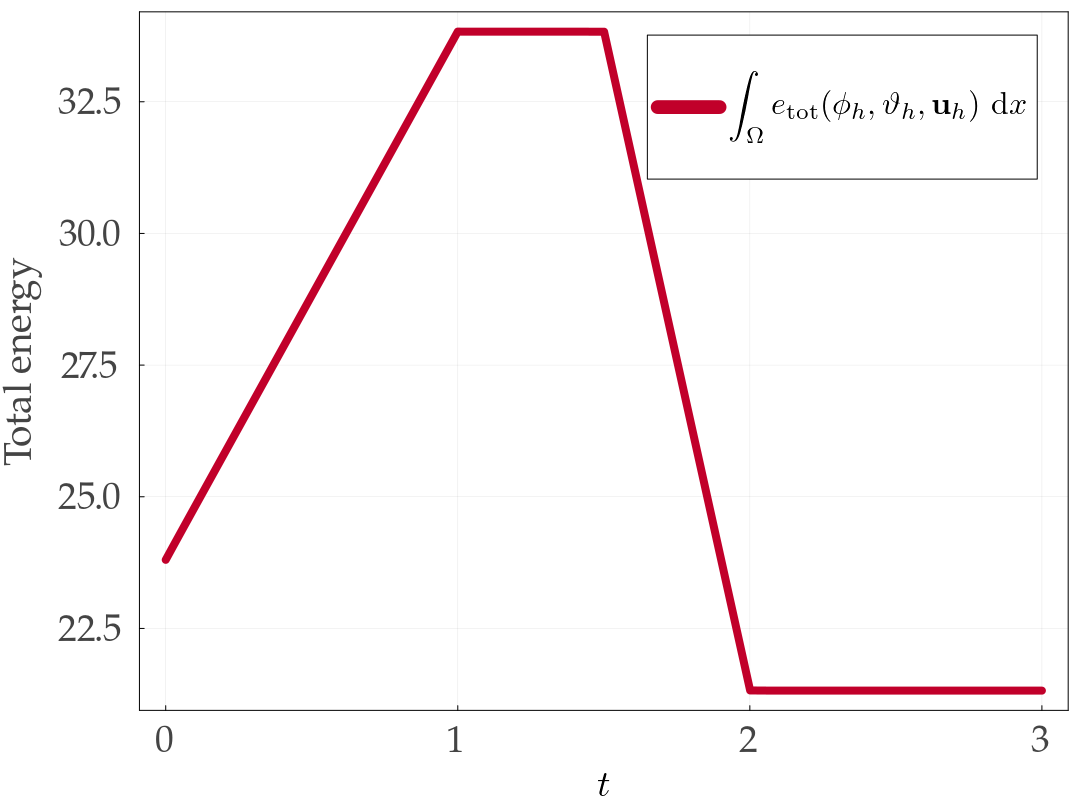}
            \end{tabular}
            \caption{Integrated entropy density $s_h$ (left) and total energy density $e_{\mathrm{tot},h}$ (right) over time.}
            \label{fig:ex_laser_struct}
        \end{figure}
        \begin{figure}[htbp!]
            \centering
            \begin{tabular}{c@{}c@{}}
                \includegraphics[width=0.49\linewidth]{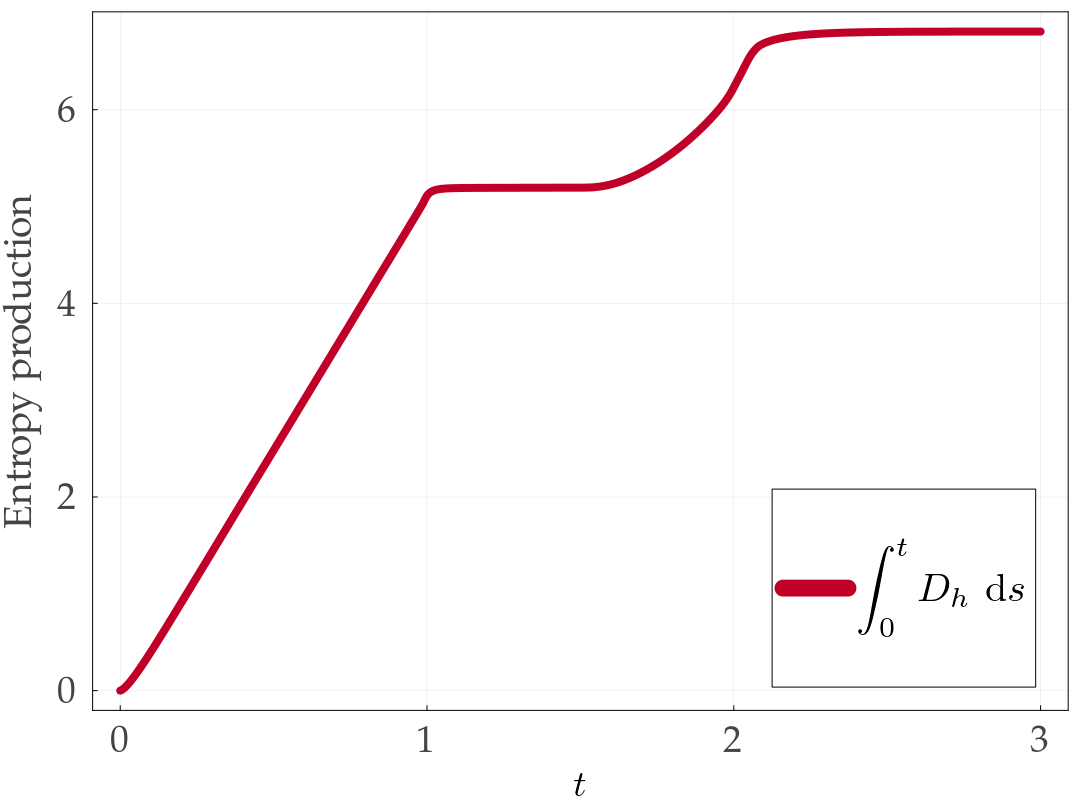}
                &
                \includegraphics[width=0.49\linewidth]{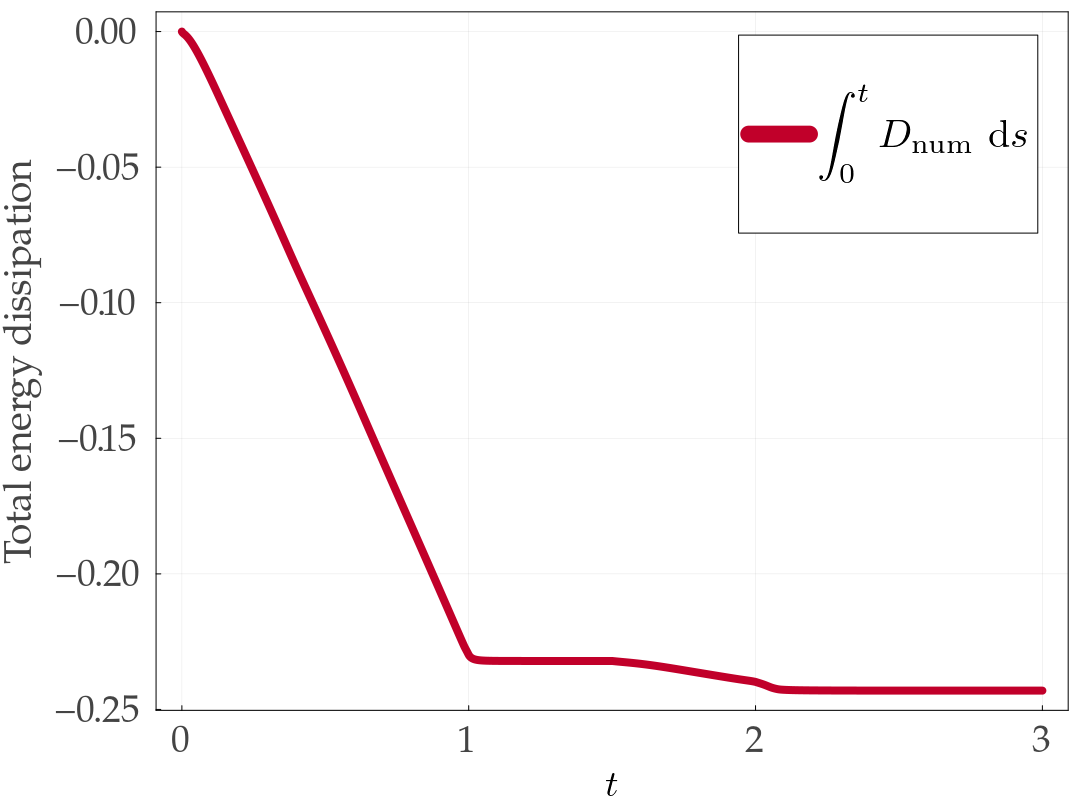}
            \end{tabular}
            \caption{Time integral of the entropy production $D_h$ (left) and numerical dissipation $D_\mathrm{num}$ (right).}
            \label{fig:ex_laser_adjstruct}
        \end{figure}
    \subsection{Dendritic growth}
        The last example utilizes an anisotropic gradient contribution $G$ of the form
        \begin{equation*}
            G(\nabla\phi)=\gamma_0^2\left(1 + \delta\dfrac{\phi_x^4 - 6\phi_x^2\phi_y^2+\phi_y^4}{(\phi_x^2+\phi_y^2+10^{-10})^2}\right)^2\snorm{\nabla\phi}^2,
        \end{equation*}
        simulating crystalline interface behavior by preferring the interface to orient in certain directions with an anisotropic strength of $\delta=0.9$. To show this behavior we consider initial values for $\phi_h,\vtheta_h$ and $\u_h$ as
        \begin{align*}
            \phi_0(x,y)&:= 0.5 + 0.5\tanh\left(\frac{(x-5)^2+(y-5)^2 - r^2}{0.008}\right),\\
            \vtheta_0(x,y)&:=\vtheta_m-\phi_0(\vtheta_m-\vtheta_b),\\
            \u_0(x,y)&:=(0,0)^\top,
        \end{align*}
        on the domain $\Omega=[0,10]^2$ which, as seen in \cref{fig:ex_dendritic_init}, represents a solid grain at melting temperature floating in an undercooled liquid, forcing the grain to grow. 
        \begin{figure}[htbp!]
            \centering
            \includegraphics[trim={16cm 0cm 9cm 0cm},clip,scale=0.14]{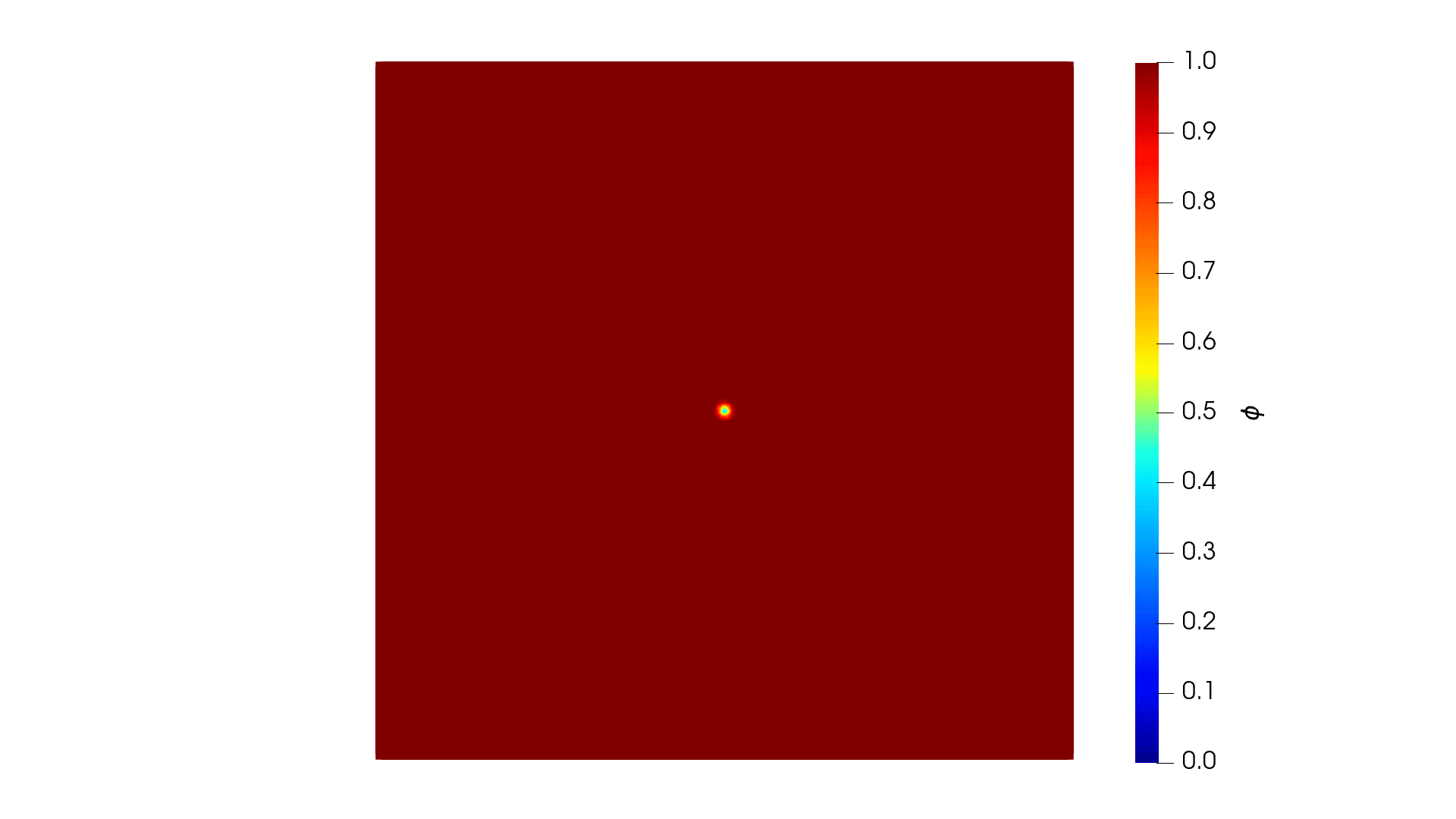}
            \includegraphics[trim={16cm 0cm 9cm 0cm},clip,scale=0.14]{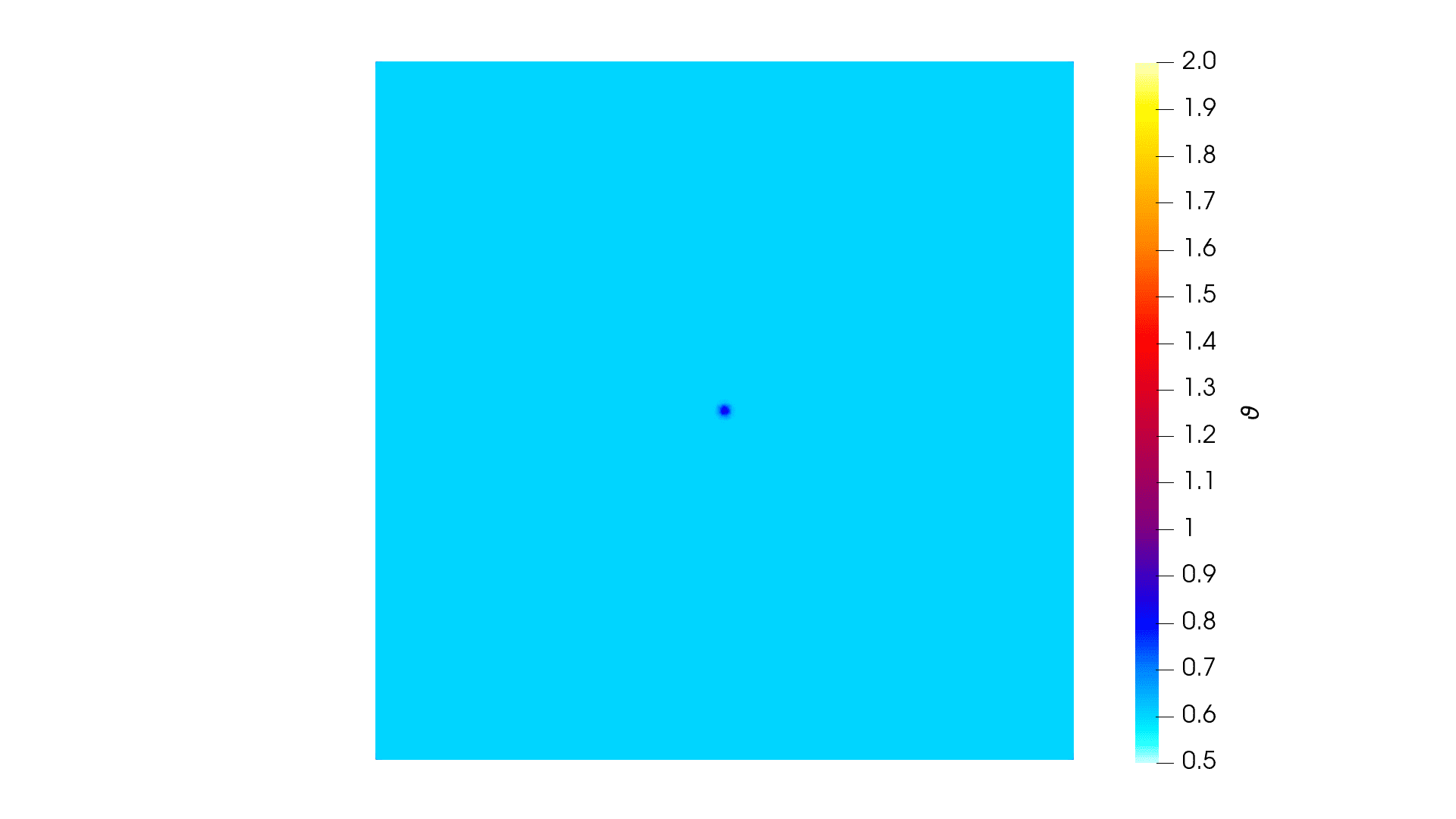} 
            \caption{Initial condition for second example of the phase variable $\phi$ (left) and the temperature $\vtheta$ (right).}
            \label{fig:ex_dendritic_init}
        \end{figure}
        
        Since we do not want the domain to heat up, slowing down the crystal growth, we use \cref{bc:thermal} to keep the boundary at a fixed temperature $\vtheta_b$. For the simulation parameters we choose the following constants:\\
        \begin{minipage}[t]{0.32\textwidth}
            \vspace{0.1cm}
            \begin{itemize}
                \item $\M=100$
                \item $\K=200$
                \item $\tau=0.00025$
                \item $\b=\mathbf{0}$
                \item $Q=0$
            \end{itemize}
            \vspace{0.2cm}
        \end{minipage}
        \begin{minipage}[t]{0.32\textwidth}
            \vspace{0.1cm}
            \begin{itemize}
                \item $\gamma_0=0.05$
                \item $\eta_s=1$
                \item $\eta_l=0.001$
                \item $\vtheta_m=1$
                \item $\vtheta_b=0.6$
            \end{itemize}
            \vspace{0.2cm}
        \end{minipage}
        \begin{minipage}[t]{0.32\textwidth}
            \vspace{0.1cm}
            \begin{itemize}
                \item $\mathcal{L}=15$
                \item $H_\mathrm{pt}=1$
                \item $H_\mathrm{cf}=0.1$
                \item $C_\mathrm{vsh}=1$
                \item $r=0.05$
            \end{itemize}
            \vspace{0.2cm}
        \end{minipage}
        Since the anisotropic gradient contribution, the grain prefers to grow along the horizontal and vertical axis as seen in \cref{fig:ex_dendritic_phitheta}, which also causes the crystal to bulge outwards on these branches near the center, eventually forming new subbranches. Meanwhile the temperature in the crystal rises slightly due to the release of latent heat by the phase change from liquid to solid.\\
        \begin{figure}[htbp!]
            \centering
            \hspace*{-1em}
            \begin{tabular}{c@{}c@{}c@{}c@{}c@{}c@{}}
                & $t=0$ & $t=0.1$ & $t=0.25$ & $t=0.35$ & 
                \\[-0.5em]
                \includegraphics[trim={8cm 0cm 52.2cm 0cm},clip,scale=0.09]{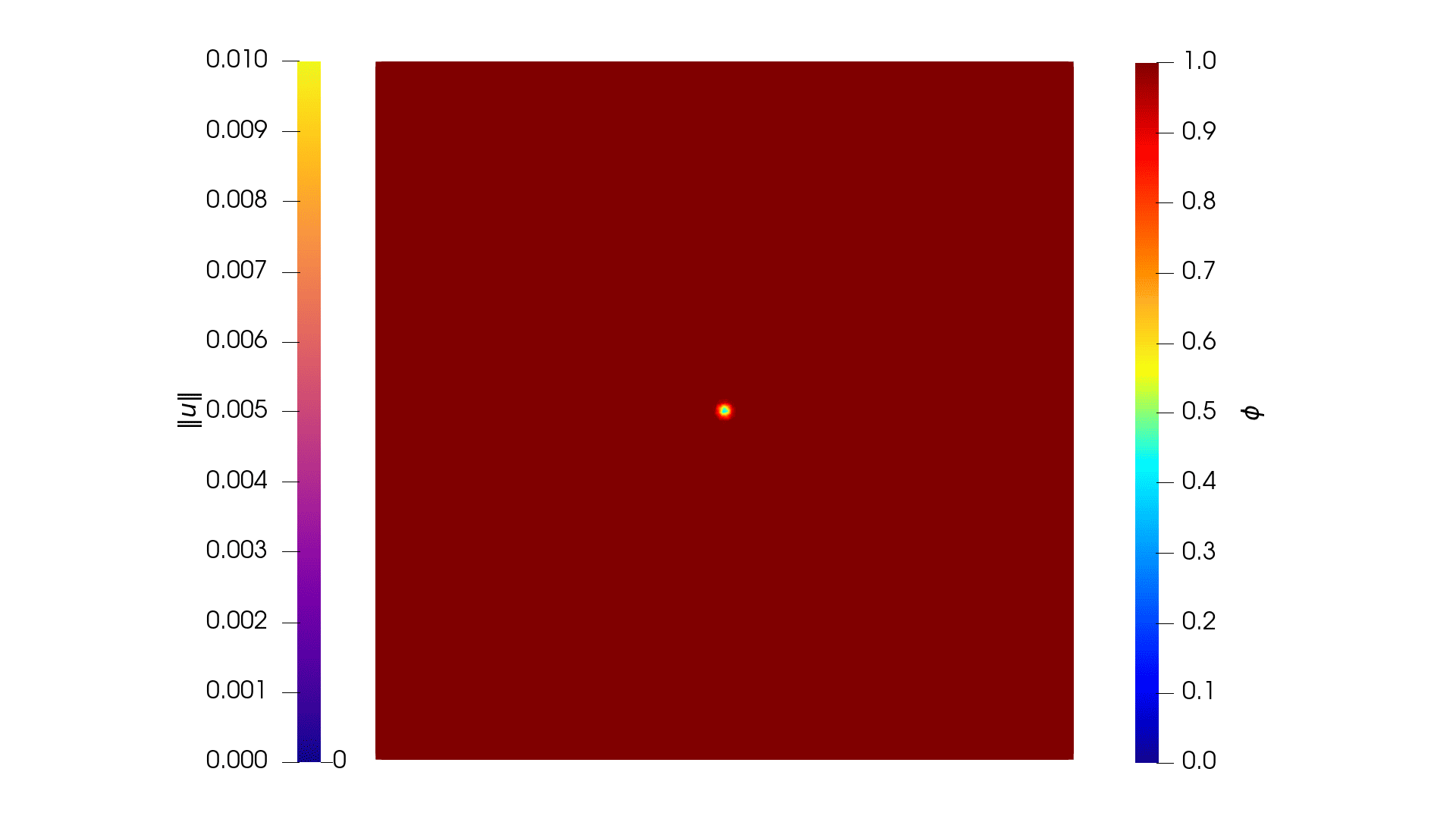}
                &
                \includegraphics[trim={17cm 0cm 17cm 0cm},clip,scale=0.09]{Bilder/Dendritic/Phi_U_0.png} 
                &
                \includegraphics[trim={17cm 0cm 17cm 0cm},clip,scale=0.09]{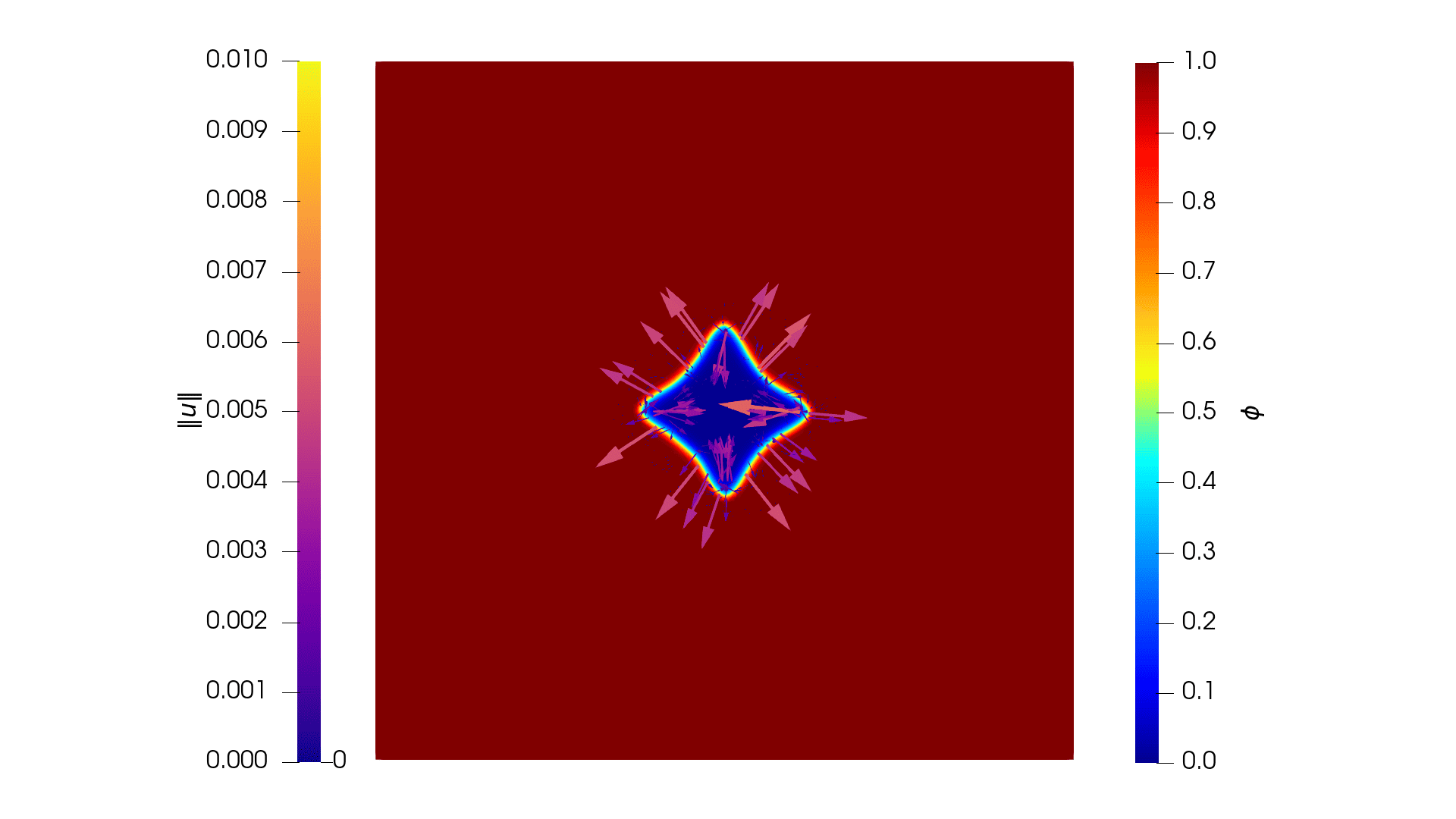} 
                &
                \includegraphics[trim={17cm 0cm 17cm 0cm},clip,scale=0.09]{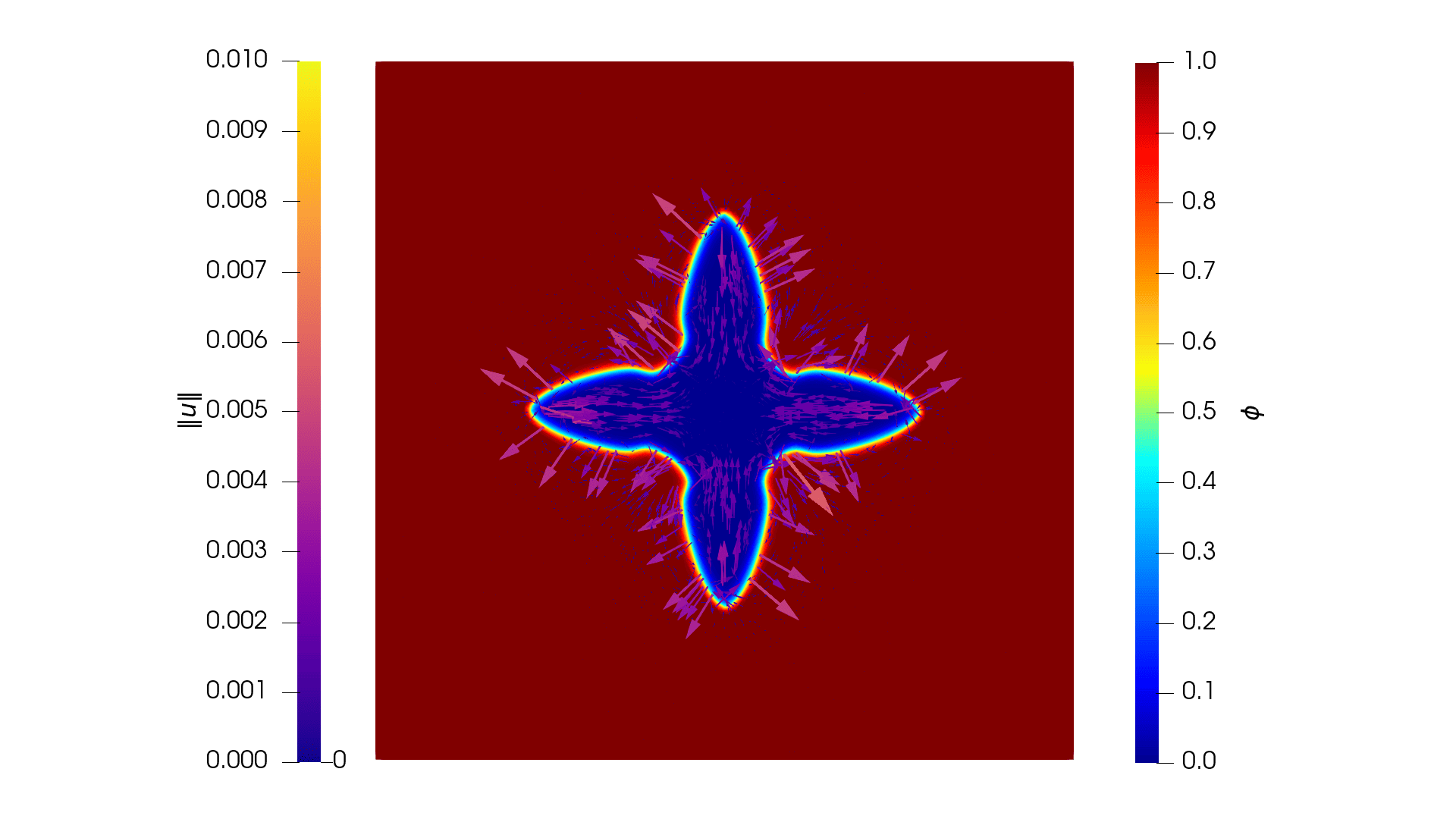}
                &
                \includegraphics[trim={17cm 0cm 17cm 0cm},clip,scale=0.09]{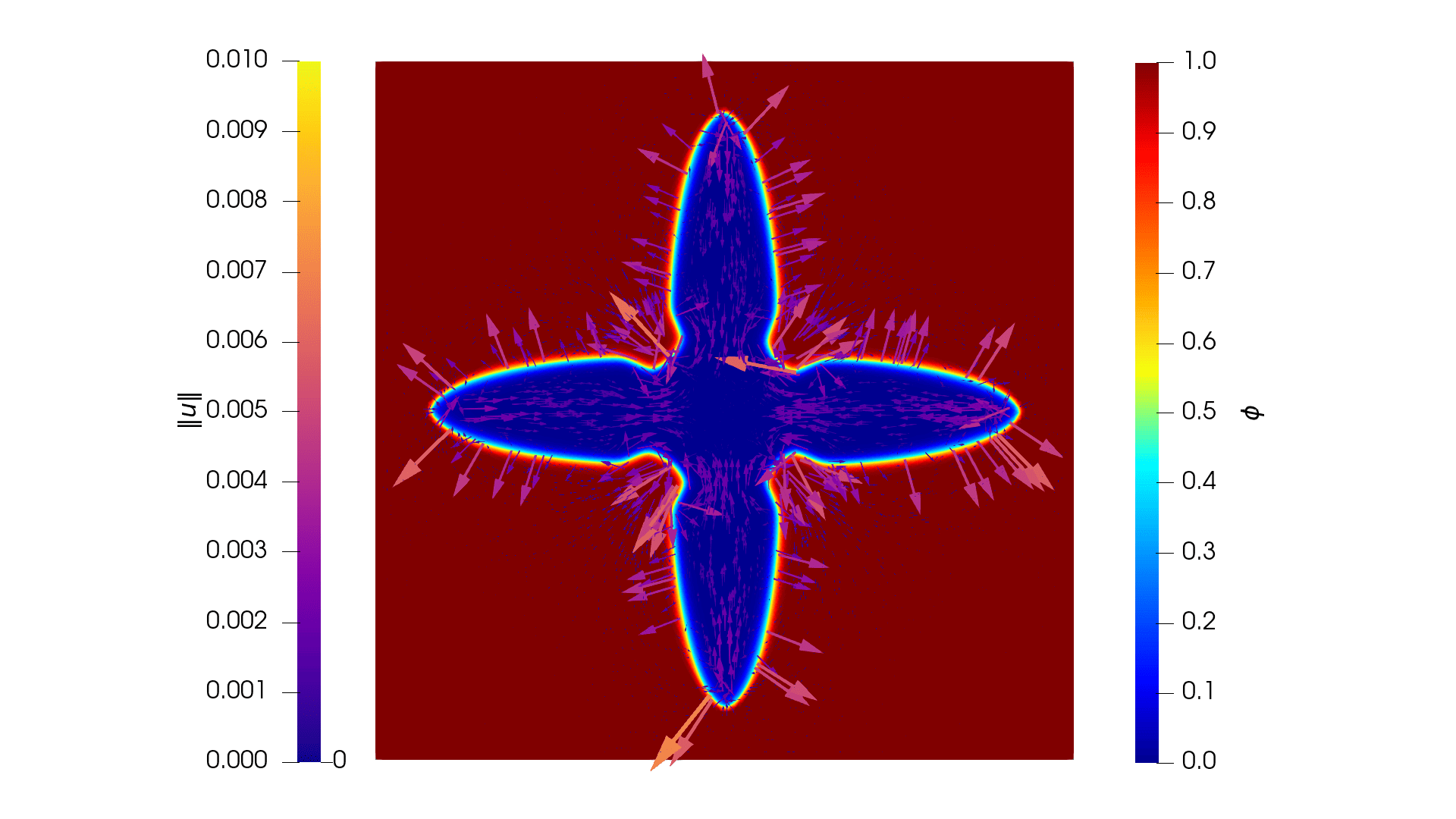}
                & 
                \includegraphics[trim={53cm 0cm 9cm 0cm},clip,scale=0.09]{Bilder/Dendritic/Phi_U_0.png}\\[-1em]
                \includegraphics[trim={9cm 0cm 52.5cm 0cm},clip,scale=0.09]{Bilder/Dendritic/Theta_0.png}
                &
                \includegraphics[trim={17cm 0cm 17cm 0cm},clip,scale=0.09]{Bilder/Dendritic/Theta_0.png} 
                &
                \includegraphics[trim={17cm 0cm 17cm 0cm},clip,scale=0.09]{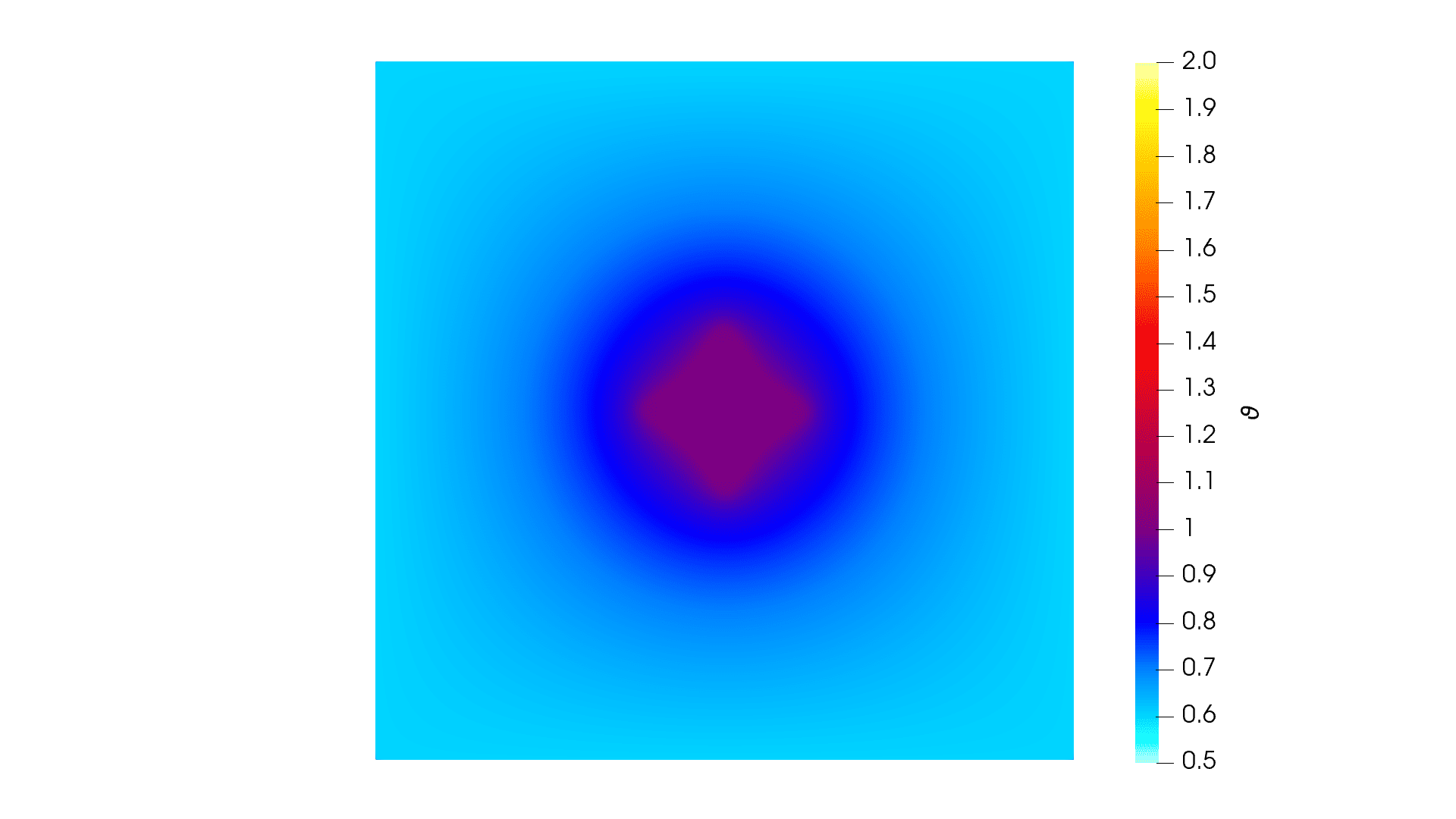} 
                &
                \includegraphics[trim={17cm 0cm 17cm 0cm},clip,scale=0.09]{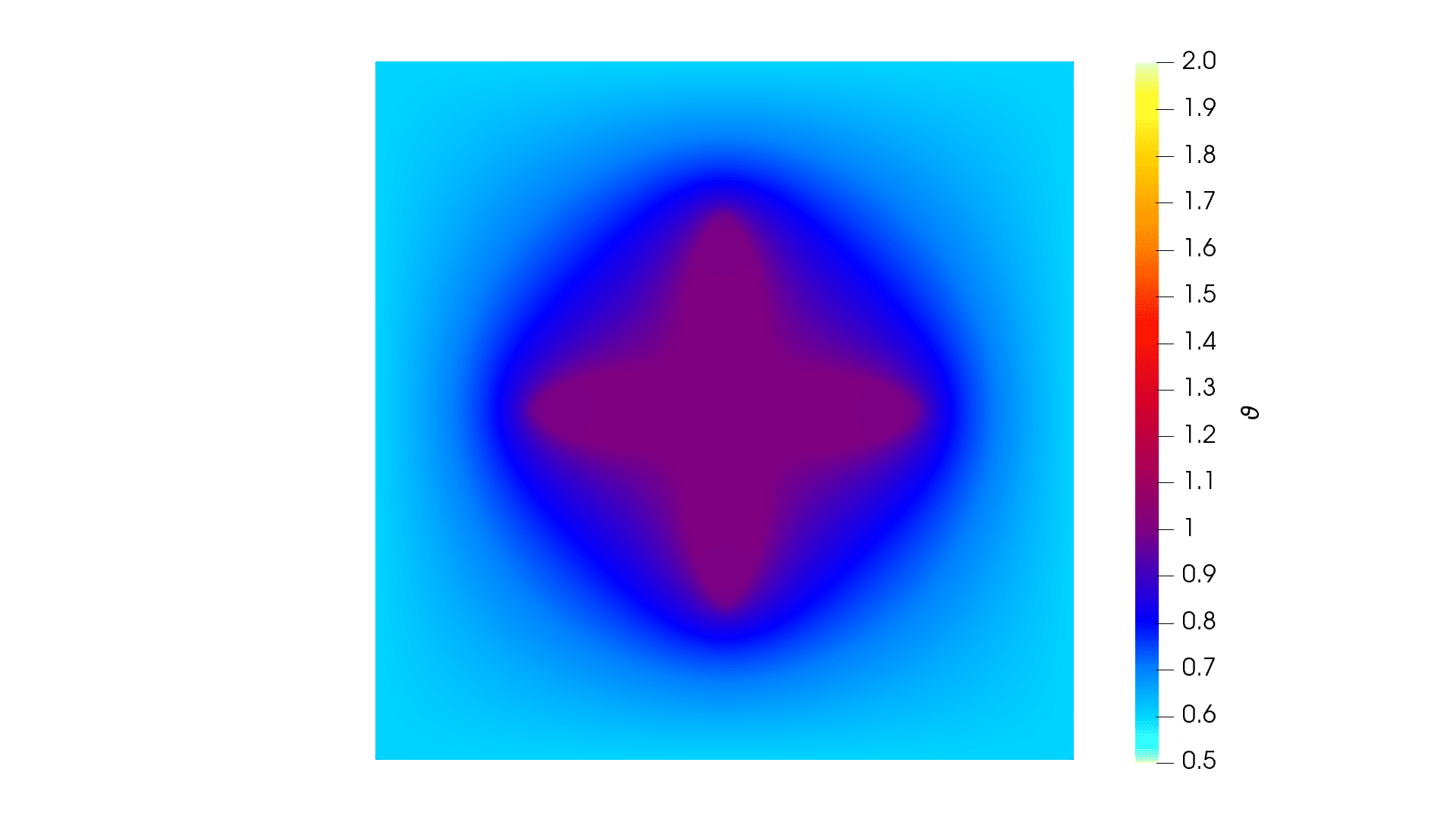}
                &
                \includegraphics[trim={17cm 0cm 17cm 0cm},clip,scale=0.09]{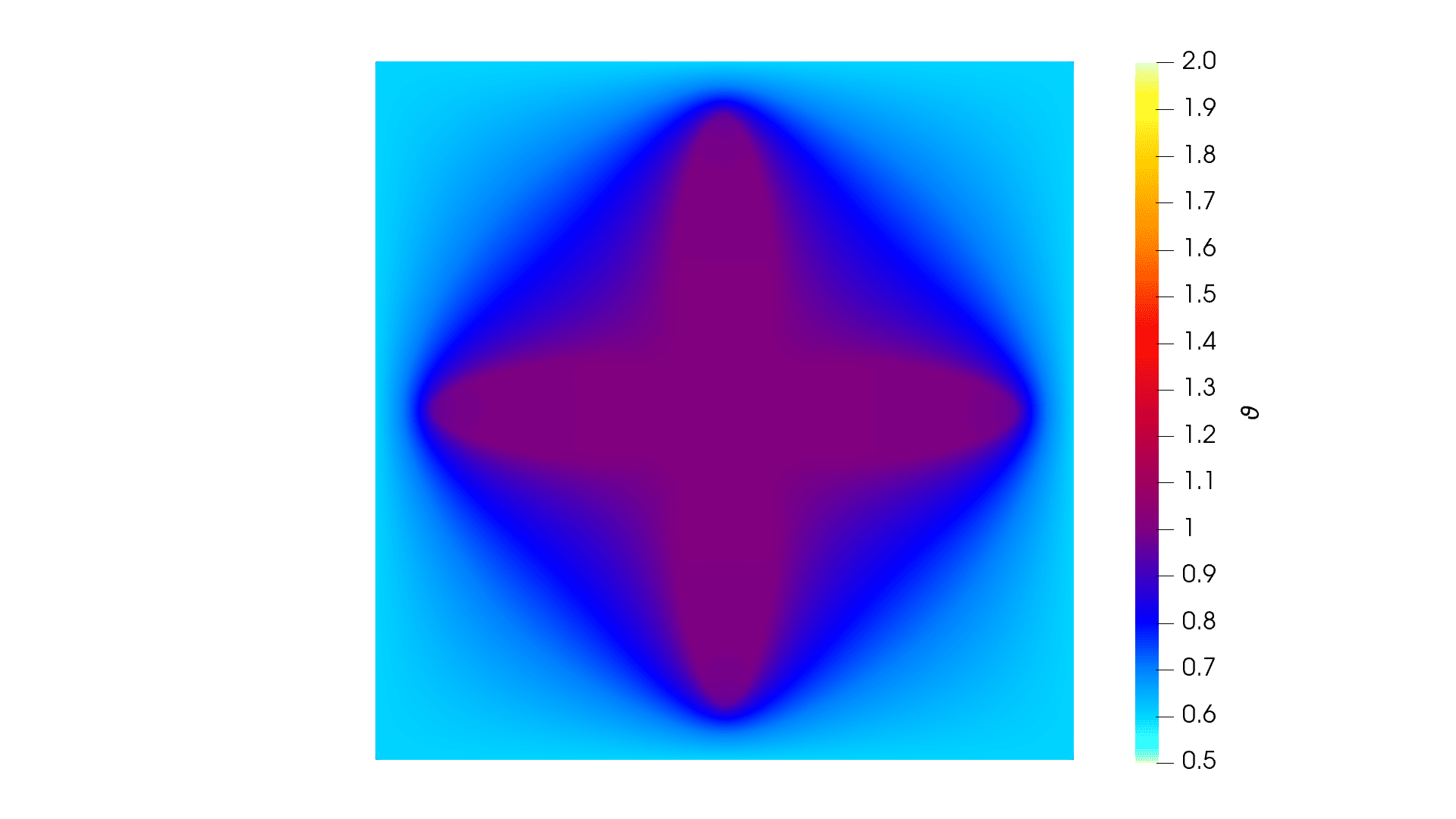}
                & 
                \includegraphics[trim={52.5cm 0cm 9cm 0cm},clip,scale=0.09]{Bilder/Dendritic/Theta_0.png}
                \\[-0.5em]
                \\
            \end{tabular}
            \caption{Snapshots of $\phi_h$ and $\u_h$ (top) and $\vtheta_h$ (bottom).}
            \label{fig:ex_dendritic_phitheta}
        \end{figure}
        
        Since we considered \cref{bc:thermal} we did not expect to see entropy production nor energy dissipation, which is confirmed by \cref{fig:ex_dendritic_struct}, instead, if we look at the exergy, one can see that \cref{thm:discstruc} holds indeed, showing a strictly negative exergy propagation.\\
        \begin{figure}[htbp!]
            \centering
            \hspace*{-2em}
            \includegraphics[width=0.6\linewidth]{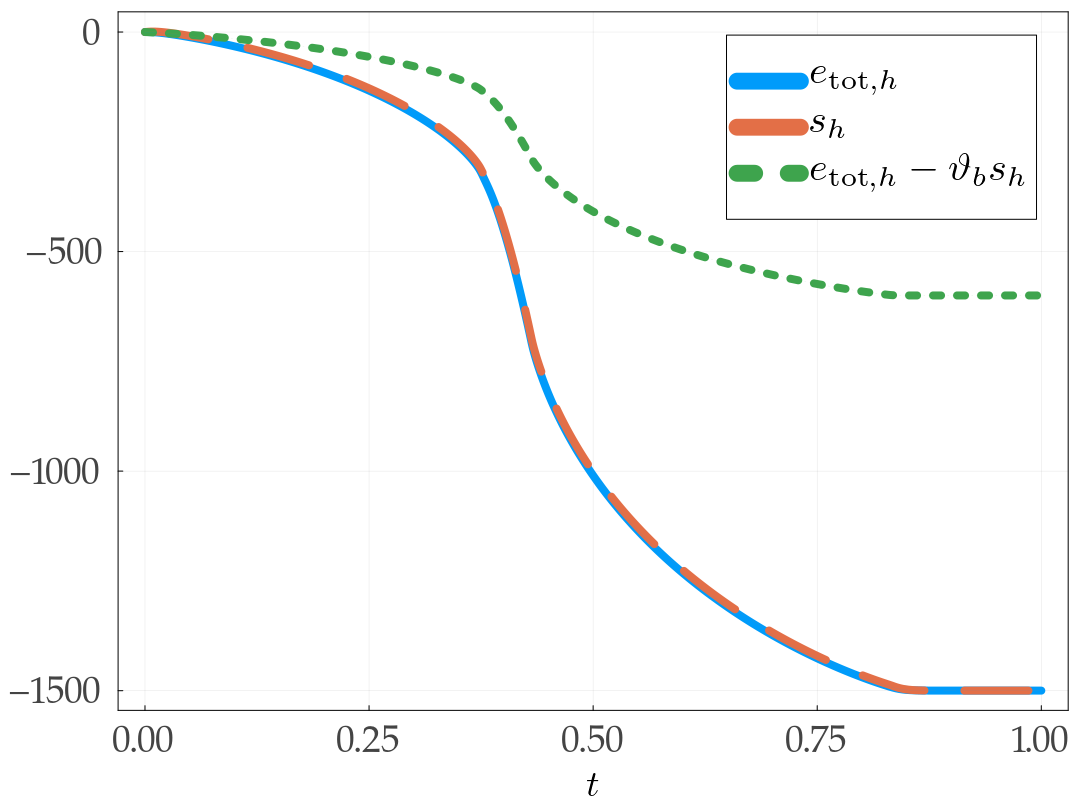}
            \caption{Integrated total energy density $e_{\mathrm{tot},h}$, entropy density $s_h$ and exergy $e_{\mathrm{tot},h}-\vtheta_bs_h$ difference over time.}
            \label{fig:ex_dendritic_struct}
        \end{figure}

\section{Conclusion}

    In this work, we have presented a fully discrete scheme for the non-isothermal Allen-Cahn-Navier-Stokes system that ensures entropy production and conserves total energy up to a numerical diffusion term, utilizing a Petrov-Galerkin-type discretization in time and standard conforming finite elements in space. The existence of at least one discrete solution has been established. The framework accommodates anisotropic interface functions for the simulation of dendritic growth. Numerical verification of the convergence of the schemes has been provided, and several applications have been demonstrated, including melt/solidification and anisotropic grain growth, while making use of different boundary conditions and external sources. 

\section*{Acknowledgment}

    The present research has been supported by the Deutsche Forschungsgemeinschaft (DFG, German Research Foundation) within the framework of the collaborative research center "Multiscale Simulation Methods for Soft-Matter Systems" (TRR 146) under Project No. 233630050 and by the SPP 2256 "Variational Methods for Predicting Complex Phenomena in Engineering Structures and Materials" under Project No. 441153493.

\section*{Conflict of interest}

	On behalf of all authors, the corresponding author states that there is no conflict of interest.
    
\appendix

\section{Proof of \cref{thm:existence}}

    The proof relies on the following consequence from topological degree theory:
    \begin{theorem}[\cite{gallouet2008unconditionally}]\label{th:topo}
        Let N and M be two positive integers and $C_1, C_2$ and $\varepsilon$ three positive constants. We define 
        \begin{align*}
            V &= \{(\x, \y) \in \RR^N\times\RR^M ~\text{such that}~\y > 0\},\\
            W &= \{(\x, \y)\in\RR^N\times\RR^M~\text{s.t.}~\|\x\|<C_1 ~\text{and}~ \varepsilon<\y<C_2\}.
        \end{align*}
        Here, the notation $\y>c$ means that each component of $\y$ is greater than a constant $c\in\RR$. Further $\|\cdot\|$ is a norm defined over $\RR^N$. Let
        $\b\in\RR^N\times\RR^M$, $\mathbf{g}$ and $\G$ be two continuous functions, respectively, from $V$ and $V\times[0, 1]$ to $\RR^N\times\RR^M$
        satisfying the following conditions:
        \begin{enumerate}[label=(\roman*)]
            \item $\G(\cdot, 1) = \mathbf{g}(\cdot)$
            \item The topological degree of $\G(\cdot, 0):=\G_0(\cdot)$ with respect to $W$ and $\b$ is equal to $d_0\neq0$, i.e.
            \begin{equation*}
                \mathrm{deg}(\G_0,W,\b):= \sum_{\x\in \G_0^{-1}(\b)}\mathrm{sgn}\left(\det\left(J_{\G_0}(\x)\right)\right) = d_0.
            \end{equation*}
            \item For all $\alpha\in[0, 1]$, if $\mathbf{v}\in V$ is such that $\G(\mathbf{v},\alpha) = \b$ then $\mathbf{v}\in W$.
        \end{enumerate}
        Then the topological degree of $\G(\cdot, 1)$ with respect to $W$ and $\b$ is also equal to $d_0\neq0$. Consequently, there exists
        at least one solution $\mathbf{v}\in W$, such that $\mathbf{g}(\mathbf{v}) = \b$.
    \end{theorem}

    In order to apply the above theorem we consider the spaces
    \begin{align*}
        V &:= \{(\boldsymbol{\phi},\boldsymbol{\mu},\boldsymbol{\vtheta},\boldsymbol{u},\boldsymbol{p})\in\mathbb{R}^{4N-1+M} \text{ such that } \boldsymbol{\vtheta}>0\}, \\
        W &:=\{(\boldsymbol{\phi},\boldsymbol{\mu},\boldsymbol{\vtheta},\boldsymbol{u},\boldsymbol{p})\in\mathbb{R}^{4N-1+M} \text{ such that } \norm{\boldsymbol{\phi}} + \norm{\boldsymbol{\mu}} + \norm{\boldsymbol{u}} + \norm{\boldsymbol{p}}< C_{\phi,\mu}, \varepsilon <\boldsymbol{\vtheta}< C_\vtheta,\}, 
    \end{align*}
    where $N$ corresponds to the number of nodal values in $\Vh$ and $M$ in $\Xh^d$. As usual the vectors contain the nodal values of the finite element functions and hence we can identity the vector of nodal values and the corresponding finite element function. We introduce the auxiliary problem  $\G(\x,\alpha):=\la T_\alpha x_h,y_h \ra=\mathbf{0}$ for $x_h=(\phi_h,\mu_h,\vtheta_h,\u_h,p_h)\in V$ and all basis functions $y_h=(\psi_h,\xi_h,\omega_h,\vv_h,q_h)\in\Vh\times\Vh\times\Vh^0\times\Xh^d\times\Vh$ with $\la T_\alpha x_h,y_h \ra=0$ given by the variational problem
    \begin{align}
        &\la\dtau\phi_h,\psi_h\ra + \alpha\la\u_h,\psi_h\nabla\phi_h\ra + \alpha\la\tfrac{\M_h}{\vtheta_h}\mu_h,\psi_h\ra = 0,\label{eq:exdis1}\\
        &\la\mu_h,\xi_h\ra - \gamma\la \nabla\phi_h,\vtheta_h\nabla\xi_h+\xi_h\nabla \vtheta_h\ra - \la \partial_\phi f(\phi_h,\phi_h^n,\vtheta_h),\xi_h\ra = 0,\label{eq:exdis2}\\
        &\la\dtau s_h,\omega_h\ra - \alpha\la\tfrac{\M_h}{\vtheta_h}\mu_h,\tfrac{\omega_h}{\vtheta_h}\mu_h\ra - \alpha\la\tfrac{\eta_h}{\vtheta_h}\snorm{\Du_h}^2,\omega_h\ra\notag\\
        &\quad- \alpha\la\tfrac{\K_h}{(\vtheta_h)^2}\nabla \vtheta_h,\tfrac{\omega_h}{\vtheta_h\cdot \vtheta_h}\nabla \vtheta_h-\tfrac{1}{\vtheta_h}\nabla\omega_h\ra  \notag\\
        &\quad- \la\gamma\nabla\phi_h\dtau\phi_h + \alpha(\bsig_h+s_h\mathbf{I})\u_h,\nabla\omega_h\ra = 0,\label{eq:exdis3}\\
        &\la\dtau\u,\vv_h\ra + \alpha\cskw(\u_h,\u_h,\vv_h) + \alpha\la\eta_h\Du_h,\Dv_h\ra - \la p_{h},\div(\vv_h)\ra\notag\\
        &\quad+ \alpha\la(\bsig_h+s_h\mathbf{I})\nabla\vtheta_h-\mu_h\nabla\phi_h,\vv_h\ra = 0,\label{eq:exdis4}\\
        &\la\div(\u_h),q_h\ra = 0.\label{eq:exdis5}
    \end{align}
    
    We briefly outline the relevant steps or deviations from \cite{Hoehn26}. By construction the first condition in \cref{th:topo} is satisfied.
    
    \textbf{Conditions (ii):}
    
        We consider $\alpha=0$ and obtain the system $\G_0(\x):=\G(\x,0)=0$ given by:
        \begin{align}
            &\la\dtau\phi_h,\psi_h\ra = 0, \label{eq:alnull1}\\
            &\la\mu_h,\xi_h\ra - \gamma\la \nabla\phi_h,\vtheta_h\nabla\xi_h + \xi_h\nabla \vtheta_h\ra - \la \partial_\phi f(\phi_h,\phi_h^n,\vtheta_h),\xi_h\ra\label{eq:alnull2}\\
            &\la\dtau s_h,\omega_h\ra - \gamma\la \nabla\phi_h\dtau\phi_h ,\nabla\omega_h\ra = 0, \label{eq:alnull3}\\
            &\la\dtau\u,\vv_h\ra - \la p_{h},\div(\vv_h)\ra = 0, \label{eq:alnull4}\\
            &\la\div(\u_h),q_h\ra = 0. \label{eq:alnull5}
        \end{align}
    
        Using $\psi_h= \phi_h-\phi_h^n$ in \eqref{eq:alnull1} we find $\phi_h=\phi_h^n$. Then using Taylor's theorem in \eqref{eq:alnull3} with $\omega=\vtheta_h-\vtheta_h^n$ yields $\vtheta_h=\vtheta_h^n$. From  which $\mu_h$ can be computed and estimated as in \cite{Hoehn26}. For the velocity we use $\vv_h=\u_h-\u_h^n$ and $q_h=p_h$ to find $\u_h=\u_h^n.$ Hence, only bounds for the pressure remains. Here we invoke inf-sup stability, i.e.
        \begin{equation*}
            \norm{p_h}_{L^2} \leq C_\mathrm{is}\sup_{\vv_h\in\Xh^d} \frac{\la p_h,\div(\vv_h) \ra}{\norm{\vv_h}_{H^1}} =0,
        \end{equation*}
        as we inserted into the momentum equation. 
        
        Next, we will show that the topological degree is non-zero. To this end we compute the Jacobian matrix of $\G_0$
        \begin{equation*}
         J_{\G_0}(\x) = \begin{pmatrix}
             \mathbf{A} & \mathbf{0} & \mathbf{0} & \mathbf{0} & \mathbf{0}\\
             \mathbf{H}^\phi & \mathbf{A} & \mathbf{H}^\vtheta & \mathbf{0} & \mathbf{0}\\
              \mathbf{S}^\phi & \mathbf{0} & \mathbf{S}^\vtheta & \mathbf{0} & \mathbf{0}\\
              \mathbf{0} & \mathbf{0} & \mathbf{0} & \mathbf{M} & -\mathbf{B} \\
              \mathbf{0} & \mathbf{0} & \mathbf{0} & -\mathbf{B}^\top &\mathbf{0}
         \end{pmatrix}  
        \end{equation*}
        with the matrices
        \begin{gather*}
            \mathbf{A}_{ij} = \la \lambda_j,\lambda_i \ra \quad\mathbf{S}^\phi_{i,j} = \la \dphi s(\phi_h,\vtheta_h)\lambda_j,\lambda_i \ra, \quad \mathbf{S}^\vtheta_{i,j} = \la \dtheta s(\phi_h,\vtheta_h)\lambda_j,\lambda_i \ra,\\
            \mathbf{H}_{i,j}^\phi = \gamma\la \nabla\lambda_j,\vtheta_h\nabla\lambda_i \ra +\gamma\la \nabla\lambda_j,\lambda_i\nabla\vtheta_h \ra + \la \partial_{\phi_h} f_\phi(\phi_h,\phi_h^n,\vtheta_h)\lambda_j,\lambda_i \ra,\\
            \mathbf{H}_{i,j}^\vtheta = \gamma\la \nabla\phi_h,\lambda_j\nabla\lambda_i\ra + \gamma\la \nabla\phi_h,\lambda_i\nabla\lambda_j \ra+ \la \partial_{\vtheta_h} f_\phi(\phi_h,\phi_h^n,\vtheta_h)\lambda_j,\lambda_i \ra, \\
            \mathbf{M}_{ij} = \la \tilde\lambda_j,\tilde\lambda_i \ra\quad\mathbf{B}_{i,j} = \la \div(\tilde\lambda_i),\lambda_j- \la \lambda_j,1 \ra \ra.
        \end{gather*}
        Here the functions $\lambda_i$ are the finite element basis functions of $\Vh$ and $\tilde\lambda_i$ the basis functions of $\Xh^d$. The topological degree involves the signum of the determinant of the Jacobian matrix. We can compute that $\det(J_{\G_0}(\x))=\det(\mathbf{A})^2\det(\mathbf{S}^\vtheta)\det(\mathbf{M})\det(-\mathbf{B}^\top\mathbf{M}^{-1}\mathbf{B})$. The mass matrix $\mathbf{A}$ is positive definite, hence $\det(\mathbf{A})^2> 0$. Since $\dtheta s = -\partial_{\vtheta\vtheta} F$, we find using strict concavity assumptions of $F$, i.e \labelcref{as:exergy}, that $-\partial_{\vtheta\vtheta}F> 0$ and therefore $\mathrm{sgn}(\det(\mathbf{A})^2\det(\mathbf{S}^\vtheta)\det(\mathbf{M}))=1$. The remaining part is the determinant connected to the inf-sup stable velocity pressure-pair and we find $\det(-\mathbf{B}^\top\mathbf{M}^{-1}\mathbf{B})=(-1)^{N}\det(\mathbf{B}^\top\mathbf{M}^{-1}\mathbf{B})$. The determinant is non-zero and hence the sign depends on $N$, which is fixed for a given mesh. The topological degree is then given by
        \begin{equation*}
            \mathrm{deg}(\G_0,W,\mathbf{0}):=  \sum_{\x\in \G_0^{-1}(\mathbf{0})}\mathrm{sgn}\left(\det\left(J_{\G_0}(\x)\right)\right) = \sum_{\x\in \G_0^{-1}(\mathbf{0})} (-1)^{N}.
        \end{equation*}
        Consequently, the topological degree is non-zero.      
                
    \textbf{Condition (iii):}
    
        We follow the same route as in \cite{Hoehn26} and establish total energy conservation and entropy production for every $\alpha$, i.e.
        \begin{align*}
            \la \dtau s_h,1\ra &=  \alpha\mathcal{D}_h^{n+1},\qquad \la \dtau e_{\mathrm{tot}},1\ra \geq 0.
        \end{align*}
        From this bounds and \cref{as:exergy} we can directly extract the following estimate, i.e.
        \begin{equation*}
            \norm{\nabla\phi_h}_{L^2}^2 + \norm{\u_h}_{L^2}^2 + \alpha\tau\norm{\tfrac{\mu_h}{\vtheta_h}}_{L^2}^2  \leq C_1.
        \end{equation*}
        To gain uniform norm control of $\phi_h$ we use $\psi_h=1$ in \eqref{eq:exdis1} and obtain
         \begin{align*}
                \la \phi_h,1 \ra &= \la \phi_h^n,1 \ra - \alpha\tau\la \u_h,\nabla\phi_h \ra - \alpha\tau\la \tfrac{\M_h}{\vtheta_h}\mu_h,1 \ra \\
                 &= \la \phi_h^n,1 \ra + \alpha\tau\la\div(\u_h),\phi_h \ra - \alpha\tau\la \tfrac{\M_h}{\vtheta_h}\mu_h,1 \ra.
            \end{align*}
        Insertion of \eqref{eq:exdis5} with $q_h = \phi_h - \la \phi_h, 1\ra$ eliminates the velocity contribution. Using the entropy bounds and the boundedness assumptions for the mobility $M$ we obtain  
        \begin{align*}
                \la \phi_h,1 \ra \leq C_1 + \alpha\tau\norm*{\tfrac{\M_h}{\vtheta_h}\mu_h}_{L^2}^2 \leq C_2.
        \end{align*}
        
        With this we can resort to Poincare inequality to obtain a uniform norm bounds of $\phi_h$. From this we can simply follow the same lines to derive uniform upper and lower bounds for $\vtheta_h$. Hence, we have to provide uniform bounds for the pressure $p_h$. Again we invoke inf-sup stability which reads  
            \begin{align*}
                \norm{p_h}_{L^2}&\leq C_\mathrm{is}\sup_{\vv_h\in\Xh^d} \frac{\la p_h,\div(\vv_h) \ra}{\norm{\vv_h}_{H^1}} \\
                &\leq C_\mathrm{is}\sup_{\vv_h\in\Xh^d} \frac{1}{\norm{\vv_h}_{H^1}} \Big( \la\dtau\u_h,\vv_h\ra + \alpha\cskw(\u_h,\u_h,\vv_h) + \alpha\la\eta_h\Du_h,\Dv_h\ra \\
                &\qquad\qquad+ \alpha\la(\bsig_h+s_h\mathbf{I})\nabla\vtheta_h-\mu_h\nabla\phi_h,\vv_h\ra\Big) \\
                & \leq C_\mathrm{is}(\norm{\dtau\u_h}_{L^2} + \norm{\u_h}_{L^3}\norm{\nabla\u_h}_{L^2} + \norm{\u_h}_{L^3}\norm{\u_h}_{L^6} + \norm{\Du_h}_{L^2} \\
                &\quad+ \norm{\bsig_h}_{L^3}\norm{\nabla\vtheta_h}_{L^2} + h^{-d/2}\norm{s_h}_{L^1}\norm{\nabla\vtheta_h}_{L^\infty} + \norm{\mu_h}_{L^2}\norm{\nabla\phi_h}_{L^3}).
            \end{align*}
            For the term regarding the entropy we used the inverse inequality. At this stage all norms involving polynomials of the finite element function can be bounded using inverse inequality. This includes $\bsig_h$ but not $s_h$. For the entropy we already know by entropy production that it is bounded from below. However, by structural \cref{as:exergy,as:Gphi} and bounded solutions we also obtain a upper bound.

\bibliography{lit}

@article{Hoehn26,
	title        = {Structure-preserving approximation of the non-isothermal {C}ahn-{H}illiard system based on the entropy equation},
	author       = {Aaron Brunk and Dennis H\"{o}hn and M\'{a}ria {{Luk\'{a}\v{c}ov\'{a}-Medvi{\softd}ov\'{a}}}},
	year         = 2026,
	journal      = {Appl. Math. Comput.},
	volume       = 522,
	pages        = 129987,
	doi          = {https://doi.org/10.1016/j.amc.2026.129987},
	url          = {https://www.sciencedirect.com/science/article/pii/S0096300326000391}
}

@article{potential,
	title        = {Non-isothermal phase-field modeling of heat--melt--microstructure-coupled processes during powder bed fusion},
	author       = {Yang, Yangyiwei and K{\"u}hn, Patrick and Yi, Min and Egger, Herbet and Xu, Bai-Xiang},
	year         = 2020,
	journal      = {JOM},
	publisher    = {Springer Science and Business Media LLC},
	volume       = 72,
	number       = 4,
	pages        = {1719--1733}
}

@article{gallouet2008unconditionally,
	title        = {{An unconditionally stable pressure correction scheme for the compressible barotropic Navier-Stokes equations}},
	author       = {Gallou{\"e}t, Thierry and Gastaldo, Laura and Herbin, Raphaele and Latch{\'e}, Jean-Claude},
	year         = 2008,
	journal      = {ESAIM. Math. Model. Numer. Anal.},
	publisher    = {EDP Sciences},
	volume       = 42,
	number       = 2,
	pages        = {303--331}
}

@article{alt1992existence,
	title        = {{Existence of solutions for non-isothermal phase separation}},
	author       = {Alt, WH},
	year         = 1992,
	journal      = {Adv. Math. Sci. Appl.},
	volume       = 1,
	pages        = {319--409}
}

@book{Rubinstein1971,
	title        = {The Stefan Problem},
	author       = {L. I. Rubinstein},
	year         = 1971,
	publisher    = {American Mathematical Society},
	address      = {Providence, RI}
}

@book{Alexiades1992,
	title        = {Mathematical Modeling of Melting and Freezing Processes},
	author       = {V. Alexiades and A. D. Solomon},
	year         = 1992,
	publisher    = {Hemisphere Publishing},
	address      = {Washington, DC}
}

@article{Caginalp1986,
	title        = {An analysis of a phase field model of a free boundary},
	author       = {Caginalp,  Gunduz},
	year         = 1986,
	journal      = {Arch. Ration. Mech. Anal.},
	publisher    = {Springer Science and Business Media LLC},
	volume       = 92,
	number       = 3,
	pages        = {205–245},
	doi          = {10.1007/bf00254827},
	url          = {http://dx.doi.org/10.1007/BF00254827}
}

@article{Caginalp1989,
	title        = {Stefan and {H}ele-{S}haw type models as asymptotic limits of the phase-field equations},
	author       = {Caginalp, G.},
	year         = 1989,
	journal      = {Phys. Rev. A},
	publisher    = {American Physical Society},
	volume       = 39,
	pages        = {5887--5896},
	doi          = {10.1103/PhysRevA.39.5887},
	url          = {https://link.aps.org/doi/10.1103/PhysRevA.39.5887},
	issue        = 11,
	numpages     = {0}
}

@book{ProvatasElder2010,
	title        = {Phase-Field Methods in Materials Science and Engineering},
	author       = {N. Provatas and K. Elder},
	year         = 2010,
	publisher    = {Wiley-VCH},
	address      = {Weinheim},
	doi          = {10.1002/9783527631520}
}

@article{Steinbach2009,
	title        = {Phase-Field Models in Materials Science},
	author       = {I. Steinbach},
	year         = 2009,
	journal      = {Model. Simul. Mater. Sci. Eng.},
	volume       = 17,
	number       = 7,
	pages        = {073001},
	doi          = {10.1088/0965-0393/17/7/073001}
}

@article{KarmaRappel1998,
	title        = {Quantitative Phase-Field Modeling of Dendritic Growth in Two and Three Dimensions},
	author       = {A. Karma and W.-J. Rappel},
	year         = 1998,
	journal      = {Phys. Rev. E},
	volume       = 57,
	number       = 4,
	pages        = {4323--4349},
	doi          = {10.1103/PhysRevE.57.4323}
}

@article{BeckermannKarma1997,
	title        = {Modeling Melt Convection in Phase-Field Simulations of Solidification},
	author       = {C Beckermann and H.-J Diepers and I Steinbach and A Karma and X Tong},
	year         = 1999,
	journal      = {JJ. Comput. Phys.},
	volume       = 154,
	number       = 2,
	pages        = {468--496},
	doi          = {https://doi.org/10.1006/jcph.1999.6323},
	url          = {https://www.sciencedirect.com/science/article/pii/S0021999199963234}
}

@book{GurtinFriedAnand2010,
	title        = {The Mechanics and Thermodynamics of Continua},
	author       = {M. E. Gurtin and E. Fried and L. Anand},
	year         = 2010,
	publisher    = {Cambridge University Press},
	address      = {Cambridge},
	doi          = {10.1017/CBO9780511762956}
}

@book{KurzFisher1998,
	title        = {Fundamentals of Solidification},
	author       = {W. Kurz and D. J. Fisher},
	year         = 1998,
	publisher    = {Trans Tech Publications},
	address      = {Aedermannsdorf, Switzerland},
	edition      = {4th ed.}
}

@article{Anderson1998,
	title        = {Diffuse-Interface Methods in Fluid Mechanics},
	author       = {D. M. Anderson and G. B. McFadden and A. A. Wheeler},
	year         = 1998,
	journal      = {Annu. Rev. Fluid Mech.},
	volume       = 30,
	pages        = {139--165},
	doi          = {10.1146/annurev.fluid.30.1.139}
}

@article{King2015,
	title        = {Observation of keyhole-mode laser melting in laser powder-bed fusion additive manufacturing},
	author       = {Wayne E. King and Holly D. Barth and Victor M. Castillo and Gilbert F. Gallegos and John W. Gibbs and Douglas E. Hahn and Chandrika Kamath and Alexander M. Rubenchik},
	year         = 2014,
	journal      = {J. Mater. Process. Technol.},
	volume       = 214,
	number       = 12,
	pages        = {2915--2925},
	doi          = {https://doi.org/10.1016/j.jmatprotec.2014.06.005},
	url          = {https://www.sciencedirect.com/science/article/pii/S0924013614002283}
}

@article{Khairallah2016,
	title        = {Laser Powder-Bed Fusion Additive Manufacturing: Physics of Melting, Vapor Depression, and Spatter},
	author       = {S. A. Khairallah and A. Anderson and J. Rubenchik and W. E. King},
	year         = 2016,
	journal      = {Acta Mater.},
	volume       = 108,
	pages        = {36--45},
	doi          = {10.1016/j.actamat.2016.02.014}
}

@book{Eringen1999,
	title        = {Microcontinuum Field Theories I: Foundations and Solids},
	author       = {A. C. Eringen},
	year         = 1999,
	publisher    = {Springer},
	address      = {New York},
	doi          = {10.1007/978-1-4612-0555-5}
}

@article{Forest2009,
	title        = {Micromorphic Approach for Gradient Elasticity, Viscoplasticity, and Damage},
	author       = {S. Forest},
	year         = 2009,
	journal      = {J. Eng. Mech.},
	volume       = 135,
	number       = 3,
	pages        = {117--131},
	doi          = {10.1061/(ASCE)0733-9399(2009)135:3(117)}
}

@article{Elliott1996,
	title        = {The {C}ahn--{H}illiard model for the kinetics of phase separation},
	author       = {Charles M. Elliott},
	year         = 1996,
	journal      = {Math. Surf.},
	pages        = {35--44}
}

@article{NovickCohenPego1991,
	title        = {Stable Patterns in a Viscous Diffusion Equation},
	author       = {A. Novick-Cohen and R. L. Pego},
	year         = 1991,
	journal      = {Trans. Amer. Math. Soc.},
	publisher    = {American Mathematical Society},
	volume       = 324,
	number       = 1,
	pages        = {331--351},
	url          = {http://www.jstor.org/stable/2001511},
	urldate      = {2026-02-13}
}

@article{ColliGilardi1999,
	title        = {On A Class Of Doubly Nonlinear Evolution Equations},
	author       = {P. Colli and A. Visintin},
	year         = 1990,
	journal      = {Commun. Partial Differ. Equ.},
	publisher    = {Taylor \& Francis},
	volume       = 15,
	number       = 5,
	pages        = {737--756},
	doi          = {10.1080/03605309908820706},
	url          = {https://doi.org/10.1080/03605309908820706},
	eprint       = {https://doi.org/10.1080/03605309908820706}
}

@article{Eyre1998,
	title        = {Unconditionally Gradient Stable Time Marching the {C}ahn-{H}illiard Equation},
	author       = {Eyre,  David J.},
	year         = 1998,
	journal      = {MRS Proc.},
	publisher    = {Springer Science and Business Media LLC},
	volume       = 529,
	doi          = {10.1557/proc-529-39},
	url          = {http://dx.doi.org/10.1557/PROC-529-39}
}

@article{ShenYang2010,
	title        = {Numerical approximations of  {A}llen-{C}ahn and {C}ahn-{H}illiard equations},
	author       = {Shen,  Jie and Yang,  Xiaofeng},
	year         = 2010,
	journal      = {Discrete Contin. Dyn. Syst. A},
	publisher    = {American Institute of Mathematical Sciences (AIMS)},
	volume       = 28,
	number       = 4,
	pages        = {1669–1691},
	doi          = {10.3934/dcds.2010.28.1669},
	url          = {http://dx.doi.org/10.3934/dcds.2010.28.1669}
}

@article{GrunKlingbeil2016,
	title        = {Two-phase flow with mass density contrast: Stable schemes for a thermodynamic consistent and frame-indifferent diffuse-interface model},
	author       = {G. Gr\"{u}n and F. Klingbeil},
	year         = 2014,
	journal      = {J. Comput. Phys.},
	volume       = 257,
	pages        = {708--725},
	doi          = {https://doi.org/10.1016/j.jcp.2013.10.028},
	url          = {https://www.sciencedirect.com/science/article/pii/S0021999113007043}
}

@article{COLLI2024113461,
	title        = {{On a {C}ahn–{H}illiard system with source term and thermal memory}},
	author       = {Pierluigi Colli and Gianni Gilardi and Andrea Signori and J\"{u}rgen Sprekels},
	year         = 2024,
	journal      = {Nonlinear Anal.},
	volume       = 240,
	pages        = 113461,
	doi          = {https://doi.org/10.1016/j.na.2023.113461},
	url          = {https://www.sciencedirect.com/science/article/pii/S0362546X23002535}
}

@article{Eleuteri2015,
	title        = {On a non-isothermal diffuse interface model for two-phase flows of incompressible fluids},
	author       = {Eleuteri,  Michela and Rocca,  Elisabetta and Schimperna,  Giulio},
	year         = 2015,
	journal      = {Discrete Contin. Dyn. Syst. A},
	publisher    = {American Institute of Mathematical Sciences (AIMS)},
	volume       = 35,
	number       = 6,
	pages        = {2497–2522},
	doi          = {10.3934/dcds.2015.35.2497},
	url          = {http://dx.doi.org/10.3934/dcds.2015.35.2497}
}

@article{Colli_memory,
	title        = {{A phase field model with thermal memory governed by the entropy balance}},
	author       = {Bonetti, Elena and Colli, Pierluigi and Fremond, Michel},
	year         = 2003,
	journal      = {Math. Models Methods Appl. Sci.},
	volume       = 13,
	number       = 11,
	pages        = {1565--1588},
	doi          = {10.1142/S0218202503003033},
	url          = {https://doi.org/10.1142/S0218202503003033},
	eprint       = {https://doi.org/10.1142/S0218202503003033}
}

@article{Colli_Penrose,
	title        = {{On a Penrose-Fife phase-field model with nonhomogeneous Neumann boundary conditions for the temperature}},
	author       = {Pierluigi Colli and Gianni Gilardi and Elisabetta Rocca and Giulio Schimperna},
	year         = 2004,
	journal      = {Differ. Integral Equ.},
	publisher    = {Khayyam Publishing, Inc.},
	volume       = 17,
	number       = {5-6},
	pages        = {511 -- 534},
	doi          = {10.57262/die/1356060345},
	url          = {https://doi.org/10.57262/die/1356060345}
}

@article{KENMOCHI19941163,
	title        = {{Evolution systems of nonlinear variational inequalities arising from phase change problems}},
	author       = {Nobuyuki Kenmochi and Marek Niezg\'{o}dka},
	year         = 1994,
	journal      = {Nonlinear Anal. Theory Methods Appl.},
	volume       = 22,
	number       = 9,
	pages        = {1163--1180},
	doi          = {https://doi.org/10.1016/0362-546X(94)90235-6},
	url          = {https://www.sciencedirect.com/science/article/pii/0362546X94902356}
}

@article{GonzalezFerreiro2014,
	title        = {{A thermodynamically consistent numerical method for a phase field model of solidification}},
	author       = {B. Gonzalez-Ferreiro and H. Gomez and I. Romero},
	year         = 2014,
	journal      = {Commun. Nonlinear Sci. Numer. Simul.},
	publisher    = {Elsevier {BV}},
	volume       = 19,
	number       = 7,
	pages        = {2309--2323},
	doi          = {10.1016/j.cnsns.2013.11.016},
	url          = {https://doi.org/10.1016/j.cnsns.2013.11.016}
}

@article{Pawlow2016,
	title        = {{A thermodynamic approach to nonisothermal phase-field models}},
	author       = {Irena Paw{\l}ow},
	year         = 2016,
	journal      = {Appl. Math.},
	publisher    = {Institute of Mathematics,  Polish Academy of Sciences},
	pages        = {1--63},
	doi          = {10.4064/am2282-12-2015},
	url          = {https://doi.org/10.4064/am2282-12-2015}
}

@article{Lasarzik2021,
	title        = {Analysis of a thermodynamically consistent {N}avier–{S}tokes–{C}ahn–{H}illiard model},
	author       = {Lasarzik,  Robert},
	year         = 2021,
	journal      = {Nonlinear Anal.},
	publisher    = {Elsevier BV},
	volume       = 213,
	pages        = 112526,
	doi          = {10.1016/j.na.2021.112526},
	url          = {http://dx.doi.org/10.1016/j.na.2021.112526}
}

@article{Heida2011,
	title        = {{On the development and generalizations of {C}ahn–{H}illiard equations within a thermodynamic framework}},
	author       = {Heida,  Martin and M\'{a}lek,  Josef and Rajagopal,  K. R.},
	year         = 2011,
	journal      = {ZAMP},
	publisher    = {Springer Science and Business Media LLC},
	volume       = 63,
	number       = 1,
	pages        = {145–169},
	doi          = {10.1007/s00033-011-0139-y},
	url          = {http://dx.doi.org/10.1007/s00033-011-0139-y}
}

@article{Alessia2014,
	title        = {{Phase separation in quasi-incompressible fluids: Cahn–Hilliard model in the Cattaneo–Maxwell framework}},
	author       = {Alessia,  Berti and Bochicchio,  Ivana and Fabrizio,  Mauro},
	year         = 2014,
	journal      = {ZAMP},
	publisher    = {Springer Science and Business Media LLC},
	volume       = 66,
	number       = 1,
	pages        = {135–147},
	doi          = {10.1007/s00033-013-0395-0},
	url          = {http://dx.doi.org/10.1007/s00033-013-0395-0}
}

@article{Freistuhler2016,
	title        = {Phase-Field and {K}orteweg-Type Models for the Time-Dependent Flow of Compressible Two-Phase Fluids},
	author       = {Freist\"{u}hler,  Heinrich and Kotschote,  Matthias},
	year         = 2016,
	journal      = {Arch. Ration. Mech. Anal.},
	publisher    = {Springer Science and Business Media LLC},
	volume       = 224,
	number       = 1,
	pages        = {1–20},
	doi          = {10.1007/s00205-016-1065-0},
	url          = {http://dx.doi.org/10.1007/s00205-016-1065-0}
}

@article{Lopes2022,
	title        = {{Existence of solutions for a non-isothermal {N}avier-{S}tokes-{A}llen-{C}ahn system with thermo-induced coefficients}},
	author       = {Lopes,  Juliana Honda and Planas,  Gabriela},
	year         = 2022,
	journal      = {Electron. J. Differ. Equations},
	publisher    = {Texas State University},
	volume       = 2022,
	number       = {01–87},
	pages        = 72,
	doi          = {10.58997/ejde.2022.72},
	url          = {http://dx.doi.org/10.58997/ejde.2022.72}
}

@article{Sun2020,
	title        = {Structure-Preserving Numerical Approximations to a Non-isothermal Hydrodynamic Model of Binary Fluid Flows},
	author       = {Shouwen Sun and Jun Li and Jia Zhao and Qi Wang},
	year         = 2020,
	journal      = {J. Sci. Comput.},
	publisher    = {Springer Science and Business Media {LLC}},
	volume       = 83,
	number       = 3,
	doi          = {10.1007/s10915-020-01229-6},
	url          = {https://doi.org/10.1007/s10915-020-01229-6}
}

@article{BrunkPamm,
	title        = {{Nonisothermal {C}ahn–{H}illiard {N}avier–{S}tokes system}},
	author       = {Brunk, Aaron and Schumann, Dennis},
	year         = 2024,
	journal      = {PAMM},
	volume       = 24,
	number       = 2,
	pages        = {e202400060},
	doi          = {https://doi.org/10.1002/pamm.202400060}
}

@article{BrunkCMAM,
	title        = {Variational Approximation for a Non-Isothermal Coupled Phase-Field System: Structure-Preservation \& {N}onlinear Stability},
	author       = {Aaron Brunk and Oliver Habrich and Timileyin David Oyedeji and Yangyiwei Yang and Bai-Xiang Xu},
	year         = 2025,
	journal      = {Comput. Methods Appl. Math.},
	volume       = 25,
	number       = 2,
	pages        = {373--396},
	doi          = {doi:10.1515/cmam-2023-0274},
	url          = {https://doi.org/10.1515/cmam-2023-0274},
	lastchecked  = {2025-01-09}
}

@article{SUN2024161,
	title        = {A thermodynamically consistent phase-field model and an entropy stable numerical method for simulating two-phase flows with thermocapillary effects},
	author       = {Yanxiao Sun and Jiang Wu and Maosheng Jiang and Steven M. Wise and Zhenlin Guo},
	year         = 2024,
	journal      = {Appl. Numer. Math.},
	volume       = 206,
	pages        = {161--189},
	doi          = {https://doi.org/10.1016/j.apnum.2024.08.010},
	url          = {https://www.sciencedirect.com/science/article/pii/S016892742400206X}
}

@inbook{Brunk2025,
	title        = {Structure-Preserving Approximation for the Non-isothermal Cahn-Hilliard-Navier-Stokes System},
	author       = {Brunk,  Aaron and Schumann,  Dennis},
	year         = 2025,
	booktitle    = {Numerical Mathematics and Advanced Applications ENUMATH 2023,  Volume 1},
	publisher    = {Springer Nature Switzerland},
	pages        = {188–197},
	doi          = {10.1007/978-3-031-86173-4\_19},
	isbn         = 9783031861734
}

@article{RuedaGmez2024,
	title        = {Numerical Analysis for a Non-isothermal Incompressible {N}avier–{S}tokes–{A}llen–{C}ahn System},
	author       = {Rueda-G\'{o}mez,  Diego A. and Rueda-Fern\'{a}ndez,  Elian E. and Villamizar-Roa,  \'{E}lder J.},
	year         = 2024,
	journal      = {J. Math. Fluid Mech.},
	publisher    = {Springer Science and Business Media LLC},
	volume       = 26,
	number       = 4,
	doi          = {10.1007/s00021-024-00898-9},
	url          = {http://dx.doi.org/10.1007/s00021-024-00898-9}
}

@article{ANDERSON2000175,
	title        = {A phase-field model of solidification with convection},
	author       = {D.M. Anderson and G.B. McFadden and A.A. Wheeler},
	year         = 2000,
	journal      = {Physica D},
	volume       = 135,
	number       = 1,
	pages        = {175--194},
	doi          = {https://doi.org/10.1016/S0167-2789(99)00109-8},
	url          = {https://www.sciencedirect.com/science/article/pii/S0167278999001098}
}

@article{WU_2017,
	title        = {Well-posedness of a diffuse-interface model for two-phase incompressible flows with thermo-induced Marangoni effect},
	author       = {Wu, Hao},
	year         = 2017,
	journal      = {Eur. J. Appl. Math.},
	volume       = 28,
	number       = 3,
	pages        = {380–434},
	doi          = {10.1017/S0956792516000322}
}

@article{Liang,
	title        = {Non-isothermal phase-field model of molten pool dynamics coupled with thermocapillary effect and phase change in metal additive manufacturing},
	author       = {Liang, Chenguang and Shen, Zhihai and Yi, Min},
	year         = 2026,
	journal      = {Phys. Fluids},
	volume       = 38,
	number       = 1,
	pages        = {013309},
	doi          = {10.1063/5.0310906},
	url          = {https://doi.org/10.1063/5.0310906},
}

@article{Hossain2024,
	title        = {Global weak solutions to a phase-field model for seawater solidification with melt convection},
	author       = {Hossain,  Md Akram and Ma,  Li},
	year         = 2024,
	journal      = {Part. Differ. Equ. Appl.},
	publisher    = {Springer Science and Business Media LLC},
	volume       = 5,
	number       = 4,
	doi          = {10.1007/s42985-024-00290-2},
	url          = {http://dx.doi.org/10.1007/s42985-024-00290-2}
}

\end{document}